\documentclass[numbers=enddot,12pt,final,onecolumn,notitlepage]{scrartcl}%
\usepackage[headsepline,footsepline,manualmark]{scrlayer-scrpage}
\usepackage[all,cmtip]{xy}
\usepackage{amssymb}
\usepackage{amsmath}
\usepackage{amsthm}
\usepackage{framed}
\usepackage{comment}
\usepackage{color}
\usepackage[sc]{mathpazo}
\usepackage[T1]{fontenc}
\usepackage{needspace}
\usepackage{tabls}
\usepackage[breaklinks=true]{hyperref}
\usepackage{tikz}
\usepackage{diagbox}
\usepackage{enumitem}

\usetikzlibrary{arrows.meta}
\usetikzlibrary{decorations.markings}
\tikzset{->-/.style={decoration={
  markings,
  mark=at position .5 with {\arrow{>}}},postaction={decorate}}}

\newcommand{\red}[1]{{\color{red} #1}}

\theoremstyle{definition}
\newtheorem{theo}{Theorem}[section]
\newenvironment{theorem}[1][]
{\begin{theo}[#1]\begin{leftbar}}
{\end{leftbar}\end{theo}}
\newtheorem{lem}[theo]{Lemma}
\newenvironment{lemma}[1][]
{\begin{lem}[#1]\begin{leftbar}}
{\end{leftbar}\end{lem}}
\newtheorem{prop}[theo]{Proposition}
\newenvironment{proposition}[1][]
{\begin{prop}[#1]\begin{leftbar}}
{\end{leftbar}\end{prop}}
\newtheorem{defi}[theo]{Definition}

\newtheorem{remk}[theo]{Remark}
\newenvironment{remark}[1][]
{\begin{remk}[#1]\begin{leftbar}}
{\end{leftbar}\end{remk}}
\newtheorem{coro}[theo]{Corollary}
\newenvironment{corollary}[1][]
{\begin{coro}[#1]\begin{leftbar}}
{\end{leftbar}\end{coro}}
\newtheorem{conv}[theo]{Convention}

\newtheorem{quest}[theo]{Question}
\newenvironment{question}[1][]
{\begin{quest}[#1]\begin{leftbar}}
{\end{leftbar}\end{quest}}
\newtheorem{warn}[theo]{Warning}

\newtheorem{conj}[theo]{Conjecture}

\newtheorem{exam}[theo]{Example}
\newenvironment{example}[1][]
{\begin{exam}[#1]\begin{leftbar}}
{\end{leftbar}\end{exam}}

\newcommand{\id}{\operatorname{id}}
\newcommand{\W}{\mathcal{W}}
\newcommand{\M}{\mathcal{M}}
\newcommand{\kk}{\mathbf{k}}
\newcommand{\CC}{\mathbb{C}}
\newcommand{\NN}{\mathbb{N}}
\newcommand{\ZZ}{\mathbb{Z}}
\newcommand{\RR}{\mathbb{R}}
\newcommand{\QQ}{\mathbb{Q}}
\newcommand{\pp}{\mathbf{p}}
\newcommand{\qq}{\mathbf{q}}
\newcommand{\rr}{\mathbf{r}}
\newcommand{\SE}{\operatorname{SE}}
\newcommand{\NE}{\operatorname{NE}}
\newcommand{\w}{\operatorname{w}}
\newcommand{\htt}{\operatorname{ht}}
\newcommand{\charr}{\operatorname{char}}
\newcommand{\multiset}{\operatorname{multiset}}
\newcommand{\balsim}{\overset{\operatorname*{bal}}{\sim}}
\newcommand{\irrsim}{\overset{\operatorname*{irr}}{\sim}}
\newcommand{\flipsim}{\overset{\operatorname*{flip}}{\sim}}
\newcommand{\set}[1]{\left\{ #1 \right\}}

\newcommand{\tup}[1]{\left( #1 \right)}
\newcommand{\ive}[1]{\left[ #1 \right]}
\newcommand{\floor}[1]{\left\lfloor #1 \right\rfloor}
\newenvironment{verlong}{}{}
\newenvironment{vershort}{}{}
\newenvironment{noncompile}{}{}
\excludecomment{verlong}
\includecomment{vershort}
\excludecomment{noncompile}
\ihead{Monomial identities in the Weyl algebra}
\ohead{page \thepage}
\begin{document}

\title{Monomial identities in the Weyl algebra}
\author{Darij Grinberg\footnote{Drexel University, Philadelphia, PA, USA // \url{mailto:darijgrinberg@gmail.com}}, Tom Roby\footnote{University of Connecticut, Storrs, CA, USA // \url{mailto:tom.roby@uconn.edu}}, Stephan Wagner\footnote{TU Graz, Austria, and Uppsala University, Sweden // \url{mailto:stephan.wagner@tugraz.at} \\ Supported by the Swedish research council (VR), grant 2022-04030.}, Mei Yin\footnote{University of Denver, Denver, CO, USA // \url{mailto:mei.yin@du.edu} \\ Supported by the University of Denver's Professional Research Opportunities for Faculty Fund 80369-145601 and Simons Foundation Grant MPS-TSM-00007227.}}
\date{8 November 2024}
\dedication{\textit{Dedicated to a Hip Trendy Rascal on his 80th birthday}}
\maketitle

\begin{abstract}
\textbf{Abstract.}
Motivated by a question and some enumerative conjectures of Richard Stanley, we explore the
equivalence classes of words in the Weyl algebra, $\mathbf{k} \left< D,U \mid DU - UD = 1 \right>$.  We show that each class is generated by the swapping of adjacent \textit{balanced subwords}, i.e., those
which have the same number of $D$'s as $U$'s, and give several other characterizations, as well as a linear-time algorithm for equivalence checking.

Armed with this, we deduce several enumerative results about such equivalence classes and their sizes.
We extend these results to the class of $c$-Dyck words, where every prefix has at least $c$ times as many $U$'s as $D$'s.  
We also connect these results to previous work on bond
percolation and rook theory, and generalize them to some other algebras. 
\medskip

\medskip

\end{abstract}

\tableofcontents

\section{Introduction}\label{sec:intro}

In 2023 Richard Stanley proposed the following problem (private communication).  Let 
$\mathcal{M}$ denote the monoid freely generated by the two non-commuting variables $D$ and $U$.
Consider the action of $\M$ on the polynomial ring $\QQ\ive{x}$ in which $D$ acts as differentiation ($\frac{d}{dx}$) and $U$ acts as multiplication by $x$.
This action is not free, as it is known to satisfy the relation
\begin{equation}\label{eq:Weyl}
DU-UD = 1 ,
\end{equation}
which is the defining relation of the \emph{Weyl algebra}
(also known as the \emph{Heisenberg--Weyl algebra} \cite{GoF09} due to its obvious quantum-physical significance).

Consider two words in $\M$ to be \emph{equivalent} if they act equally on $\QQ\ive{x}$ (that is, become equal in the Weyl algebra).
It can be shown that equivalent words have the same number of $U$'s and the same number of $D$'s.
Thus we can ask: \textit{How many distinct equivalence classes
are there} for words with $k$ many $D$'s and $n-k$ many $U$'s?
We call this number $a(n,k)$.  For example, among the six words with two $D$'s
and two $U$'s, there is one equivalence: $DUUD = UDDU$, so that $a(4,2)=5$.

We prove explicit formulas for $a(n,k)$ and for $\sum_k a(n,k)$ originally conjectured by Stanley (Section~\ref{sec:enum}).
Along the way, we study the equivalence from several directions and give several equivalent descriptions of it.
Call a word in $\M$ \textit{balanced} if it has the same number of $U$'s
as $D$'s.
We show (in Theorem~\ref{thm.DU-eqs2}, another conjecture of Stanley) that two words $v$ and $w$ are equivalent if and only if one can be obtained from
another by a series of \textit{balanced commutations}, i.e., by a sequence of swaps of
adjacent balanced factors.  In the example above, the transposition of $DU$ with $UD$ gives
the equivalence.
This and several further criteria serve as the linchpin for the enumerative results. 

In much earlier work, Stanley identified a particular class of posets, including Young's
lattice of integer partitions, which he called \emph{differential
posets}~\cite{Stan88DiffPos}.  Standard and semi-standard Young tableaux are in bijection
with certain chains in Young's lattice, $\mathbb{Y}$, which can be studied using the
standard Down and Up operators in $\mathbb{Y}$. For example, applying $U^{n}$ to the empty
shape $\varnothing$ gives a formal sum of all partitions $\lambda$ of $n$, each weighted by
the number of standard tableaux of that shape $f^{\lambda}$.  Applying $D^{n}U^{n}$ to
$\varnothing$ yields $\left(\sum_{\lambda \vdash n} (f^{\lambda})^{2}\right) \varnothing$, and a
simple inductive argument shows that $D^nU^n\varnothing = n!\varnothing$, recovering the basic
enumerative identity shown by the Robinson--Schensted correspondence. The key insight is
that these operators satisfy the fundamental relation of the Weyl algebra,
Equation~(\ref{eq:Weyl}), allowing counting problems to be expressed in terms of these operators.
Many enumerative identities, expressed as generating functions in these operators, can be
proved by solving certain elementary partial differential equations. This study motivated
the current problem, though it is natural enough on its own.  

The outline of the paper is as follows.
In Section~\ref{sec:defs} we formally define the monoid $\M$, the above monoid, the Weyl algebra $\W$, and many further related combinatorial objects.
Then we state our main result (Theorem~\ref{thm.DU-eqs2}), which gives several different criteria for words
in the monoid $\M$ to be equivalent (i.e., to represent the same operator in $\W$).
\footnote{We note that one of our criteria in Theorem~\ref{thm.DU-eqs2} gives rise to an efficient (linear time) algorithm for the word problem in the monoid generated by $D$ and $U$ in $\W$, in contrast to the naive ``expand and compare coefficients'' approach (which requires quadratic time at best). See Remark~\ref{rmk.algorithm} for details.}
A second main result (Theorem~\ref{thm.phiomega}) says that each balanced word $u$ is equivalent to its ``reverse toggle-image'' (i.e., to the word obtained from $u$ by reversing the order of the letters and also replacing each $D$ by $U$ and vice versa).
The proofs of these two results occupy the next three sections.

Section~\ref{sec:basicform} provides some basic formulas for products of $D$'s and $U$'s in the Weyl algebra.  In Section~\ref{sec:words} we define a normal form for
our words and start working our way towards the proof of our main result, which we finish
in Section~\ref{sec:mainproofs}. 

Enumerative results -- including the formulas for $a(n,k)$ and for $\sum_k a(n,k)$ -- are then obtained in Section~\ref{sec:enum}.
The formula for $\sum_k a(n,k)$ (Corollary~\ref{cor:total_number}) is surprisingly intricate, despite involving nothing more complicated than the Fibonacci sequence.
Then we turn our attention to \textit{$c$-Dyck words}, where every prefix has at least $c$ times as many $U$'s as $D$'s,
finding analogous results for this situation, and exploring some interesting special cases.
Finally, we give an explicit formula for the \textit{size} of the Weyl-equivalence class of
a word $w$.

An intriguing digression is pursued in Section~\ref{sec:bondperc}. Indeed, a search for the sequence of the numbers $\sum_k a(n,k)$ in the OEIS reveals a sequence  \cite[A006727]{OEIS} originating in statistical physics (bond percolation on the directed square lattice).
This sequence, however, agrees with ours only for $n \leq 11$, and in fact contains negative terms later on.
We briefly introduce the physical context and explain this seeming coincidence.

Section~\ref{sec:rook} continues the study of equivalence of words and relates it to the part of combinatorics known as rook theory.
The equivalence of two words $u$ and $v$ is revealed to be a  stronger version of the \emph{rook equivalence} of two Ferrers boards $B_u$ and $B_v$ induced by these words.
This leads to two further ``main theorems'' (Theorem~\ref{thm.DU-eqs3} and~\ref{thm.DU-eqs4}) that provide further equivalent conditions for two words to be equivalent.
Their proofs piggyback on work by Navon, Haglund, Cotardo, Gruica and Ravagnani.
We note that our equivalent criteria in Theorem~\ref{thm.DU-eqs2} can thus also be seen as equivalent criteria for rook-equivalence of Ferrers boards, although some care is needed to ensure that the correspondence really is one-to-one (see Remark~\ref{rmk.rook.as-crit}).

In Section~\ref{sec:irredbal}, we take a closer look at our balanced commutations, and show that a subset of these commutations actually suffices to connect any two equivalent words.
It is not hard to see that we can generate Weyl equivalence by transpositions
only of \textit{irreducible} balanced words, i.e., those which themselves cannot be factored into two or
more balanced words.
Even better, we show that we only need to swap irreducible balanced words starting with a $U$ with ones starting with a $D$ (Theorem~\ref{thm.irr=bal}).

In the final Section~\ref{sec:otheralg}, we discuss other algebras that allow for the same or similar questions to be asked instead of the Weyl algebra.
We generalize our results to multivariate Weyl algebras, and give a partial result for the down-up algebras of Benkart and Roby~\cite{BenRob98}.
The case of Weyl algebras in positive characteristic appears to be more intricate, and we offer several open questions for exploration.

\bigskip

\textbf{Remark.} A more detailed version of this work, with some proofs expanded, is available as an ancillary file to this preprint on the arXiv.

\section{Definitions and main results}\label{sec:defs}

In this section, we introduce the main notions and notations involved in our main results, which will be stated at the end of the section.

\subsection{\label{subsec:defs.action}The monoid $\M$ and the Weyl algebra $\W$}

Let $\M$ be the free (noncommutative) monoid generated by two symbols
$D$ and $U$. Its elements are the words with letters $D$ and $U$, such as
$DUUDDUDD$.

Let $\kk$ be a field of characteristic $0$, and let $\W$ be the Weyl algebra over $\kk$ with generators $D$
and $U$ and relation $DU-UD=1$. This algebra $\W$ acts on the
univariate polynomial ring $\kk\left[  x\right]  $ in a standard way:
$D$ acts as the derivative operator $\dfrac{\partial}{\partial x}$, whereas $U$ acts as multiplication
by $x$. It is known that this action is faithful, and $\W$ has a
basis\footnote{Here and in the following, $\NN$ denotes the set $\set{0,1,2,\ldots}$.}
$\left(  D^{i}U^{j}\right)  _{i,j\in\NN}$ as well as a basis
$\left(  U^{j}D^{i}\right)  _{i,j\in\NN}$.
See \cite{ManSch16}, \cite{vanOys13} and several exercises in \cite{Lorenz18} for much
there is to know about $\W$ and then some.
(The Weyl algebra $\W$ is often denoted by $A_1\tup{\kk}$ or $\mathbb{A}_1\tup{\kk}$.)

Let $\phi : \M \to \W$ be the canonical monoid
morphism\footnote{``Morphism'' means ``homomorphism'' throughout this work.} from $\M$ to the monoid $\left(  \W,\cdot,1\right)  $ that
sends $D$ to $D$ and $U$ to $U$.
Thus, $\phi$ sends any product of $D$'s and $U$'s to the same product of $D$'s and $U$'s, but now computed in $\W$ instead of $\M$.
This morphism $\phi$ is \textbf{not}
injective, since (for example) $DUUD=UDDU$ in $\W$ (but not in
$\M$).
Thus, one may naturally wonder what pairs of words $u, v \in \M$ have equal images under $\phi$.
In the parlance of monoid theory, this is asking about the \emph{kernel} of the monoid morphism $\phi$ -- that is, the equivalence relation ``$\phi\left(u\right) = \phi\left(v\right)$'' on $\M$.

We will give several descriptions of this equivalence relation in terms of different objects, and subsequently study its enumerative properties (such as the number of equivalence classes of a given word length).
We now define some of these objects.

\subsection{Words}

\begin{itemize}
\item The word ``word'' will always mean an element of $\M$ (that is, a word built of $D$'s and $U$'s), unless we say otherwise.

\item A word $v\in\M$ is said to be a \emph{factor} of a word
$w\in\M$ if there exist words $u,u^{\prime}\in\M$ (possibly
empty) such that $w=uvu^{\prime}$.

\item A word $v\in\M$ is said to be a \emph{prefix} of a word
$w\in\M$ if there exists a word $u^{\prime}\in\M$ (possibly
empty) such that $w=vu^{\prime}$.

\item A word $v\in\M$ is said to be a \emph{suffix} of a word
$w\in\M$ if there exists a word
$u\in\M$ (possibly
empty) such that $w=uv$.
\end{itemize}

For example, the word $DUD$ is a prefix of $DUDUDD$ and is a suffix of $UDDUD$.
Furthermore, the word $DUD$ is a factor of $UDUDDD$, but neither a prefix nor a suffix.
The word $DUD$ is not a factor of $DUUD$ (since a factor must appear contiguously).
Note that each word $w$ is a factor, a prefix and a suffix of itself.

\subsection{Diagonal paths}

A notion closely related to words are \emph{diagonal paths}, which we will now introduce along with their various features:

\begin{itemize}

\item The \emph{diagonal lattice} means the digraph (i.e., directed graph) with vertex set
$\ZZ^2$ and arcs $\left(  i,j\right)  \to \left(
i+1,j+1\right)  $ (called \emph{NE-arcs}) and $\left(  i,j\right)
\to \left(  i+1,j-1\right)  $ (called \emph{SE-arcs}). Given two
vertices $u$ and $v$ of the diagonal lattice, we write \textquotedblleft%
$u\nearrow v$\textquotedblright\ for \textquotedblleft$u \to v$ is an
NE-arc\textquotedblright, and we  write \textquotedblleft$u\searrow
v$\textquotedblright\ for \textquotedblleft$u \to v$ is an
SE-arc\textquotedblright.

We imagine the diagonal lattice as being drawn in the Cartesian plane, but its arcs are not parallel to the axes but rather parallel to the two diagonals ($x=y$ and $x=-y$).
Thus, NE-arcs and SE-arcs look like the arrows $\nearrow$ and $\searrow$, respectively (whence our notations for them).

\item
A \emph{diagonal path} means a walk on the diagonal lattice.
Since the diagonal lattice is acyclic (i.e., has no directed cycles), any such walk is a path.

\item If $\pp=\left(  p_{0},p_{1},\ldots,p_{k}\right)  $ is a diagonal path, then

\begin{itemize}
\item the \emph{vertices} of $\pp$ are $p_{0},p_{1},\ldots,p_{k}$;

\item the \emph{NE-steps} of $\pp$ are the vertices $p_{i}$ of
$\pp$ for which $i < k$ and $p_{i}\nearrow p_{i+1}$;

\item the \emph{SE-steps} of $\pp$ are the vertices $p_{i}$ of
$\pp$ for which $i < k$ and $p_{i}\searrow p_{i+1}$.
\end{itemize}

For example, if $\pp = \tup{p_0, p_1, p_2, p_3}$ is a diagonal path with $p_0 \nearrow p_1 \searrow p_2 \searrow p_3$, then its vertices are $p_0, p_1, p_2, p_3$; its only NE-step is $p_0$; and its SE-steps are $p_1$ and $p_2$.

\item The \emph{height} $\htt\left(  i,j\right)  $ of a vertex
$\left(  i,j\right)  $ of $\ZZ^2$ is its y-coordinate $j$.

\item If $\pp = \tup{p_0, p_1, \ldots, p_k}$ is a diagonal path, then the \emph{initial height} of $\pp$ is the height $\htt\tup{p_0}$ of its initial vertex, whereas the \emph{final height} of $\pp$ is the height $\htt\tup{p_k}$ of its final vertex.
We say that the path $\pp$ \emph{starts at height $\htt\tup{p_0}$} and \emph{ends at height $\htt\tup{p_k}$}.

\item If $\pp = (p_0,p_1,\ldots,p_k)$ is any diagonal path, then we associate three Laurent polynomials (in the indeterminate $z$) to $\pp$:
    \begin{itemize}
    \item the \emph{height polynomial} $H(\pp,z) = \sum\limits_{i=0}^k z^{\htt(p_i)}$;
    \item the \emph{NE-height polynomial} $H_{\NE}(\pp,z) = \sum\limits_{p_{i}\text{ is an NE-step of }\pp} z^{\htt(p_i)}$;
    \item and the \emph{SE-height polynomial} $H_{\SE}(\pp,z) = \sum\limits_{p_{i}\text{ is an SE-step of }\pp} z^{\htt(p_i)}$.
    \end{itemize}

\item The \emph{reading word} $\w\left(  \pp\right)  $ of a diagonal
path $\pp=\left(  p_{0},p_{1},\ldots,p_{k}\right)  $ is defined to be
the word $w_{0}w_{1}\cdots w_{k-1}\in\M$, where%
\[
w_{i}=%
\begin{cases}
U, & \text{if }p_{i}\nearrow p_{i+1};\\
D, & \text{if }p_{i}\searrow p_{i+1}.
\end{cases}
\]
For instance, if $\pp = \tup{p_0, p_1, p_2, p_3, p_4}$ with $p_0 \nearrow p_1 \searrow p_2 \searrow p_3 \searrow p_4$, then $\w\tup{\pp} = UDDD$.
\end{itemize}

For example, if $\pp$ is the diagonal path shown in Figure~\ref{fig.exa.wp.1}, then
the heights of all vertices of $\pp$ are shown on Figure~\ref{fig.exa.heights}; in particular,
its initial height is $0$, its final height is $-1$, its height polynomial is $H\tup{\pp, z} = z^{-1} + 3z^0 + 3z^1 + z^2$,
its NE-height polynomial is $H_{\NE}\tup{\pp, z} = 2z^0 + z^1$,
its SE-height polynomial is $H_{\SE}\tup{\pp, z} = z^0 + 2z^1 + z^2$,
and its reading word is $\w\tup{\pp} = UDUUDDD$.

\begin{figure}[htbp]
\begin{center}
\begin{tikzpicture}
\draw[help lines] (-0.2, -1.2) grid (7.2, 2.2);
\draw
[line width=2pt,-stealth] (0, 0) edge (1, 1) (1, 1) edge (2, 0) (2, 0) edge (3, 1) (3, 1) edge (4, 2) (4, 2) edge (5, 1) (5, 1) edge (6, 0) -- (7, -1);
\draw[fill=blue] (0, 0) circle (0.1) node[left] {\textcolor{blue}{$\left(0, 0\right)\ $}};
\draw[fill=blue] (7, -1) circle (0.1) node[right] {\textcolor{blue}{$\ \left(7, -1\right)$}};
\end{tikzpicture}
\end{center}
\caption{A diagonal path.}
\label{fig.exa.wp.1}
\end{figure}

\begin{figure}[htbp]
\begin{center}
\begin{tabular}[c]{c}
heights of SE-steps\\
\begin{tikzpicture}
\draw[help lines] (-0.2, -1.2) grid (7.2, 2.2);
\draw
[line width=2pt,-stealth] (0, 0) edge (1, 1) (1, 1) edge (2, 0) (2, 0) edge (3, 1) (3, 1) edge (4, 2) (4, 2) edge (5, 1) (5, 1) edge (6, 0) -- (7, -1);
\draw[fill=blue] (0, 0) circle (0.1) node[left] {\textcolor{blue}{$\left(0, 0\right)\ $}};
\draw[fill=blue] (7, -1) circle (0.1) node[below] {\textcolor{blue}{$\ \left(7, -1\right)$}};
\draw (0,0) -- (0, -3);
\node at (0,-3) [below] {$0$};
\draw (2,0) -- (2, -3);
\node at (2,-3) [below] {$0$};
\draw (3,1) -- (3, -3);
\node at (3,-3) [below] {$1$};
\draw (1,1) -- (1, 4);
\node at (1,4) [above] {$1$};
\draw (4,2) -- (4, 4);
\node at (4,4) [above] {$2$};
\draw (5,1) -- (5, 4);
\node at (5,4) [above] {$1$};
\draw (6,0) -- (6, 4);
\node at (6,4) [above] {$0$};
\draw (7,-1) -- (9,-1);
\node at (9,-1) [right] {$-1$};
\end{tikzpicture}\\
heights of NE-steps
\end{tabular}
\end{center}
\caption{The heights of the vertices of $\pp$. The heights of the NE-steps are written at the bottom, while those of the SE-steps are written at the top. The last vertex of $\pp$ counts neither as an SE-step nor as an NE-step.}
\label{fig.exa.heights}
\end{figure}

We note that if $\pp$ is any diagonal path, then
\begin{align}
    & \tup{\text{final height of } \pp} - \tup{\text{initial height of } \pp} \nonumber \\
    = &
    \tup{\text{\# of $U$'s in } \w\tup{\pp}}
    - \tup{\text{\# of $D$'s in } \w\tup{\pp}} .
    \label{eq.U-D}
\end{align}
\begin{verlong}
This is because, as we walk the path $\pp$ from its beginning to its end, our height increases by $1$ with each NE-step (which corresponds to a $U$ in $\w\tup{\pp}$) and decreases by $1$ with each SE-step (which corresponds to a $D$ in $\w\tup{\pp}$).
\end{verlong}

Note that a diagonal path $\pp$ is uniquely determined by its initial vertex $p_0$ and its reading word $\w\tup{\pp}$.
\begin{verlong}
(Knowing $p_0$ and $\w\tup{\pp}$, we can reconstruct $\pp$ by starting at $p_0$ and walking in the directions provided by $\w\tup{\pp}$: namely, we make an NE-step for each $U$ in $w$ and an SE-step for each $D$ in $w$.)
\end{verlong}
In particular, for any word $w \in \M$, there is a unique diagonal path $\pp$ that starts at $\tup{0, 0}$ and has reading word $\w\tup{\pp} = w$.
We will call this path $\pp$ the \emph{standard path} of $w$.
For example, the path shown in Figure~\ref{fig.exa.wp.1} is the standard path of $UDUUDDD$.

Given a word $w \in \M$, we define the Laurent polynomials
\[
H(w,z) = H(\pp,z)
\quad \text{ and } \quad
H_{\NE}(w,z) = H_{\NE}(\pp,z)
\quad \text{ and } \quad
H_{\SE}(w,z) = H_{\SE}(\pp,z) ,
\]
where $\pp$ is the standard path of $w$.
We call $H(w,z)$ the \emph{height polynomial} of $w$.
Furthermore, we define the \emph{final height} of our word $w$ to be the final height of its standard path $\pp$.
Since its initial height is $0$ (because it starts at $\tup{0,0}$), and since it satisfies $\w\tup{\pp} = w$, we obtain from \eqref{eq.U-D} the equality
\begin{align}
    & \tup{\text{final height of } w} \nonumber \\
    = &
    \tup{\text{\# of $U$'s in } w}
    - \tup{\text{\# of $D$'s in } w} .
    \label{eq.U-D0}
\end{align}

\subsection{The $\omega$ maps}

We furthermore define two useful maps, both of which we call $\omega$:

\begin{itemize}
\item Let $\omega : \M \to \M$ be the monoid
anti-morphism%
\footnote{A \emph{monoid anti-morphism} is a map $f : M \to N$ between two monoids such that $f\tup{1_M} = 1_N$ and $f\tup{ab} = f\tup{b}f\tup{a}$ for all $a, b \in M$. In other words, it is a monoid morphism from $M$ to the opposite monoid of $N$.}
that sends $U$ and $D$ to $D$ and $U$.
Thus, acting on a
word $w\in\M$, it reverses the word and toggles\footnote{To \emph{toggle} a letter means to replace it by a $D$ if it is a $U$, and to replace it by a $U$ if it is a $D$.} every letter.

For example, $\omega\tup{UDD} = UUD$.

\begin{vershort}
This map $\omega$ has a simple geometric meaning:
It simply reflects the standard path of the word across a vertical axis (see Figure~\ref{fig.reflect}).

\begin{figure}[htbp]
\begin{center}
\begin{tikzpicture}[scale=0.5]
\draw
[line width=2pt] (5,1) -- (6,2) -- (7,3) -- (8,4) -- (9,3) -- (10,2) -- (11,3) -- (12,2) -- (13,1) -- (14,0);
\draw[fill=blue] (5, 1) circle (0.1) node[left] {};
\draw[fill=blue] (14, 0) circle (0.1) node[right] {};
\node (p) at (12, 3) {$\pp$};
\end{tikzpicture}
\qquad
\begin{tikzpicture}
\draw (0,0.5mm) -- (0,-0.5mm);
\newlength\mylength
\setlength{\mylength}{\widthof{reflection}}
\draw[->] (0,0) -- (1.2\mylength,0) node[above,midway] {reflection};
\end{tikzpicture}
\qquad
\begin{tikzpicture}[xscale=-1, scale=0.5]
\draw
[line width=2pt] (5,1) -- (6,2) -- (7,3) -- (8,4) -- (9,3) -- (10,2) -- (11,3) -- (12,2) -- (13,1) -- (14,0);
\draw[fill=blue] (5, 1) circle (0.1) node[left] {};
\draw[fill=blue] (14, 0) circle (0.1) node[right] {};
\node (q) at (6, 3) {$\qq$};
\end{tikzpicture}
\end{center}
\caption{A path $\pp$ (left) and its reflection $\qq$ across a vertical axis (right).}
\label{fig.reflect}
\end{figure}
\end{vershort}

\item
Let $\omega : \W \to \W$ be the $\kk$-algebra
anti-morphism%
\footnote{A \emph{$\kk$-algebra anti-morphism} is a map $f : A \to B$ between two $\kk$-algebras such that $f$ is a morphism of additive groups and a monoid anti-morphism of multiplicative monoids. In other words, it is a $\kk$-algebra morphism from $A$ to the opposite algebra of $B$.}
that sends $U$ and $D$ to $D$ and $U$. (This is well-defined,
since the operation of swapping $U$ with $D$ transforms the defining relation $DU-UD=1$ of $\W$ into the relation $UD-DU = 1$, which holds in the opposite algebra of $\W$.)

For example, $\omega\tup{UDD} = UUD$ (but now in $\W$).
\end{itemize}

Both maps $\omega$ are involutions, i.e., satisfy
$\omega \circ \omega = \id$.
Moreover, the two $\omega$'s commute with $\phi$:
That is, we have $\omega\circ\phi =\phi\circ\omega$.
In other words, the diagram
\begin{align}
    \xymatrix{
    \M \ar[r]^{\phi} \ar[d]_{\omega} & \W \ar[d]^{\omega} \\
    \M \ar[r]_{\phi} & \W
    }
    \label{eq.omega.comm}
\end{align}
commutes.
This justifies us calling the two $\omega$'s by the same letter.

\subsection{Balanced words, commutations and flips}

Finally, we define the concept of a \emph{balanced word} and two
equivalence relations on words:

\begin{itemize}
\item A word $w\in\M$ is said to be \emph{balanced} if it has the
same number of $D$'s and $U$'s. For example, $DUUDDU$ is balanced, whereas
$DUUUD$ is not.

\item Given two words $v,w\in\M$, we say that $v$ is obtained from
$w$ by a \emph{balanced commutation} if and only if we can write $v$ and $w$
as $v=pxyq$ and $w=pyxq$, where $p,q\in\M$ are two words and where
$x,y\in\M$ are two balanced words. Roughly speaking, this means that
$v$ can be obtained from $w$ by swapping two balanced factors that abut each
other in $w$.

For instance, from $DUDDUUDUUD$ we can obtain $DDUUDUDUUD$ by a balanced
commutation (swapping the prefix $DU$ with the infix $DDUU$, both of which are
balanced). By a further balanced commutation, we can turn $DDUUDUDUUD$ into
$DDUUDUUDDU$ (swapping the infix $DU$ with the suffix $UD$, both of which are balanced).

We define an equivalence relation $\balsim$ on
$\M$ by stipulating that two words $w,v\in\M$ satisfy
$w \balsim v$ if and only if $v$ can be obtained
from $w$ by a sequence (possibly empty) of balanced commutations. Thus, our above 
examples show that $DUDDUUDUUD \balsim DDUUDUUDDU$.

\item Given two words $v,w\in\M$, we say that $v$ is obtained from
$w$ by a \emph{balanced flip} if and only if we can write $v$ and $w$
as $v=pxq$ and $w=p\omega\tup{x}q$, where $p,q\in\M$ are two words
and where $x \in \M$ is a balanced word. Roughly speaking, this means
that $v$ can be obtained from $w$ by picking a balanced factor and applying
the involution $\omega$ to it.

For instance, from $DUDDUU$ we can obtain $DDUUDU$ by a balanced
flip (applying $\omega$ to the balanced factor $UDDU$).

We define an equivalence relation $\flipsim$ on
$\M$ by stipulating that two words $w,v\in\M$ satisfy
$w \flipsim v$ if and only if $v$ can be obtained
from $w$ by a sequence (possibly empty) of balanced flips. Thus, our above 
example shows that $DUDDUU \flipsim DDUUDU$.

\end{itemize}

Another example for a balanced transformation and a balanced flip can be seen on Figure~\ref{fig.balanced-transformations}.

\begin{figure}[htbp]
\begin{center}
\begin{tabular}[c]{ccc}
\textbf{word} &  & \textbf{standard path}\\\hline
& & \\
$%
\begin{array}[c]{c}%
U\underline{UUUDDUDD}\overline{DDUU}DD\\
\vphantom{\int}\\
\vphantom{\int}
\end{array}
$ &  & $
\begin{tikzpicture}[scale=0.5]
\draw
[fill=blue] (1,1) -- (2,2) -- (3,3) -- (4,4) -- (5,3) -- (6,2) -- (7,3) -- (8,2) -- (9,1) -- (1,1);
\draw[fill=yellow] (9,1) -- (10,0) -- (11,-1) -- (12,0) -- (13,1) -- (9,1);
\draw[help lines] (-0.2, -1.2) grid (15.2, 4.2);
\draw[line width=2pt] (0, 0) -- (1,1);
\draw
[line width=2pt] (1,1) -- (2,2) -- (3,3) -- (4,4) -- (5,3) -- (6,2) -- (7,3) -- (8,2) -- (9,1);
\draw[line width=2pt] (9,1) -- (10,0) -- (11,-1) -- (12,0) -- (13,1);
\draw[line width=2pt] (13,1) -- (14,0) -- (15,-1);
\draw[fill=blue] (0, 0) circle (0.1) node[left] {};
\draw[fill=blue] (15, -1) circle (0.1) node[right] {};
\end{tikzpicture}
$\\
balanced commutation & $\downarrow$ & balanced commutation\\
$%
\begin{array}
[c]{c}%
UDDUU\underline{UUUDDUDD}DD\\
\vphantom{\int}\\
\vphantom{\int}
\end{array}
$ &  & $
\begin{tikzpicture}[scale=0.5]
\draw
[fill=blue] (5,1) -- (6,2) -- (7,3) -- (8,4) -- (9,3) -- (10,2) -- (11,3) -- (12,2) -- (13,1) -- (5,1);
\draw[fill=yellow] (1,1) -- (2,0) -- (3,-1) -- (4,0) -- (5,1) -- (1,1);
\draw[help lines] (-0.2, -1.2) grid (15.2, 4.2);
\draw[line width=2pt] (0, 0) -- (1,1);
\draw
[line width=2pt] (5,1) -- (6,2) -- (7,3) -- (8,4) -- (9,3) -- (10,2) -- (11,3) -- (12,2) -- (13,1);
\draw[line width=2pt] (1,1) -- (2,0) -- (3,-1) -- (4,0) -- (5,1);
\draw[line width=2pt] (13,1) -- (14,0) -- (15,-1);
\draw[fill=blue] (0, 0) circle (0.1) node[left] {};
\draw[fill=blue] (15, -1) circle (0.1) node[right] {};
\end{tikzpicture}
$\\
balanced flip & $\downarrow$ & balanced flip\\
$%
\begin{array}
[c]{c}%
UDDUUUUUDDUDDDD\\
\vphantom{\int}\\
\vphantom{\int}
\end{array}
$ &  & $%
\begin{tikzpicture}[scale=0.5]
\draw
[fill=blue] (5,1) -- (6,2) -- (7,3) -- (8,2) -- (9,3) -- (10,4) -- (11,3) -- (12,2) -- (13,1) -- (5,1);
\draw[help lines] (-0.2, -1.2) grid (15.2, 4.2);
\draw[line width=2pt] (0, 0) -- (1,1);
\draw
[line width=2pt] (5,1) -- (6,2) -- (7,3) -- (8,2) -- (9,3) -- (10,4) -- (11,3) -- (12,2) -- (13,1) -- (13,1);
\draw[line width=2pt] (1,1) -- (2,0) -- (3,-1) -- (4,0) -- (5,1);
\draw[line width=2pt] (13,1) -- (14,0) -- (15,-1);
\draw[fill=blue] (0, 0) circle (0.1) node[left] {};
\draw[fill=blue] (15, -1) circle (0.1) node[right] {};
\end{tikzpicture}
$
\end{tabular}
\end{center}
\caption{A word undergoing a balanced commutation followed by a balanced flip (left); the corresponding standard paths (right). The factors that are swapped or transformed are marked by underlines and overlines.}
\label{fig.balanced-transformations}
\end{figure}

\subsection{Main result: Equivalent descriptions of Weyl equivalence}

Everything is now in place to state our main result, which gives several (necessary and sufficient) criteria for when two words $u, v \in \M$ have the same image under $\phi$.

\begin{theorem}
\label{thm.DU-eqs2}
Let $u$ and $v$ be two words in $\M$.
Then, the following seven statements are equivalent:

\begin{itemize}
\item $\mathcal{S}_1$: We have $\phi\tup{u} = \phi\tup{v}$.

\item $\mathcal{S}_2$: The elements $\phi\tup{u}$ and
$\phi\tup{v}$ act equally on the polynomial ring $\kk\ive{x}$.
(That is, we have
$\tup{\phi\tup{u}} \tup{p} = \tup{\phi\tup{v}} \tup{p}$
for each polynomial $p \in \kk\ive{x}$.
Note that the action of $\W$ on $\kk\left[x\right]$
was defined in Subsection~\ref{subsec:defs.action}.)

\item $\mathcal{S}_3$: The words $u$ and $v$ have the same final
height and satisfy $H_{\NE}\tup{u, z} = H_{\NE}\tup{v, z}$.

\item $\mathcal{S}'_3$: The words $u$ and $v$ have the same final
height and satisfy $H_{\SE}\tup{u, z} = H_{\SE}\tup{v, z}$.

\item $\mathcal{S}_4$: The words $u$ and $v$ have the same final
height and satisfy $H\tup{u, z} = H\tup{v, z}$.

\item $\mathcal{S}_5$: We have $u \balsim v$.

\item $\mathcal{S}_6$: We have $u \flipsim v$.
\end{itemize}
\end{theorem}

In particular, this shows that the two relations $\balsim$ and $\flipsim$ are the same. 

\begin{remark}
\label{rmk.algorithm}
Criterion $\mathcal{S}_3$ (or $\mathcal{S}'_3$) can also be turned into an efficient algorithm to decide whether or not $\phi\tup{u} = \phi\tup{v}$, which requires linear time and space in the length of $u$ and $v$: 

We use an array $A$ with both positive and negative indices. 
Start at $i=0$ and read the word $u$. When a value of $i$ is reached for the first time, $A[i]$ is initialized as $0$. In particular, $A[0]$ is initialized as $0$ right at the beginning. 
Each time the letter $U$ is read, increase $A[i]$ by $1$ and increase $i$ by $1$; else decrease $i$ by $1$ (and leave the values of the array unchanged). The final value of $i$ is the final height of $u$, whereas the final entries of the array are the coefficients of $H_{\NE}\tup{u, z}$. 
Once $u$ has been read completely, do the same with $v$ (starting at $i=0$ again), but decrease $A[i]$ whenever the letter $U$ occurs. Condition $\mathcal{S}_3$ is clearly only satisfied if the final heights of $u$ and $v$ are the same and the array at the end consists entirely of zeros. If a value of $i$ is reached in the second part that was not reached in the first part, or if an entry of the array becomes negative during the second stage of the algorithm, one can stop immediately: $\phi\tup{u} \neq \phi\tup{v}$ in this case.
The length of the array is thus bounded by the length of $u$. Even more, since the sum of the entries of the array is bounded by the length of word $u$, a linear total number of bits is sufficient for all entries of the array.
\end{remark}

In Section~\ref{sec:rook}, we will add some more equivalent statements to the list in Theorem~\ref{thm.DU-eqs2}, albeit under the additional assumption that $u$ and $v$ have the same number of $U$'s and the same number of $D$'s.

The following result is a curious consequence of Theorem~\ref{thm.DU-eqs2} (although we will prove it first and then use it in the proof of Theorem~\ref{thm.DU-eqs2}).

\begin{theorem}
\label{thm.phiomega}
Let $u \in \M$ be a balanced word. Then,
$\phi\tup{u} = \phi\tup{\omega\tup{u}}$ and
$u \balsim \omega\tup{u}$.
\end{theorem}

\section{Basic formulas for the Weyl algebra action}\label{sec:basicform}

The proof of our main result requires significant build-up and preparation.
We begin with a closer look at the Weyl algebra $\W$ and its action on the polynomial ring $\kk\left[  x\right]  $.

\begin{lemma}
\label{lem.weyl.faith}
The action of the Weyl algebra $\W$ on the
polynomial ring $\kk\left[  x\right]  $ is faithful: That is, if two
elements $a, b\in\W$ satisfy $a\left(  p\right)  = b\left(p\right)$ for all
$p\in\kk\left[  x\right]  $, then $a = b$.
\end{lemma}

\begin{proof}
This is a folklore result, and can be easily derived from facts in the literature.
For instance, \cite[Theorem 5.11]{Milici17} (applied to $n = 1$) shows that the $\kk$-algebra $D\left(1\right)$ of differential operators on the polynomial ring $\kk\left[x\right]$ is isomorphic to the Weyl algebra $\kk\left<D,U \mid DU - UD = 1\right> = \W$.
The actual proof of \cite[Theorem 5.11]{Milici17} shows that the $\kk$-algebra morphism $\W \to D\left(1\right)$ that sends $D$ and $U$ to $\dfrac{\partial}{\partial x}$ and the multiplication by $x$ is injective.
But this morphism is precisely the action of $\W$ on $\kk\left[x\right]$ (except that its target has been restricted to $D\left(1\right)$).
Thus, its injectivity means that the action is faithful.
This proves Lemma~\ref{lem.weyl.faith}.
\end{proof}

\begin{remark}
The Weyl algebra $\W$ acts not only on the polynomial ring $\kk\left[x\right]$, but also on the rings of Laurent polynomials $\kk\left[x,x^{-1}\right]$, of formal power series $\kk\left[\left[x\right]\right]$, of formal Laurent series $\kk\left(\left(x\right)\right)$, and (if $\kk = \RR$ or $\kk = \CC$) of infinitely differentiable functions $\mathcal{C}^\infty\left(\kk\right)$.
In each case, the action is faithful (since the polynomial ring $\kk\left[x\right]$ embeds into each of these rings),
and thus all our results still apply.
\end{remark}

Each word $u\in\M$ is mapped by $\phi$ to an element of the Weyl
algebra $\W$, which in turn acts on the polynomial ring
$\kk\left[  x\right]  $. Our next goal is to give an explicit formula
for this action in terms of a diagonal path $\pp$ that has $u$ as its
reading word. For the sake of simplicity, we extend the action of
$\W$ from the polynomial ring $\kk\left[  x\right]  $ to the
Laurent polynomial ring $\kk\left[  x,x^{-1}\right]  $ (so that we
don't have to worry about possible negative exponents on a power of $x$).

\begin{proposition}
\label{prop.UD-act} Let $\pp=\left(  p_{0},p_{1},\ldots,p_{k}\right)  $
be a diagonal path. Let $h_{i}:=\htt\left(  p_{i}\right)  $ for
each $i\in\left\{  0,1,\ldots,k\right\}  $. Then, for each $s\in\ZZ$,
we have%
\begin{equation}
 \left(  \phi\left(  \w\left(  \pp\right)  \right)
\right)  \left(  x^{s}\right)
= \left(  \prod_{p_{i}\text{ is an SE-step of }\pp}\left(
s+h_{k}-h_{i+1}\right)  \right)  \cdot x^{s+h_{k}-h_{0}}.
\label{eq.prop.UD-act.1}%
\end{equation}

\end{proposition}

\begin{example}
If the last steps of $\pp$ are
$\cdots p_{k-4}\searrow p_{k-3}\nearrow p_{k-2}\nearrow p_{k-1}\searrow p_{k}$,
then $\w\left(
\pp\right)  =\cdots DUUD$ and thus any $s\in\NN$ satisfies
\begin{align*}
\left(  \phi\left(  \w\left(  \pp\right)  \right)
\right)  \left(  x^{s}\right)   &  =\cdots DUUDx^{s}=\cdots\dfrac{\partial
}{\partial x}xx\underbrace{\dfrac{\partial}{\partial x}x^{s}}_{=sx^{s-1}%
}=s\cdots\dfrac{\partial}{\partial x}x\underbrace{xx^{s-1}}_{=x^{s}}\\
&  =s\cdots\dfrac{\partial}{\partial x}\underbrace{xx^{s}}_{=x^{s+1}}%
=s\cdots\underbrace{\dfrac{\partial}{\partial x}x^{s+1}}_{=\left(  s+1\right)
x^{s}}=s\left(  s+1\right)  \cdots x^{s}.
\end{align*}
The two factors $s$ and $s+1$ that we have found correspond precisely to the
two SE-steps $p_{k-1}$ and $p_{k-4}$ of $\pp$. Of course, further
SE-steps in $\pp$ will contribute more such factors.
\end{example}

The above example illustrates where Proposition \ref{prop.UD-act} comes from:
In general, we can compute $\left(  \phi\left(  \w\left(
\pp\right)  \right)  \right)  \left(  x^{s}\right)  $ in the same way,
decomposing $\phi\left(  \w\left(  \pp\right)  \right)
$ into a product of $D$'s and $U$'s (which correspond, respectively, to
SE-steps and NE-steps of $\w\left(  \pp\right)  $), and
letting each of these $D$'s and $U$'s act on the monomial $x^{s}$ sequentially
(starting with the last one). The letter $U$ acts as multiplication by $x$ and
thus sends $x^{k}$ to $x^{k+1}$, whereas the letter $D$ acts as $\dfrac
{\partial}{\partial x}$ and thus sends $x^{k}$ to $kx^{k-1}$. Thus, in total,
the exponent on our monomial $x^{s}$ is incremented once for each $U$ (that
is, for each NE-step of $\w\left(  \pp\right)  $) and
decremented once for each $D$ (that is, for each SE-step of $\w\left(  \pp\right)  $), so that it
becomes $s+h_{k}-h_{0}$ at the end (an easy consequence of \eqref{eq.U-D}).
The factors accumulating in front of the monomial are precisely the $k$'s
coming from the $D$'s, and thus are precisely the $s+h_{k}-h_{i+1}$
corresponding to the SE-steps $p_{i}$ of $\pp$, since it is these
SE-steps that turn into letters $D$ in $\phi\left(  \w\tup{\pp}\right)  $ (and the
current degree of the monomial at the time when $D$ is applied is exactly
$s+h_{k}-h_{i+1}$). The final result is precisely the right hand side of
(\ref{eq.prop.UD-act.1}).

\begin{vershort}
This argument can be easily translated into a rigorous proof of Proposition~\ref{prop.UD-act}, which can be found in the detailed version of this paper.
\end{vershort}

\begin{verlong}
This argument can be translated into the following rigorous proof of
Proposition~\ref{prop.UD-act}:

\begin{proof}
[Proof of Proposition~\ref{prop.UD-act}.]We induct on $k$.

The \textit{base case} ($k=0$) is obvious (here, $h_{k}=h_{0}$, whereas
$\w\left(  \pp\right)  $ is the empty word, and thus
$\phi\left(  \w\left(  \pp\right)  \right)  $ is the
unity of $\W$, and the product over the SE-steps of $\pp$ is empty).

For the \textit{induction step}, we fix a positive integer $k$, and assume
that the proposition is already proved for $k-1$ instead of $k$. Let
$s\in\ZZ$ be arbitrary. Our goal is to prove (\ref{eq.prop.UD-act.1}).

The definition of the action of $\W$ on $\kk\left[
x,x^{-1}\right]  $ yields%
\[
Ux^{s}=xx^{s}=x^{s+1}\ \ \ \ \ \ \ \ \ \ \text{and}\ \ \ \ \ \ \ \ \ \ Dx^{s}%
=\dfrac{\partial}{\partial x}x^{s}=sx^{s-1}.
\]

By removing the last step from our path $\pp=\left(  p_{0},p_{1}%
,\ldots,p_{k}\right)  $, we obtain the shorter diagonal path $\qq%
=\left(  p_{0},p_{1},\ldots,p_{k-1}\right)  $. We are in one of the following
two cases:

\textit{Case 1:} We have $p_{k-1}\nearrow p_{k}$.

\textit{Case 2:} We have $p_{k-1}\searrow p_{k}$.

Let us first consider Case 1. In this case, we have $p_{k-1}\nearrow p_{k}$.
Thus, the diagonal path $\pp$ is $\qq$ followed by the NE-step
$p_{k-1}\nearrow p_{k}$. Hence, the definition of a reading word yields
$w\left(  \pp\right)  =w\left(  \qq\right)  U$. Therefore,
\begin{align}
\phi\left(  \w\left(  \pp\right)  \right)   &
=\phi\left(  \w\left(  \qq\right)  U\right)
=\phi\left(  \w\left(  \qq\right)  \right)
\underbrace{\phi\left(  U\right)  }_{=U}\ \ \ \ \ \ \ \ \ \ \left(
\text{since }\phi\text{ is a monoid morphism}\right)  \nonumber\\
&  =\phi\left(  \w\left(  \qq\right)  \right)
U.\label{pf.prop.UD-act.2}%
\end{align}
Therefore,%
\begin{align}
\left(  \phi\left(  \w\left(  \pp\right)  \right)
\right)  \left(  x^{s}\right)   &  =\left(  \phi\left(  \w%
\left(  \qq\right)  \right)  U\right)  \left(  x^{s}\right)  =\left(
\phi\left(  \w\left(  \qq\right)  \right)  \right)
\underbrace{\left(  Ux^{s}\right)  }_{=x^{s+1}}\nonumber\\
&  =\left(  \phi\left(  \w\left(  \qq\right)  \right)
\right)  \left(  x^{s+1}\right)  .\label{pf.prop.UD-act.3}%
\end{align}

Recall that the path $\pp$ is $\qq$ followed by the NE-step
$p_{k-1}\nearrow p_{k}$. Hence, the SE-steps of $\pp$ are exactly the
SE-steps of $\qq$.

Since $p_{k-1}\nearrow p_{k}$, we have $h_{k-1}=h_{k}-1$ (since $h_{k-1}%
=\htt\left(  p_{k-1}\right)  $ and $h_{k}=\htt%
\left(  p_{k}\right)  $). Thus, $1+h_{k-1}=h_{k}$.

However, $\qq$ is a diagonal path of length $k-1$. Hence, the induction
hypothesis allows us to apply Proposition~\ref{prop.UD-act} to $k-1$,
$\qq$ and $s+1$ instead of $k$, $\pp$ and $s$. We thus obtain%
\begin{align*}
&  \left(  \phi\left(  \w\left(  \qq\right)  \right)
\right)  \left(  x^{s+1}\right) \\
&  =\left(  \prod_{p_{i}\text{ is an SE-step of }\qq}\left(
s+1+h_{k-1}-h_{i+1}\right)  \right)  \cdot x^{s+1+h_{k-1}-h_{0}}\\
&  =\left(  \prod_{p_{i}\text{ is an SE-step of }\qq}\left(
s+h_{k}-h_{i+1}\right)  \right)  \cdot x^{s+h_{k}-h_{0}}%
\ \ \ \ \ \ \ \ \ \ \left(  \text{since }1+h_{k-1}=h_{k}\right) \\
&  =\left(  \prod_{p_{i}\text{ is an SE-step of }\pp}\left(
s+h_{k}-h_{i+1}\right)  \right)  \cdot x^{s+h_{k}-h_{0}}%
\end{align*}
(here, we have replaced $\qq$ by $\pp$ under the product sign,
since the SE-steps of $\pp$ are exactly the SE-steps of $\qq$).
In view of (\ref{pf.prop.UD-act.3}), we can rewrite this as
\[
\left(  \phi\left(  \w\left(  \pp\right)  \right)
\right)  \left(  x^{s}\right)
= \left(  \prod_{p_{i}\text{ is an SE-step of
}\pp}\left(  s+h_{k}-h_{i+1}\right)  \right)  \cdot x^{s+h_{k}-h_{0}}.
\]
Thus, (\ref{eq.prop.UD-act.1}) is proved in Case 1.

Let us now consider Case 2. In this case, we have $p_{k-1}\searrow p_{k}$.
Thus, the diagonal path $\pp$ is $\qq$ followed by the SE-step
$p_{k-1}\searrow p_{k}$. Hence, just like we proved (\ref{pf.prop.UD-act.2})
in Case 1, we can now show that%
\[
\phi\left(  \w\left(  \pp\right)  \right)  =\phi\left(
\w\left(  \qq\right)  \right)  D.
\]
Therefore,%
\begin{align}
\left(  \phi\left(  \w\left(  \pp\right)  \right)
\right)  \left(  x^{s}\right)   &  =\left(  \phi\left(  \w%
\left(  \qq\right)  \right)  D\right)  \left(  x^{s}\right)  =\left(
\phi\left(  \w\left(  \qq\right)  \right)  \right)
\underbrace{\left(  Dx^{s}\right)  }_{\substack{=sx^{s-1}}}\nonumber\\
&  =s\cdot\left(  \phi\left(  \w\left(  \qq\right)
\right)  \right)  \left(  x^{s-1}\right)  .\label{pf.prop.UD-act.6}%
\end{align}

Recall that the path $\pp$ is $\qq$ followed by the SE-step
$p_{k-1}\searrow p_{k}$. Hence, the SE-steps of $\pp$ are exactly the
SE-steps of $\qq$ as well as the additional SE-step $p_{k-1}$.
Therefore,%
\begin{align}
&  \prod_{p_{i}\text{ is an SE-step of }\pp}\left(  s+h_{k}%
-h_{i+1}\right) \nonumber\\
&  =\underbrace{\left(  s+h_{k}-h_{\left(  k-1\right)  +1}\right)  }%
_{=s+h_{k}-h_{k}=s}\cdot\left(  \prod_{p_{i}\text{ is an SE-step of
}\qq}\left(  s+h_{k}-h_{i+1}\right)  \right) \nonumber\\
&  =s\cdot\left(  \prod_{p_{i}\text{ is an SE-step of }\qq}\left(
s+h_{k}-h_{i+1}\right)  \right)  . \label{pf.prop.UD-act.7}%
\end{align}

Since $p_{k-1}\searrow p_{k}$, we have $h_{k-1}=h_{k}+1$ (since $h_{k-1}%
=\htt\left(  p_{k-1}\right)  $ and $h_{k}=\htt%
\left(  p_{k}\right)  $). Thus, $-1+h_{k-1}=h_{k}$.

However, $\qq$ is a diagonal path of length $k-1$. Hence, the induction
hypothesis allows us to apply Proposition~\ref{prop.UD-act} to $k-1$,
$\qq$ and $s-1$ instead of $k$, $\pp$ and $s$. We thus obtain%
\begin{align*}
&  \left(  \phi\left(  \w\left(  \qq\right)  \right)
\right)  \left(  x^{s-1}\right)  \\
&  =\left(  \prod_{p_{i}\text{ is an SE-step of }\qq}\left(
s-1+h_{k-1}-h_{i+1}\right)  \right)  \cdot x^{s-1+h_{k-1}-h_{0}}\\
&  =\left(  \prod_{p_{i}\text{ is an SE-step of }\qq}\left(
s+h_{k}-h_{i+1}\right)  \right)  \cdot x^{s+h_{k}-h_{0}}%
\ \ \ \ \ \ \ \ \ \ \left(  \text{since }-1+h_{k-1}=h_{k}\right)  .
\end{align*}
Thus, (\ref{pf.prop.UD-act.6}) can be rewritten as%
\begin{align*}
\left(  \phi\left(  \w\left(  \pp\right)  \right)
\right)  \left(  x^{s}\right)   &  =\underbrace{s\cdot\left(  \prod
_{p_{i}\text{ is an SE-step of }\qq}\left(  s+h_{k}-h_{i+1}\right)
\right)  }_{\substack{=\prod\limits_{p_{i}\text{ is an SE-step of }\pp}\left(
s+h_{k}-h_{i+1}\right)  \\\text{(by (\ref{pf.prop.UD-act.7}))}}}\cdot
\,x^{s+h_{k}-h_{0}}\\
&  =\left(  \prod_{p_{i}\text{ is an SE-step of }\pp}\left(
s+h_{k}-h_{i+1}\right)  \right)  \cdot x^{s+h_{k}-h_{0}}.
\end{align*}
Thus, (\ref{eq.prop.UD-act.1}) is proved in Case 2.

We have now proved (\ref{eq.prop.UD-act.1}) in both Cases 1 and 2. Thus, the
induction step is complete. This completes the proof of
Proposition~\ref{prop.UD-act}.
\end{proof}
\end{verlong}

The following is a version of Proposition~\ref{prop.UD-act} in which the
reading word $\w\left(  \pp\right)  $ additionally
undergoes the ``toggle-and-reverse''\ anti-automorphism $\omega$:

\begin{proposition}
\label{prop.DU-act}
Let $\pp=\left(  p_{0},p_{1},\ldots,p_{k}\right)  $
be a diagonal path. Let $h_{i}:=\htt\left(  p_{i}\right)  $ for
each $i\in\left\{  0,1,\ldots,k\right\}  $. Then, for each $s\in\ZZ$,
we have%
\[
\left(  \omega\left(  \phi\left(  \w\left(  \pp\right)
\right)  \right)  \right)  \left(  x^{s}\right)  =\left(  \prod_{p_{i}\text{
is an NE-step of }\pp}\left(  s+h_{0}-h_{i}\right)  \right)  \cdot
x^{s+h_{0}-h_{k}}.
\]

\end{proposition}

\begin{vershort}
The proof of this is similar to that of Proposition~\ref{prop.UD-act}; again, we refer to the detailed version for the details.
\end{vershort}

\begin{verlong}
\begin{proof}
This can be proved similarly to Proposition~\ref{prop.UD-act}. Alternatively,
this can be derived from Proposition~\ref{prop.UD-act} as follows:

\begin{figure}[htbp]
\begin{center}
\begin{tikzpicture}[scale=0.5]
\draw
[line width=2pt] (5,1) -- (6,2) -- (7,3) -- (8,4) -- (9,3) -- (10,2) -- (11,3) -- (12,2) -- (13,1) -- (14,0);
\draw[fill=blue] (5, 1) circle (0.1) node[left] {};
\draw[fill=blue] (14, 0) circle (0.1) node[right] {};
\node (p) at (12, 3) {$\pp$};
\end{tikzpicture}
\qquad
\begin{tikzpicture}
\draw (0,0.5mm) -- (0,-0.5mm);
\newlength\mylength
\setlength{\mylength}{\widthof{reflection}}
\draw[->] (0,0) -- (1.2\mylength,0) node[above,midway] {reflection};
\end{tikzpicture}
\qquad
\begin{tikzpicture}[xscale=-1, scale=0.5]
\draw
[line width=2pt] (5,1) -- (6,2) -- (7,3) -- (8,4) -- (9,3) -- (10,2) -- (11,3) -- (12,2) -- (13,1) -- (14,0);
\draw[fill=blue] (5, 1) circle (0.1) node[left] {};
\draw[fill=blue] (14, 0) circle (0.1) node[right] {};
\node (q) at (6, 3) {$\qq$};
\end{tikzpicture}
\end{center}
\caption{A path $\pp$ (left) and its reflection $\qq$ across a vertical axis (right).}
\label{fig.reflect}
\end{figure}

Let $s\in\ZZ$. Let $\qq$ be a diagonal path obtained by
reflecting $\pp$ across a vertical axis. Then, each NE-step of
$\pp$ becomes an SE-step in $\qq$ and vice versa, and moreover,
the order of all steps is reversed
(see Figure~\ref{fig.reflect}). Hence, $\w\left(
\qq\right)  =\omega\left(  \w\left(  \pp\right)
\right)  $ and thus
\[
\phi\left(  \w\left(  \qq\right)  \right)  =\phi\left(
\omega\left(  \w\left(  \pp\right)  \right)  \right)  =\omega\left(
\phi\left(  \w\left(  \pp\right)  \right)  \right)
\ \ \ \ \ \ \ \ \ \ \left(  \text{since }\omega\circ\phi=\phi\circ
\omega\right)  .
\]
Let us write the path $\qq$ as $\qq=\left(  q_{0},q_{1}%
,\ldots,q_{k}\right)  $ (we can do this, since it clearly has the same length
as $\pp=\left(  p_{0},p_{1},\ldots,p_{k}\right)  $). Let $g_{i}%
:=\htt\left(  q_{i}\right)  $ for each $i\in\left\{
0,1,\ldots,k\right\}  $. Since $\qq$ is the reflection of $\pp$
across a vertical axis, we thus have
\begin{equation}
g_{j}=h_{k-j}\ \ \ \ \ \ \ \ \ \ \text{for each }j\in\left\{  0,1,\ldots
,k\right\}  \label{pf.prop.DU-act.g=h}%
\end{equation}
(since the height of a point does not change when we reflect it across a
vertical axis). In particular, $g_{0}=h_{k}$ and $g_k=h_{0}$. However,
Proposition~\ref{prop.UD-act} (applied to $\qq$, $q_{i}$ and $g_{i}$
instead of $\pp$, $p_{i}$ and $h_{i}$) yields%
\begin{align}
\left(  \phi\left(  \w\left(  \qq\right)  \right)
\right)  \left(  x^{s}\right)   &  =\left(  \prod_{q_{i}\text{ is an SE-step
of }\qq}\left(  s+g_k-g_{i+1}\right)  \right)  \cdot x^{s+g_k-g_{0}}\nonumber\\
&  =\left(  \prod_{q_{i}\text{ is an SE-step of }\qq}\left(
s+h_{0}-h_{k-i-1}\right)  \right)  \cdot x^{s+h_{0}-h_{k}}%
\label{pf.prop.DU-act.2}%
\end{align}
(since (\ref{pf.prop.DU-act.g=h}) yields $g_{i+1}=h_{k-\left(  i+1\right)
}=h_{k-1-i}$ and $g_k=h_{0}$ and $g_{0}=h_{k}$). But the SE-steps of
$\qq$ correspond to the NE-steps of $\pp$; more specifically, a
given vertex $q_{i}$ of $\qq$ is an  SE-step of $\qq$ if and
only if $p_{k-1-i}$ is an NE-step of $\pp$. Hence,
\begin{align*}
& \prod_{q_{i}\text{ is an SE-step of }\qq}\left(  s+h_{0}%
-h_{k-i-1}\right)  \\
& =\prod_{p_{k-1-i}\text{ is an NE-step of }\pp}\left(  s+h_{0}%
-h_{k-i-1}\right)  \\
& =\prod_{p_{i}\text{ is an NE-step of }\pp}\left(  s+h_{0}%
-h_{i}\right)  \ \ \ \ \ \ \ \ \ \ \left(
\begin{array}
[c]{c}%
\text{here, we have substituted }i\\
\text{for }k-i-1\text{ in the product}%
\end{array}
\right)  .
\end{align*}
Thus, we can rewrite (\ref{pf.prop.DU-act.2}) as%
\[
\left(  \phi\left(  \w\left(  \qq\right)  \right)
\right)  \left(  x^{s}\right)  =\left(  \prod_{p_{i}\text{ is an NE-step of
}\pp}\left(  s+h_{0}-h_{i}\right)  \right)  \cdot x^{s+h_{0}-h_{k}}.
\]

In view of $\phi\left(  \w\left(  \qq\right)  \right)
=\omega\left(  \phi\left(  \w\left(  \pp\right)
\right)  \right)  $, we can rewrite this further as%
\[
\left(  \omega\left(  \phi\left(  \w\left(  \pp\right)
\right)  \right)  \right)  \left(  x^{s}\right)  =\left(  \prod_{p_{i}\text{
is an NE-step of }\pp}\left(  s+h_{0}-h_{i}\right)  \right)  \cdot
x^{s+h_{0}-h_{k}}.
\]
This proves Proposition~\ref{prop.DU-act}.
\end{proof}
\end{verlong}

As a consequence of Proposition~\ref{prop.DU-act}, we can easily obtain the following:

\begin{proposition}
\label{prop.DU-eqlett} Let $\pp=\left(  p_{0},p_{1},\ldots
,p_{k}\right)  $ and $\qq=\left(  q_{0},q_{1},\ldots,q_{m}\right)  $ be
two diagonal paths with the same initial height that satisfy $\phi\left(
\w\left(  \pp\right)  \right)  =\phi\left(
\w\left(  \qq\right)  \right)  $. Then, the final heights
of $\pp$ and $\qq$ are equal, and we have
\begin{align}
&  \left\{  \htt\left(  p_{i}\right)  \ \mid\ p_{i}\text{ is an
NE-step of }\pp\right\}  _{\multiset} \nonumber \\
&  =\left\{  \htt\left(  q_{i}\right)  \ \mid\ q_{i}\text{ is an
NE-step of }\qq\right\}  _{\multiset}.\label{eq:equality_step_heights}
\end{align}

\end{proposition}

\begin{vershort}
\begin{proof}
Let $h_{i}:=\htt\left(  p_{i}\right)  $ for each $i\in\left\{
0,1,\ldots,k\right\}  $. Let $g_{i}:=\htt\left(  q_{i}\right)  $
for each $i\in\left\{  0,1,\ldots,m\right\}  $. Note that $h_{0}=g_{0}$ (since
$\pp$ and $\qq$ have the same initial height).

Let $s\in\NN$ be high enough to be larger than all numbers $h_{i}%
-h_{0}$ for all NE-steps $p_{i}$ of $\pp$ and also larger than all
numbers $g_{i}-g_{0}$ for all NE-steps $q_{i}$ of $\qq$.

Then, Proposition \ref{prop.DU-act} yields
\[
\left(  \omega\left(  \phi\left(  \w\left(  \pp\right)
\right)  \right)  \right)  \left(  x^{s}\right)  =\left(  \prod_{p_{i}\text{
is an NE-step of }\pp}\left(  s+h_{0}-h_{i}\right)  \right)  \cdot
x^{s+h_{0}-h_{k}}%
\]
and similarly%
\[
\left(  \omega\left(  \phi\left(  \w\left(  \qq\right)
\right)  \right)  \right)  \left(  x^{s}\right)  =\left(  \prod_{q_{i}\text{
is an NE-step of }\qq}\left(  s+g_{0}-g_{i}\right)  \right)  \cdot
x^{s+g_{0}-g_{m}}.
\]
But the left hand sides of these two equalities are equal, since $\phi\left(
\w\left(  \pp\right)  \right)  =\phi\left(
\w\left(  \qq\right)  \right)  $. Thus, the right hand
sides are also equal. In other words, we have%
\begin{align}
& \left(  \prod_{p_{i}\text{ is an NE-step of }\pp}\left(
s+h_{0}-h_{i}\right)  \right)  \cdot x^{s+h_{0}-h_{k}}\nonumber\\
& =\left(  \prod_{q_{i}\text{ is an NE-step of }\qq}\left(
s+g_{0}-g_{i}\right)  \right)  \cdot x^{s+g_{0}-g_{m}}%
.\label{pf.prop.DU-eqlett.1}%
\end{align}
The products on both sides of this equality are nonzero (since $s$ is larger
than all numbers $h_{i}-h_{0}$ for all NE-steps $p_{i}$ of $\pp$ and
also larger than all numbers $g_{i}-g_{0}$ for all NE-steps $q_{i}$ of
$\qq$). Thus,~\eqref{pf.prop.DU-eqlett.1} entails that the exponents $s+h_{0}-h_{k}$ and
$s+g_{0}-g_{m}$ are equal, and therefore we conclude that $h_{0}-h_{k}%
=g_{0}-g_{m}$. Since $h_{0}=g_{0}$, this entails $h_{k}=g_{m}$. In other
words, the final heights of $\pp$ and $\qq$ are equal.

Since the two exponents $s+h_{0}-h_{k}$ and $s+g_{0}-g_{m}$ in
(\ref{pf.prop.DU-eqlett.1}) are equal, we obtain%
\begin{equation}\label{eq:NE-step-identity}
\prod_{p_{i}\text{ is an NE-step of }\pp}\left(  s+h_{0}-h_{i}\right)
=\prod_{q_{i}\text{ is an NE-step of }\qq}\left(  s+g_{0}-g_{i}\right)
\end{equation}
by comparing coefficients in (\ref{pf.prop.DU-eqlett.1}). Now forget that we fixed $s$. 
We just proved the equality~\eqref{eq:NE-step-identity}
for each sufficiently high $s\in\NN$. But this equality is a polynomial
identity in $s$, and thus must hold formally (since it holds for each
sufficiently high $s\in\NN$). In other words, it must hold
in the polynomial ring $\kk\left[  x\right]  $ if we replace $s$ by~$x$.

Now recall that $h_{0} =g_{0}$. So~\eqref{eq:NE-step-identity} still holds
as a polynomial identity if we substitute $x$ for $s + h_0 = s + g_0$
on both sides to transform the identity into%
\[
\prod_{p_{i}\text{ is an NE-step of }\pp}\left(  x-h_{i}\right)
=\prod_{q_{i}\text{ is an NE-step of }\qq}\left(  x-g_{i}\right)  .
\]

Since $\kk\left[  x\right]  $ is a unique factorization domain, this
yields that%
\begin{align*}
&  \left\{  h_{i}\ \mid\ p_{i}\text{ is an NE-step of }\pp\right\}
_{\multiset}\\
&  =\left\{  g_{i}\ \mid\ q_{i}\text{ is an NE-step of }\qq\right\}
_{\multiset}.
\end{align*}
This is exactly~\eqref{eq:equality_step_heights} since $h_{i}=\htt\left(  p_{i}\right)  $ and $g_{i}%
=\htt\left(  q_{i}\right)  $ for all respective $i$. This
completes the proof of Proposition \ref{prop.DU-eqlett}, since we have already
shown that the final heights of $\pp$ and $\qq$ are equal.
\end{proof}
\end{vershort}

\begin{verlong}
\begin{proof}
Let $h_{i}:=\htt\left(  p_{i}\right)  $ for each $i\in\left\{
0,1,\ldots,k\right\}  $. Let $g_{i}:=\htt\left(  q_{i}\right)  $
for each $i\in\left\{  0,1,\ldots,m\right\}  $. Note that $h_{0}=g_{0}$ (since
$\pp$ and $\qq$ have the same initial height).

Let $s\in\NN$ be high enough to be larger than all numbers $h_{i}%
-h_{0}$ for all NE-steps $p_{i}$ of $\pp$ and also larger than all
numbers $g_{i}-g_{0}$ for all NE-steps $q_{i}$ of $\qq$.

Then, Proposition \ref{prop.DU-act} yields
\[
\left(  \omega\left(  \phi\left(  \w\left(  \pp\right)
\right)  \right)  \right)  \left(  x^{s}\right)  =\left(  \prod_{p_{i}\text{
is an NE-step of }\pp}\left(  s+h_{0}-h_{i}\right)  \right)  \cdot
x^{s+h_{0}-h_{k}}%
\]
and similarly%
\[
\left(  \omega\left(  \phi\left(  \w\left(  \qq\right)
\right)  \right)  \right)  \left(  x^{s}\right)  =\left(  \prod_{q_{i}\text{
is an NE-step of }\qq}\left(  s+g_{0}-g_{i}\right)  \right)  \cdot
x^{s+g_{0}-g_{m}}.
\]
But the left hand sides of these two equalities are equal, since $\phi\left(
\w\left(  \pp\right)  \right)  =\phi\left(
\w\left(  \qq\right)  \right)  $. Thus, the right hand
sides are also equal. In other words, we have%
\begin{align}
& \left(  \prod_{p_{i}\text{ is an NE-step of }\pp}\left(
s+h_{0}-h_{i}\right)  \right)  \cdot x^{s+h_{0}-h_{k}}\nonumber\\
& =\left(  \prod_{q_{i}\text{ is an NE-step of }\qq}\left(
s+g_{0}-g_{i}\right)  \right)  \cdot x^{s+g_{0}-g_{m}}%
.\label{pf.prop.DU-eqlett.1}%
\end{align}
The products on both sides of this equality are nonzero (since $s$ is larger
than all numbers $h_{i}-h_{0}$ for all NE-steps $p_{i}$ of $\pp$ and
also larger than all numbers $g_{i}-g_{0}$ for all NE-steps $q_{i}$ of
$\qq$). Thus,~\eqref{pf.prop.DU-eqlett.1} entails that the exponents $s+h_{0}-h_{k}$ and
$s+g_{0}-g_{m}$ are equal, and therefore we conclude that $h_{0}-h_{k}%
=g_{0}-g_{m}$. Since $h_{0}=g_{0}$, this entails $h_{k}=g_{m}$. In other
words, the final heights of $\pp$ and $\qq$ are equal.

Since the two exponents $s+h_{0}-h_{k}$ and $s+g_{0}-g_{m}$ in
(\ref{pf.prop.DU-eqlett.1}) are equal, we obtain%
\[
\prod_{p_{i}\text{ is an NE-step of }\pp}\left(  s+h_{0}-h_{i}\right)
=\prod_{q_{i}\text{ is an NE-step of }\qq}\left(  s+g_{0}-g_{i}\right)
\]
by comparing coefficients in (\ref{pf.prop.DU-eqlett.1}).

Forget that we fixed $s$. We just proved the equality
\[
\prod_{p_{i}\text{ is an NE-step of }\pp}\left(  s+h_{0}-h_{i}\right)
=\prod_{q_{i}\text{ is an NE-step of }\qq}\left(  s+g_{0}-g_{i}\right)
\]
for each sufficiently high $s\in\NN$. But this equality is a polynomial
identity in $s$, and thus must hold formally (since it holds for each
sufficiently high $s\in\NN$). In other words, we have%
\[
\prod_{p_{i}\text{ is an NE-step of }\pp}\left(  x+h_{0}-h_{i}\right)
=\prod_{q_{i}\text{ is an NE-step of }\qq}\left(  x+g_{0}-g_{i}\right)
\]
in the polynomial ring $\kk\left[  x\right]  $. In view of $h_{0}%
=g_{0}$, we can rewrite this equality as%
\[
\prod_{p_{i}\text{ is an NE-step of }\pp}\left(  x+g_{0}-h_{i}\right)
=\prod_{q_{i}\text{ is an NE-step of }\qq}\left(  x+g_{0}-g_{i}\right)
.
\]
Substituting $x$ for $x+g_{0}$ on both sides of this equality, we transform it
into%
\[
\prod_{p_{i}\text{ is an NE-step of }\pp}\left(  x-h_{i}\right)
=\prod_{q_{i}\text{ is an NE-step of }\qq}\left(  x-g_{i}\right)  .
\]

Since $\kk\left[  x\right]  $ is a unique factorization domain, this
yields that%
\begin{align*}
&  \left\{  h_{i}\ \mid\ p_{i}\text{ is an NE-step of }\pp\right\}
_{\multiset}\\
&  =\left\{  g_{i}\ \mid\ q_{i}\text{ is an NE-step of }\qq\right\}
_{\multiset}.
\end{align*}
In other words,%
\begin{align*}
&  \left\{  \htt\left(  p_{i}\right)  \ \mid\ p_{i}\text{ is an
NE-step of }\pp\right\}  _{\multiset}\\
&  =\left\{  \htt\left(  q_{i}\right)  \ \mid\ q_{i}\text{ is an
NE-step of }\qq\right\}  _{\multiset}%
\end{align*}
(since $h_{i}=\htt\left(  p_{i}\right)  $ and $g_{i}%
=\htt\left(  q_{i}\right)  $ for all respective $i$). This
completes the proof of Proposition \ref{prop.DU-eqlett} (since we have already
shown that the final heights of $\pp$ and $\qq$ are equal).
\end{proof}
\end{verlong}

\begin{proposition}
\label{prop.DU-eqlett2}
Let $u$ and $v$ be two words in $\M$ such
that $\phi\left(  u\right)  =\phi\left(  v\right)  $. Then, $u$ and $v$
contain the same number of $D$'s and the same number of $U$'s.
\end{proposition}

\begin{proof}
Let $\pp=\left(  p_{0},p_{1},\ldots,p_{k}\right)  $ be the standard path of $u$, and let $\qq=\left(  q_{0},q_{1},\ldots
,q_{m}\right)  $ be the standard path of $v$.
The paths $\pp$ and $\qq$ both start at $\tup{0,0}$ (by the definition of a standard path), and thus have the same initial height (namely, $0$).
Moreover, their reading words are $\w\tup{\pp} = u$ and $\w\tup{\qq} = v$ (since $\pp$ and $\qq$ are the standard paths of $u$ and $v$).
Thus, from $\phi\left(  u\right)  =\phi\left(  v\right)  $, we obtain
$\phi\left(  \w\left(  \pp\right)  \right)  =\phi\left(
\w\left(  \qq\right)  \right)  $.
Hence, Proposition \ref{prop.DU-eqlett} yields
that the final heights of $\pp$ and $\qq$ are equal, and that
\begin{align*}
&  \left\{  \htt\left(  p_{i}\right)  \ \mid\ p_{i}\text{ is an
NE-step of }\pp\right\}  _{\multiset}\\
&  =\left\{  \htt\left(  q_{i}\right)  \ \mid\ q_{i}\text{ is an
NE-step of }\qq\right\}  _{\multiset}.
\end{align*}
The latter equality yields (in particular) that the number of NE-steps of
$\pp$ equals the number of NE-steps of $\qq$. Since the NE-steps
of $\pp$ are in bijection with the $U$'s in $\w\left(
\pp\right)  =u$, and similarly for $\qq$, we can rewrite this as
follows: The number of $U$'s in $u$ equals the number of $U$'s in $v$. In
other words, $u$ and $v$ contain the same number of $U$'s.

It remains to see that the words $u$ and $v$ contain the same number of $D$'s.
But \eqref{eq.U-D} shows that%
\[
\left(  \text{final height of }\pp\right)  -\left(  \text{initial height of }\pp\right)
=\left(  \text{\# of $U$'s in }\w\left(  \pp\right)  \right)
-\left(  \text{\# of $D$'s in }\w\left(  \pp\right)  \right)
\]
and
\[
\left(  \text{final height of }\qq\right)
-\left(  \text{initial height of }\qq\right)
=\left(  \text{\# of $U$'s in }\w\left(  \qq\right)  \right)
-\left(  \text{\# of $D$'s in }\w\left(  \qq\right)  \right)  .
\]
The left hand sides of these two equalities are equal (since the paths
$\pp$ and $\qq$ have the same initial height and the same final
height). Thus, so are the right hand sides:
\[
\left(  \text{\# of $U$'s in }\w\left(  \pp\right)
\right)  -\left(  \text{\# of $D$'s in }\w\left(  \pp%
\right)  \right)  =\left(  \text{\# of $U$'s in }\w\left(
\qq\right)  \right)  -\left(  \text{\# of $D$'s in }\w%
\left(  \qq\right)  \right)  .
\]
In other words,%
\[
\left(  \text{\# of $U$'s in }u\right)  -\left(  \text{\# of $D$'s in
}u\right)  =\left(  \text{\# of $U$'s in }v\right)  -\left(  \text{\# of $D$'s
in }v\right)
\]
(since $u=\w\left(  \pp\right)  $ and
$v=\w\left(  \qq\right)  $). Since $\left(  \text{\# of
$U$'s in }u\right)  =\left(  \text{\# of $U$'s in }v\right)  $ (because $u$
and $v$ contain the same number of $U$'s), we thus conclude that $\left(
\text{\# of $D$'s in }u\right)  =\left(  \text{\# of $D$'s in }v\right)  $. In
other words, $u$ and $v$ contain the same number of $D$'s. Proposition
\ref{prop.DU-eqlett2} is thus fully proved.
\end{proof}

\begin{lemma}
\label{lem.act-diag}
Let $u\in\M$ be a balanced word. Let
$s\in\ZZ$. Then, $\left(  \phi\left(  u\right)  \right)  \left(
x^{s}\right)  =\lambda_{u,s}x^{s}$ for some $\lambda_{u,s}\in\ZZ$.
\end{lemma}

\begin{vershort}
\begin{proof}
Let $\pp = \tup{p_0, p_1, \ldots, p_k}$ be the standard path of $u$.
Both the initial height and the final height of $\pp$ are $0$ (since $u$ is balanced).
Thus, the claim follows easily from Proposition \ref{prop.UD-act}.
\end{proof}
\end{vershort}

\begin{verlong}
\begin{proof}
Let $\pp=\left(  p_{0},p_{1},\ldots,p_{k}\right)  $
be the standard path of $u$.
This path $\pp$ starts at $\tup{0,0}$ and thus has initial height $0$.
Thus, the final height of $\pp$ is also $0$ (by
\eqref{eq.U-D}, since the word $\w\left(  \pp\right)  =u$
is balanced).

Let $h_{i}:=\htt\left(  p_{i}\right)  $ for each $i\in\left\{
0,1,\ldots,k\right\}  $. Thus, $h_{k}=0$ (since the final height of
$\pp$ is $0$) and $h_{0}=0$ (since the initial height of $\pp$
is $0$). Now, Proposition \ref{prop.UD-act} yields%
\[
\left(  \phi\left(  \w\left(  \pp\right)  \right)
\right)  \left(  x^{s}\right)  =\left(  \prod_{p_{i}\text{ is an SE-step of
}\pp}\left(  s+h_{k}-h_{i+1}\right)  \right)  \cdot x^{s+h_{k}-h_{0}}.
\]
Since $\w\left(  \pp\right)  =u$ and
$s+\underbrace{h_{k}}_{=0}-\underbrace{h_{0}}_{=0}=s$, we can rewrite this as%
\[
\left(  \phi\left(  u\right)  \right)  \left(  x^{s}\right)  =\left(
\prod_{p_{i}\text{ is an SE-step of }\pp}\left(  s+h_{k}-h_{i+1}%
\right)  \right)  \cdot x^{s}.
\]
In other words, $\left(  \phi\left(  u\right)  \right)  \left(  x^{s}\right)
=\lambda_{u,s}x^{s}$ for some $\lambda_{u,s}\in\ZZ$ (namely, for \newline
$\lambda_{u,s}=\prod\limits_{p_{i}\text{ is an SE-step of }\pp}\left(
s+h_{k}-h_{i+1}\right)  $). This proves Lemma \ref{lem.act-diag}.
\end{proof}
\end{verlong}

\begin{lemma}
\label{lem.bal.weyl0}
Let $a$ and $b$ be two balanced words in $\M$.
Then, $\phi\tup{a} \phi\tup{b} = \phi\tup{b} \phi\tup{a}$.
\end{lemma}

\begin{vershort}
\begin{proof}[First proof.]
Lemma \ref{lem.act-diag} shows that the elements $\phi\tup{a}$ and $\phi\tup{b}$ of $\W$ act on the $\W$-module $\kk\ive{x}$ as diagonal matrices.
Since diagonal matrices commute, we thus conclude that the actions of $\phi\tup{a}$ and $\phi\tup{b}$ on $\kk\ive{x}$ commute.
Hence, the elements $\phi\tup{a}$ and $\phi\tup{b}$ themselves commute (since the $\W$-module $\kk\ive{x}$ is faithful).
\end{proof}
\end{vershort}

\begin{verlong}
\begin{proof}[First proof.]
Let $s\in\NN$. Then,
Lemma \ref{lem.act-diag} (applied to $u=a$) shows that
$\left(  \phi\left(  a\right)  \right)  \left(  x^{s}\right)  =\lambda
_{a,s}x^{s}$ for some $\lambda_{a,s}\in\ZZ$. Similarly, $\left(
\phi\left(  b\right)  \right)  \left(  x^{s}\right)  =\lambda_{b,s}x^{s}$ for
some $\lambda_{b,s}\in\ZZ$. Consider these $\lambda_{a,s}$ and
$\lambda_{b,s}$.
Now,
\begin{align*}
\left(  \phi\tup{a} \phi\tup{b} \right)  \left(  x^{s}\right) 
&=\left(
\phi\left(  a\right)  \right)  \underbrace{\left(  \phi\left(  b\right)
\right)  \left(  x^{s}\right)  }_{=\lambda_{b,s}x^{s}}=\left(  \phi\left(
a\right)  \right)  \left(  \lambda_{b,s}x^{s}\right)  =\lambda_{b,s}%
\underbrace{\left(  \phi\left(  a\right)  \right)  \left(  x^{s}\right)
}_{=\lambda_{a,s}x^{s}} \\
&=\lambda_{b,s}\lambda_{a,s}x^{s}.
\end{align*}
Similarly,
\[
\left(  \phi\tup{b} \phi\tup{a}  \right)  \left(  x^{s}\right)  =\lambda
_{a,s}\lambda_{b,s}x^{s}.
\]
The right hand sides of these two equalities are equal (since $\lambda
_{b,s}\lambda_{a,s}=\lambda_{a,s}\lambda_{b,s}$). Hence, so are their left
hand sides. In other words, $\left( \phi\tup{a} \phi\tup{b} \right)  \left(
x^{s}\right)  =\left(  \phi\tup{b} \phi\tup{a} \right)  \left(  x^{s}\right)  $.

Forget that we fixed $s$. We thus have shown that
$\left( \phi\tup{a} \phi\tup{b} \right)  \left(
x^{s}\right)  =\left(  \phi\tup{b} \phi\tup{a} \right)  \left(  x^{s}\right)  $
for each $s\in\NN$. By linearity, this
entails that $\left( \phi\tup{a} \phi\tup{b} \right)  \left(  p\right)
=\left( \phi\tup{b} \phi\tup{a} \right)  \left(  p\right)  $ for any
polynomial $p\in\kk\left[  x\right]  $. Since the action of
$\W$ on $\kk\left[  x\right]  $ is faithful, this shows that
$\phi\tup{a} \phi\tup{b} = \phi\tup{b} \phi\tup{a}$ in $\W$.
This proves Lemma \ref{lem.bal.weyl0}.
\end{proof}
\end{verlong}

\begin{proof}[Second proof.]
The following alternative proof has been suggested to us by J\"orgen Backelin.
It shows that Lemma~\ref{lem.bal.weyl0} is a disguised form of a classical result known already to Dixmier \cite{Dixmie68}.

We equip the Weyl algebra
$\W$ with a $\ZZ$-grading by deciding that its generators $U$
and $D$ be homogeneous of degrees $1$ and $-1$, respectively. The $0$-th
graded component $\W_{0}$ of $\W$ is then spanned by the
images of the balanced words under $\phi$.
In particular, both $\phi\tup{a}$ and $\phi\tup{b}$ belong to $\W_0$ (since $a$ and $b$ are balanced words).

However, a result of Dixmier (\cite[last equation in \S 3.2]{Dixmie68}) says that the $\kk$-algebra
$\W_{0}$ is generated by the single element $\phi\left(  DU\right)$.
The proof of this result in \cite{Dixmie68} is fairly easy: We abbreviate
$\phi\left(  D\right)  $ and $\phi\left(  U\right)  $ as $\overline{D}$ and
$\overline{U}$. First it is shown that $\W$ is spanned by the
elements of the form $\overline{D}^{i}\overline{U}^{j}$ with $i,j\in
\NN$, which are homogeneous of respective degrees $j-i$. Therefore the
$0$-th graded component $\W_{0}$ is spanned by the elements of the
form $\overline{D}^{i}\overline{U}^{i}$ with $i\in\NN$. But each of the
latter elements can be rewritten as
\[
\overline{D}^{i}\overline{U}^{i}=\left(  \overline{D}\overline{U}\right)
\left(  \overline{D}\overline{U}+1\right)  \left(  \overline{D}\overline
{U}+2\right)  \cdots\left(  \overline{D}\overline{U}+i-1\right)
\]
(as can be proved by induction on $i$), which is clearly a polynomial in
$\overline{D}\overline{U}=\phi\left(  DU\right)  $. Thus the algebra
$\W_0$ is generated by the single element $\phi\left(  DU\right)$.

Now, the algebra $\W_0$ is commutative (since we have just shown that it is generated by a single element).
Therefore, any two of its elements commute.
In particular, $\phi\tup{a}$ and $\phi\tup{b}$ commute
(since both $\phi\tup{a}$ and $\phi\tup{b}$ belong to $\W_0$).
In other words, $\phi\tup{a} \phi\tup{b} = \phi\tup{b} \phi\tup{a}$.
This proves Lemma~\ref{lem.bal.weyl0} again.
\end{proof}

\begin{lemma}
\label{lem.bal.weyl}
Let $u,v\in\M$ be two words such that
$u \balsim v$. Then, $\phi\left(  u\right)
=\phi\left(  v\right)  $.
\end{lemma}

\begin{vershort}
\begin{proof}
It suffices to show that if $p, q \in \M$ are two words, and if $x, y \in \M$ are two balanced words, then $\phi\tup{pxyq} = \phi\tup{pyxq}$.
But this follows from Lemma~\ref{lem.bal.weyl0} (which yields $\phi\tup{x} \phi\tup{y} = \phi\tup{y} \phi\tup{x}$), since $\phi$ is a monoid morphism.
\end{proof}
\end{vershort}

\begin{verlong}
\begin{proof}
We have $u \balsim v$. Thus, by the definition of
the relation $\balsim$, the word $u$ can be
transformed into $v$ by a sequence of balanced commutations. Thus, we must
show that the image $\phi\left(  w\right)  $ of a word $w\in\M$ is
preserved whenever we apply a balanced commutation to $w$.
In other words, we must show that if a word $v$ is obtained from a word $w$ by a balanced commutation, then $\phi\tup{v} = \phi\tup{w}$.

So let $v$ and $w$ be two words such that $v$ is obtained from $w$ by a balanced commutation.
Thus, we can write $v$ and $w$ as $v = pxyq$ and $w = pyxq$, where $p, q \in \M$ are two words
and where $x, y \in \M$ are two balanced words.
Consider these $p, q, x, y$.
Since $x$ and $y$ are balanced, we have $\phi\tup{x} \phi\tup{y} = \phi\tup{y} \phi\tup{x}$ by Lemma~\ref{lem.bal.weyl0}.
However, from $v = pxyq$, we obtain
\[
\phi\tup{v} = \phi\tup{pxyq} = \phi\tup{p} \phi\tup{x} \phi\tup{y} \phi\tup{q} \ \ \ \ \ \ \ \ \ \ \ \left(\text{since $\phi$ is a monoid morphism}\right) .
\]
Similarly,
\[
\phi\tup{w} = \phi\tup{p} \phi\tup{y} \phi\tup{x} \phi\tup{q}
\ \ \ \ \ \ \ \ \ \ \ \left(\text{since $w = pyxq$}\right).
\]
The right hand sides of these two equalities are equal (since $\phi\tup{x} \phi\tup{y} = \phi\tup{y} \phi\tup{x}$).
Thus, so are the left hand sides.
In other words, $\phi\tup{v} = \phi\tup{w}$.
This completes our proof 
of Lemma \ref{lem.bal.weyl}.
\end{proof}
\end{verlong}

\section{Some words on words} \label{sec:words}

Next, we take a closer look at some properties of the monoid $\M$ of words.
We introduce some more terminology:

\begin{itemize}
\item A word $w\in\M$ is said to be \emph{rising} if it has at least
as many $U$'s as it has $D$'s.

\item A word $w\in\M$ is said to be \emph{falling} if it has at least
as many $D$'s as it has $U$'s.
\end{itemize}

Thus, each word $w\in\M$ is rising or falling or both. Moreover, the
balanced words $w\in\M$ are exactly the words $w\in\M$ that
are both rising and falling.

\begin{itemize}

\item A \emph{down-zig} means a word of the form $UD^{k}U$ for some $k\geq2$.

\item A rising word $w\in\M$ is said to be \emph{up-normal} if it
contains no down-zig as a factor.

\end{itemize}

For example, the rising word $UUDUDUDD$ is up-normal, whereas the rising word
$UUDDUU$ is not (since it has the down-zig $UDDU=UD^{2}U$ as a factor).

We could similarly define ``up-zigs'' and ``down-normal words'' (by
toggling each letter in down-zigs and up-normal words, respectively), but we
will have no need for them.
\medskip

For what follows, we need a simple property of products in a monoid:

\begin{lemma}
\label{lem.arb}
Let $a$ and $b$ be two elements of a monoid $M$.
Let $w \in M$ be any product of $a$'s and $b$'s ending with a $b$.
Then, $w$ can be written in the form $a^{r_1}b\, a^{r_2}b\, \cdots a^{r_h}b$
for some nonnegative integers $h$ and $r_1, r_2, \ldots, r_h$. (Note that these integers are allowed to be $0$.)
\end{lemma}

\begin{vershort}
\begin{proof}
Almost immediate. (See the detailed version of this paper.)
\end{proof}
\end{vershort}

\begin{verlong}
\begin{proof}
We assumed that $w$ is a product of $a$'s and $b$'s ending with a $b$.
Locate all the $b$ factors in this product. Between any two consecutive $b$ factors lies a (possibly empty) product of $a$'s.
Rewrite this product as $a^r$ for some nonnegative integer $r$. Do the same for the product of $a$'s that lies to the left of the first $b$.
Thus, the total product becomes $a^{r_1}b\, a^{r_2}b\, \cdots a^{r_h}b$, where $h$ is the number of all $b$ factors and where the $r_i$ are the exponents $r$ obtained from the rewriting process.
This proves Lemma~\ref{lem.arb}.
\end{proof}
\end{verlong}

\begin{proposition}
\label{rem.shape_normal}
Every up-normal word has the form
\[
D^a\, (UD)^{r_1}U\, (UD)^{r_2}U\, \cdots\, (UD)^{r_h}U\, D^b
\]
for some nonnegative integers $a,b,h$ and $r_1,r_2,\ldots,r_h$
(that is, a power of $D$, followed by a product of several factors of the form $(UD)^r U$, followed by a further power of $D$, where all powers are allowed to be empty).
\end{proposition}

\begin{proof}
If the word consists entirely of $D$'s, then this is trivial.
Otherwise, we first remove the initial and the final run of $D$'s (of length $a$ and $b$ respectively, both of which can also be $0$) from our word. The remaining word is still up-normal, but starts and ends with a $U$.

This remaining word therefore has no two consecutive $D$s, since any run of $D$s longer than a single $D$ would create a down-zig factor (when combined with the last $U$ before the run and the first $U$ after it).
Thus, each $D$ in this remaining word has to be preceded
by a $U$ (since the word starts with $U$).
This allows us to decompose this word into a product of $UD$'s and $U$'s (for instance, by reading it from right to left, and pairing each $D$ with the $U$ that necessarily precedes it); this product ends with a $U$ (since our word ends with a $U$).
Thus, our word can be written as
\[
(UD)^{r_1}U\, (UD)^{r_2}U\, \cdots\, (UD)^{r_h}U
\]
for some nonnegative integers $h$ and $r_1, r_2, \ldots, r_h$ (by Lemma~\ref{lem.arb}, applied to $a = UD$ and $b = U$).
This proves Proposition~\ref{rem.shape_normal}.
\end{proof}

The converse of Proposition~\ref{rem.shape_normal} also holds: Each rising word of the form shown in Proposition~\ref{rem.shape_normal} is up-normal. The proof is nearly trivial, but we will not use this fact in the following.
\medskip

Next, we observe a near-trivial symmetry of balanced commutations:

\begin{proposition}
    \label{prop.bal-symm}
    Let $u$ and $v$ be two words in $\M$.
    Then, $u \balsim v$ if and only if $\omega\tup{u} \balsim \omega\tup{v}$.
\end{proposition}

\begin{proof}
    The map $\omega$ transforms a word by reversing it and toggling each letter\footnote{We recall: To \emph{toggle} a letter means to replace it by the opposite letter (i.e., to replace a $U$ by a $D$ or a $D$ by a $U$).}.
    Clearly, both of these operations turn balanced factors of our word into balanced factors of the resulting word.
    Thus, if a word $a$ is obtained from a word $b$ by a balanced commutation, then $\omega\tup{a}$ is obtained from $\omega\tup{b}$ by a balanced commutation as well.
    The same must therefore hold for multiple balanced commutations applied in sequence.
    In other words, if $u \balsim v$, then $\omega\tup{u} \balsim \omega\tup{v}$.
    The converse holds for similar reasons (or can also be obtained by applying the preceding sentence to $\omega\tup{u}$ and $\omega\tup{v}$ instead of $u$ and $v$, since $\omega\circ\omega = \id$).
    Thus, Proposition~\ref{prop.bal-symm} is proved.
\end{proof}

Our main goal in this section is to prove the following proposition:

\begin{proposition}
\label{prop.upnorm.nf}
Let $w\in\M$ be a rising word. Then, there
exists a unique up-normal word $t\in\M$ such that
$t \balsim w$.
\end{proposition}

This up-normal word $t$ will be called the \emph{up-normal form} of $w$.

In order to prove Proposition \ref{prop.upnorm.nf}, we need some lemmas.
First, we show some simple identities that allow us to convert between the three height polynomials (height, NE-height and SE-height) of a diagonal path:

\begin{lemma}\label{lem:laurent}
Let $\pp$ be any diagonal path starting at height $a$ and ending at height $b$. We have
\begin{align}
    H(\pp,z) = (1+z)H_{\NE}(\pp,z) + \sum_{j \geq b} z^j - \sum_{j \geq a+1} z^j
    \label{eq.lem.laurent.NE}
\end{align}
and
\begin{align}
    H(\pp,z) = (1+z^{-1})H_{\SE}(\pp,z) + \sum_{j \geq a} z^j - \sum_{j \geq b+1} z^j.
    \label{eq.lem.laurent.SE}
\end{align}
(The infinite sums are formal Laurent series, but only finitely many addends survive the cancellation.)
\end{lemma}

\begin{proof}
We only prove \eqref{eq.lem.laurent.NE}, since the proof of \eqref{eq.lem.laurent.SE} is completely analogous.
If the path $\pp$ has length $0$, then $a = b$ as well as $H(\pp,z) = z^a = z^b$ and $H_{\NE}(\pp,z) = H_{\SE}(\pp,z) = 0$.
The identity clearly holds in this case.
Now proceed by induction, and let $\pp'$ be the diagonal path obtained by removing the last vertex from $\pp$.
Then $H(\pp,z) = H(\pp',z) + z^b$, and the final height of $\pp'$ is $b-1$ if the last step of $\pp$ is an NE-step, and $b+1$ otherwise.
In the former case, we have $H_{\NE}(\pp,z) = H_{\NE}(\pp',z) + z^{b-1}$, and the induction hypothesis gives us
\[
H(\pp',z) = (1+z)H_{\NE}(\pp',z) + \sum_{j \geq b-1} z^j - \sum_{j \geq a+1} z^j,
\]
from which the desired statement follows by adding $z^b$ on both sides (since $z^{b-1} + z^b = \tup{1+z} z^{b-1}$).
In the latter case, we have $H_{\NE}(\pp,z) = H_{\NE}(\pp',z)$, and the induction hypothesis gives us
\[
H(\pp',z) = (1+z)H_{\NE}(\pp',z) + \sum_{j \geq b+1} z^j - \sum_{j \geq a+1} z^j.
\]
Again, we add $z^b$ on both sides to obtain the desired identity. This completes the induction and thus the proof.
\end{proof}

\begin{remark}
Let $\pp$ and $b$ be as in Lemma~\ref{lem:laurent}.
By definition, we have
\[H(\pp,z) = H_{\NE}(\pp,z) + H_{\SE}(\pp,z) + z^b,\]
since $H_{\NE}$ covers all NE-steps, $H_{\SE}$ covers all SE-steps, and $z^b$ the final vertex. This can also be used to derive~\eqref{eq.lem.laurent.SE} from~\eqref{eq.lem.laurent.NE}.
\end{remark}

Recall that the height polynomial $H\tup{w, z}$ of a word $w$ was defined as the height polynomial $H\tup{\pp, z}$ of its standard path $\pp$.
Thus, Lemma~\ref{lem:laurent} can be applied to words:

\begin{lemma}\label{lem:laurent-w}
Let $w \in \M$ be a word with final height $b$. Then,
\begin{align}
    H(w,z) = (1+z)H_{\NE}(w,z) + \sum_{j \geq b} z^j - \sum_{j \geq 1} z^j
    \label{eq.lem.laurent-w.NE}
\end{align}
and
\begin{align}
    H(w,z) = (1+z^{-1})H_{\SE}(w,z) + \sum_{j \geq 0} z^j - \sum_{j \geq b+1} z^j.
    \label{eq.lem.laurent-w.SE}
\end{align}
\end{lemma}

\begin{proof}
Let $\pp$ be the standard path of $w$.
Then, the initial height of $\pp$ is $0$ (since $\pp$ starts at $\tup{0,0}$), whereas the final height of $\pp$ is the final height of $w$ (by the definition of the latter), and we have
\[
H(w,z) = H(\pp,z)
\quad \text{ and } \quad
H_{\NE}(w,z) = H_{\NE}(\pp,z)
\quad \text{ and } \quad
H_{\SE}(w,z) = H_{\SE}(\pp,z)
\]
(again by the definitions of the respective left hand sides).
Thus, Lemma~\ref{lem:laurent-w} follows from Lemma~\ref{lem:laurent} (applied to $a = 0$).
\end{proof}

As we said, height polynomials of words are a particular case of height polynomials of diagonal paths.
But the general case can easily be reduced to this particular case:

\begin{lemma}
\label{lem.H.shift}
Let $\rr$ be a diagonal path. Let $j$ be its initial
height, and let $w=\w\left(  \rr\right)  $ be its
reading word. Then, $H\left(  \rr,z\right)  =z^{j}H\left(  w,z\right)
$.
\end{lemma}

\begin{vershort}
\begin{proof}
Let $\pp$ be the standard path of $w$.
Then, $\rr$ is the image of $\pp$ under a parallel translation with vertical component $j$.
Thus, $H\tup{\rr, z} = z^j H\tup{\pp, z}$.
Since $H\tup{w, z}$ is defined as $H\tup{\pp, z}$, we can rewrite this as $H\tup{\rr, z} = z^j H\tup{w, z}$.
\end{proof}
\end{vershort}

\begin{verlong}
\begin{proof}
Let $\pp=\left(  p_{0},p_{1},\ldots,p_{k}\right)  $ be the standard path of $w$. Then, $H\left(  w,z\right)
=H\left(  \pp,z\right)  $ (by the definition of $H\left(  w,z\right)  $).

However, the path $\pp$ has reading word $w$ (by the definition of a standard path). Thus,
the path $\rr$ has the same reading word as $\pp$
(since both paths have reading word $w$). Thus, it is an image of $\pp$
under the parallel translation by some vector $\left(  a,b\right)  $. Consider
this $\left(  a,b\right)  $. Clearly, $b$ must be the difference between the
initial height of $\rr$ and the initial height of $\pp$. Since
the initial height of $\rr$ is $j$, while the initial height of
$\pp$ is $0$ (since any standard path has initial height $0$), we thus obtain $b=j-0=j$.

On the other hand, $\rr$ is the image of the path $\pp=\left(
p_{0},p_{1},\ldots,p_{k}\right)  $ under the translation by the vector
$\left(  a,b\right)  $. Hence, $\rr=\left(  r_{0},r_{1},\ldots
,r_{k}\right)  $, where each $r_{i}$ is the image of the corresponding $p_{i}$
under this translation. Thus, for each $i\in\left\{  0,1,\ldots,k\right\}  $,
we have
\begin{equation}
\htt\left(  r_{i}\right)  =\htt\left(
p_{i}\right)  +b.
\label{pf.lem.H.shift.2}
\end{equation}

Now, the definition of $H\left(  \rr,z\right)  $ yields%
\begin{align*}
H(\rr,z)  & =\sum_{i=0}^{k}z^{\htt\left(  r_{i}\right)
}=\sum_{i=0}^{k}\underbrace{z^{\htt\left(  p_{i}\right)  +b}%
}_{=z^{b}z^{\htt\left(  p_{i}\right)  }}%
\ \ \ \ \ \ \ \ \ \ \left(  \text{by (\ref{pf.lem.H.shift.2})}\right)  \\
& =\underbrace{z^{b}}_{\substack{=z^{j}\\\text{(since }b=j\text{)}%
}}\ \ \underbrace{\sum_{i=0}^{k}z^{\htt\left(  p_{i}\right)  }%
}_{=H\left(  \pp,z\right)  =H\left(  w,z\right)  }=z^{j}H\left(
w,z\right)  .
\end{align*}
This proves Lemma \ref{lem.H.shift}.
\end{proof}
\end{verlong}

\begin{lemma} \label{lem.H.prod}
Let $u$ and $v$ be two words. Let
\[
k = \tup{\text{\# of $U$'s in $u$}} - \tup{\text{\# of $D$'s in $u$}}.
\]
Then,
\begin{align}
    H\tup{uv, z} = H\tup{u,z} + z^k \tup{H\tup{v,z}-1}; \label{eq.lem.H.prod.H} \\
    H_{\NE}\tup{uv, z} = H_{\NE}\tup{u,z} + z^k H_{\NE}\tup{v,z}; \label{eq.lem.H.prod.HNE} \\
    H_{\SE}\tup{uv, z} = H_{\SE}\tup{u,z} + z^k H_{\SE}\tup{v,z}.
    \label{eq.lem.H.prod.HSE}
\end{align}
\end{lemma}

\begin{vershort}
\begin{proof}
The standard path of $uv$ can be obtained by splicing the standard path of $u$ together with a translated copy of the standard path of $v$.
The translation increases the heights of all vertices by $k$ (since the standard path of $u$ ends at height $k$).
Thus, all three equalities follow.
(See the detailed version for more details.)
\end{proof}
\end{vershort}

\begin{verlong}
\begin{proof}
    Let $\pp$, $\qq$ and $\rr$ be the standard paths of the words $u$, $v$ and $uv$, respectively.
    Then, the path $\pp$ has initial height $0$ (like any standard path) and reading word $\w\tup{\pp} = u$. Hence, the path
    $\pp$ has final height $k$ (by the definition of $k$, and by \eqref{eq.U-D}).
    Let $\tup{\ell, k}$ be the final vertex of $\pp$.
    
    The reading words of the paths $\pp$, $\qq$ and $\rr$ are $u$, $v$ and $uv$ (since $\pp$, $\qq$ and $\rr$ are the standard paths of $u$, $v$ and $uv$).
    In particular,
    the reading word of the path $\rr$ is $uv$, which is the concatenation of the reading words of $\pp$ and $\qq$. Thus, the path $\rr$ is obtained by gluing together $\pp$ with a copy of $\qq$ that has been translated by the vector $\tup{\ell, k}$
    (since the final vertex of $\pp$ is $\tup{\ell, k}$).
    Hence, the vertices of $\rr$ are the vertices of $\pp$ as well as the non-initial\footnote{``Non-initial'' means that we exclude the initial vertex.} vertices of $\qq$ translated by $\tup{\ell, k}$.
    Clearly, the translation increases the height of each of the latter vertices by $k$.
    Thus, the sum that defines $H\tup{\rr, z}$ can be split into two sub-sums, one of which collects all the vertices of $\pp$ and thus amounts to $H\tup{\pp, z}$, while the other collects all the non-initial vertices of $\qq$ (translated by $\tup{\ell, k}$) and thus amounts to $z^k \tup{H\tup{\qq, z} - 1}$ (the $z^k$ factor comes from the translation, whereas the ``$-1$'' stems from removing the initial vertex). Thus, we obtain
    \begin{align}
    H\tup{\rr, z}
    = H\tup{\pp, z} + z^k \tup{H\tup{\qq, z} - 1} .
    \label{pf.lem.H.prod.4}
    \end{align}
    However, $\pp$, $\qq$ and $\rr$ are the standard paths of the words $u$, $v$ and $uv$. Hence, $H\tup{\pp, z} = H\tup{u, z}$ and $H\tup{\qq, z} = H\tup{v, z}$ and $H\tup{\rr, z} = H\tup{uv, z}$.
    Thus, we can rewrite \eqref{pf.lem.H.prod.4} as
    \[
    H\tup{uv, z} = H\tup{u,z} + z^k \tup{H\tup{v,z}-1}.
    \]
    This proves \eqref{eq.lem.H.prod.H}.
    The other two equalities can be proved in similar ways (but we no longer need to remove the initial vertex of $\qq$, since the last vertex of a diagonal path counts neither as an NE-step nor as an SE-step).
\end{proof}
\end{verlong}

\begin{lemma} \label{lem:unique}
An up-normal word $w$ is uniquely determined by its final height (i.e., the difference \# of $U$'s $-$ \# of $D$'s) and the height polynomial $H(w,z)$.
\end{lemma}

\begin{proof}
Let $f$ be the final height of $w$. We shall recover $w$ from $f$ and $H\tup{w,z}$.

From \eqref{eq.lem.laurent-w.NE} (applied to $b=f$), we obtain
\[
H(w,z) = (1+z)H_{\NE}(w,z) + \sum_{j \geq f} z^j - \sum_{j \geq 1} z^j .
\]
All terms in this equality other than $H_{\NE}(w,z)$ are determined by $H\tup{w,z}$ and $f$. 
Thus, we can use this equality to determine $H_{\NE}(w,z)$ from $H\tup{w,z}$ and $f$ (using polynomial division by $1+z$).

Now, let $\pp$ be the standard path of $w$. Thus, the path $\pp$ starts at $\tup{0,0}$ and has reading word $\w\tup{\pp} = w$.
Moreover, the definition of $H_{\NE}(w,z)$ yields $H_{\NE}(w,z) = H_{\NE}(\pp,z)$.
Thus, $H_{\NE}(\pp,z)$ can be determined from $H\tup{w,z}$ and $f$ (since we can determine $H_{\NE}(w,z)$ from these inputs).

By Proposition~\ref{rem.shape_normal}, we can write $w$ in the form
\begin{align}
w = D^a\, (UD)^{r_1}U\, (UD)^{r_2}U\,  \cdots\,  (UD)^{r_h}U\, D^b .
\label{eq.lem.unique.w=}
\end{align}
Each $U$ here corresponds to an NE-step of the diagonal path $\pp$ (since $w = \w\tup{\pp}$).
Thus, $\pp$ has $r_1+1$ NE-steps of height $-a$ (corresponding to the $r_1+1$ many $U$'s in the $(UD)^{r_1} U$ factor),
followed by $r_2+1$ NE-steps of height $-a+1$ (corresponding to the $r_2+1$ many $U$'s in the $(UD)^{r_2} U$ factor),
and so on.
Altogether, we thus obtain
\[
H_{\NE}\tup{\pp,z} = \sum_{i=1}^h (r_i+1)z^{-a+i-1}.
\]
\begin{verlong}
Hence, the Laurent polynomial $H_{\NE}\tup{\pp,z}$ has the nonzero coefficients $r_1+1, r_2+1, \ldots, r_h+1$ in front of the respective monomials $z^{-a}, z^{-a+1}, \ldots, z^{-a+h-1}$, and zero coefficients in front of all other monomials.
In particular, the number $-a$ is the smallest exponent that appears with nonzero coefficient in the Laurent polynomial $H_{\NE}\tup{\pp,z}$, whereas $h$ is the total number of nonzero coefficients of this Laurent polynomial, and furthermore, the numbers $r_1+1, r_2+1, \ldots, r_h+1$ are the coefficients of the monomials $z^{-a}, z^{-a+1}, \ldots, z^{-a+h-1}$ in this Laurent polynomials.
\end{verlong}

Thus, the numbers $a$, $h$, and $r_1,r_2,\ldots,r_h$ can be determined from $H_{\NE}(\pp,z)$.
Since $H_{\NE}(\pp,z)$ can be determined from $H\tup{w,z}$ and $f$,
we thus conclude that the numbers $a$, $h$, and $r_1,r_2,\ldots,r_h$ can be determined from $H\tup{w,z}$ and $f$.
Knowing these numbers, we can now determine $b$ from $f$ (since $f$ is the final height of $w$, that is, the \# of $U$'s in $w$ minus the \# of $D$'s in $w$).
Knowing $a, b, h, r_1, r_2, \ldots, r_h$, we can now reconstruct $w$ using \eqref{eq.lem.unique.w=}.
\end{proof}

\begin{remark}
The statement of Lemma~\ref{lem:unique} would be false without the assumption that the final height is known. For example, the up-normal words $UDUU$ and $UUDD$ both have the height polynomial $2+2z+z^2$.
\end{remark}

An analogue of Lemma~\ref{lem:unique} is true for down-normal words (with a similar proof).

\begin{lemma} \label{lem:invariant}
The height polynomial of a word is invariant under balanced commutations.
In other words: If two words $v$ and $w$ satisfy $v \balsim w$, then $H(v,z) = H(w,z)$.
\end{lemma}

\begin{proof}
It suffices to prove that $H(v, z) = H(w, z)$ whenever $v$ can be obtained from $w$ by a single balanced commutation.
(The general case will then follow by induction.)

So let $v$ be obtained from $w$ by a single balanced commutation.
Thus $v=pxyq$ and $w=pyxq$, where $p,q\in\M$ are two words and where $x,y\in\M$ are two balanced words.
Consider these words $p,q,x,y$.
Let $a$ be the final height of $p$. Since $x$ and $y$ are balanced words, their final heights are $0$, and thus the final heights of the four words $px$, $py$, $pxy$ and $pyx$ equal $a$ as well.
Thus, by repeated application of \eqref{eq.lem.H.prod.H}, we find
\begin{vershort}
\begin{align*}
    H(pxyq, z)
    &= H(pxy, z) + z^a (H(q,z) - 1)
	\\
    &= H(px, z) + z^a (H(y, z) - 1) + z^a (H(q,z) - 1)
	\\
    &= H(p,z) + z^a(H(x,z)-1) + z^a(H(y,z)-1) + z^a(H(q,z)-1)
\end{align*}
\end{vershort}
\begin{verlong}
\begin{align*}
    H(pxyq, z)
    &= H(pxy, z) + z^a (H(q,z) - 1)
	\ \ \ \ \ \ \ \ \ \ \left(\text{by \eqref{eq.lem.H.prod.H}}\right)
	\\
    &= H(px, z) + z^a (H(y, z) - 1) + z^a (H(q,z) - 1)
	\ \ \ \ \ \ \ \ \ \ \left(\text{by \eqref{eq.lem.H.prod.H}}\right)
	\\
    &= H(p,z) + z^a(H(x,z)-1) + z^a(H(y,z)-1) + z^a(H(q,z)-1)
\end{align*}
(by \eqref{eq.lem.H.prod.H})
\end{verlong}
and similarly
\[
H(pyxq,z) = H(p,z) + z^a(H(y,z)-1) + z^a(H(x,z)-1) + z^a(H(q,z)-1).
\]
The right hand sides of these two equalities are visibly equal.
Hence, so are their left hand sides: $H(pxyq,z) = H(pyxq,z)$.
In other words, $H(v,z) = H(w,z)$ (since $v=pxyq$ and $w=pyxq$), and our proof is complete.
\end{proof}

The above lemmas will be used in showing the uniqueness part of Proposition~\ref{prop.upnorm.nf}.
Let us now present some lemmas for the existence part.

\begin{vershort}
\begin{lemma}
\label{lem.balanced.0}
Let $w\in\M$ be a balanced word that starts
with a $U$ and ends with a $U$.
Then, we can write $w$ as a concatenation $w=pq$, where $p$ is a balanced word starting with a $U$, and where $q$ is a balanced word starting with a $D$.
\end{lemma}

\begin{proof}
This readily follows from the ``discrete intermediate value theorem''.
Indeed, the standard path of $w$ starts and ends on the x-axis (since $w$ is balanced), but it starts with an NE-step (since $w$ starts with a $U$) and ends with an NE-step as well (likewise).
Hence, it must cross the x-axis at some intermediate point, and the last such intermediate point must be followed by an SE-step (since the path must fall below the x-axis in order to return to it via an NE-step).
Splitting the path at this point, we obtain two smaller paths corresponding to the desired factors $p$ and $q$.
\end{proof}

\begin{lemma}
\label{lem.upnorm.1}Let $w\in\M$ be a rising word that is not
up-normal. Then, we can write $w$ in the form $w=upqv$, where $u$ and $v$ are
two words, where $p$ is a balanced word starting with a $U$, and where $q$ is
a balanced word starting with a $D$.
\end{lemma}

\begin{figure}[htbp]
\begin{center}
\begin{tikzpicture}[scale=0.7]
\draw[help lines] (-0.2, -3.2) grid (16.2, 2.2);
\draw [line width=2pt] (-0.2, 0) -- (16.2, 0);
\draw
[line width=2pt,-stealth] (6, 0) edge (7, -1) (7, -1) edge (8, -2) (8, -2) edge (9, -3) (9, -3) edge (10, -2) (10, -2) edge (11, -1) (11, -1) edge (12, 0) (12, 0) edge (13, 1) (13, 1) edge (14, 0) (14, 0) edge (15, -1) (15, -1) -- (16, 0);
\draw
[line width=2pt,-stealth, red] (1, 1) edge (2, 2) (2, 2) edge (3, 1) (3, 1) edge (4, 0) (4, 0) edge (5, -1) (5, -1) -- (6, 0);
\draw
[line width=2pt,-stealth, blue] (0, 0) -- (1, 1);
\node at (8, -4) {$u=\varnothing, \quad p=UUDD, \quad q=DU, \quad v=DDDUUUUDDU$};
\end{tikzpicture}
\vskip.1truein
\begin{tikzpicture}[scale=0.7]
\draw[help lines] (-0.2, -3.2) grid (16.2, 2.2);
\draw [line width=2pt] (-0.2, 0) -- (16.2, 0);
\draw
[line width=2pt,-stealth] (0, 0) edge (1, 1) (1, 1) edge (2, 2) (2, 2) edge (3, 1) (3, 1) edge (4, 0) (4, 0) edge (5, -1) (7, -1) edge (8, -2) (8, -2) edge (9, -3) (9, -3) edge (10, -2) (11, -1) edge (12, 0) (12, 0) edge (13, 1) (13, 1) edge (14, 0) (14, 0) edge (15, -1) (15, -1) -- (16, 0);
\draw
[line width=2pt,-stealth, red] (5, -1) edge (6, 0) (6, 0) edge (7, -1) (7, -1) edge (8, -2) (8, -2) edge (9, -3) (9, -3) -- (10, -2);
\draw
[line width=2pt,-stealth, blue] (10, -2) -- (11, -1);
\node at (8, -4) {$u=UUDDD, \quad p=UD, \quad q=DDUU, \quad v=UUDDU$};
\end{tikzpicture}
\vskip.1truein
\begin{tikzpicture}[scale=0.7]
\draw[help lines] (-0.2, -3.2) grid (16.2, 2.2);
\draw [line width=2pt] (-0.2, 0) -- (16.2, 0);
\draw
[line width=2pt,-stealth] (0, 0) edge (1, 1) (1, 1) edge (2, 2) (2, 2) edge (3, 1) (3, 1) edge (4, 0) (4, 0) edge (5, -1) (5, -1) edge (6, 0) (6, 0) edge (7, -1) (7, -1) edge (8, -2) (8, -2) edge (9, -3) (9, -3) edge (10, -2) (10, -2) edge (11, -1) (11, -1) -- (12, 0);
\draw
[line width=2pt,-stealth, red] (12, 0) edge (13, 1) (13, 1) edge (14, 0) (14, 0) edge (15, -1) (15, -1) -- (16, 0);
\node at (8, -4) {$u=UUDDDUDDDUUU, \quad p=UD, \quad q=DU, \quad v=\varnothing$};
\end{tikzpicture}
\end{center}
\caption{Possible decompositions of $w=upqv$ for different choices of the down-zig $w_i w_{i+1} \ldots w_j$. The (arbitrarily chosen) down-zig is indicated by red lines. The minimum-length copair-subpath $\qq$ contains all these red steps as well as some extra steps, which are indicated by blue lines.}
\label{fig.convert.standard}
\end{figure}

\begin{proof}
Our word $w = w_1 w_2 \ldots w_\ell$ (like any word in $\M$) corresponds to a unique standard path $\pp = \tup{p_0, p_1, \dots, p_\ell}$ that starts at $p_0=\tup{0, 0}$ and has reading word $\w\tup{\pp} = w$. Since $w$ is not up-normal, it contains at least one down-zig as a
factor. We arbitrarily pick one such down-zig, which we denote by $w_{i}w_{i+1}\cdots w_{j}$.
By the definition of a down-zig, we have $w_{i}=U$ and
$w_{i+1}=w_{i+2}=\cdots=w_{j-1}=D$ and $w_{j}=U$ and
$i<j-2$ and $p_{i}\geq p_{j+1}$ (since a
down-zig must always have length $\geq4$).
This shows, in particular, that the subpath $\tup{p_i, p_{i+1}, \ldots, p_{j+1}}$ of $\pp$ is falling.
See Figure \ref{fig.convert.standard} for an illustration.

A \emph{copair-subpath} is a rising%
\footnote{We call a path \emph{rising} if its reading word is rising (i.e., if the path ends at the same height as it starts or higher).
Likewise we define \emph{balanced} and \emph{falling} paths.}
subpath of $\pp$ that starts at $p_{i}$ or earlier and ends at $p_{j+1}$ or later. Note that a copair-subpath always exists, since $w$ itself is rising.
We pick a copair-subpath $\qq$ of minimum length.

Then, by minimality, $\qq$ cannot be shortened without breaking either the ``starts at $p_{i}$ or earlier'' condition or the ``ends at $p_{j+1}$ or later'' condition or the ``rising'' condition.
Hence, if $\qq$ starts before $p_i$, then $\qq$ must be balanced (since we could otherwise remove the first step from $\qq$ without breaking any of the conditions).
Likewise, if $\qq$ ends after $p_{j+1}$, then $\qq$ must be balanced (since we could otherwise remove the last step from $\qq$).
In the remaining case (i.e., if $\qq$ starts at $p_i$ and ends at $p_{j+1}$) it is also clear that $\qq$ must be balanced:
since $\qq$ is rising by definition, but the subpath $\tup{p_i, p_{i+1}, \ldots, p_{j+1}}$ of $\pp$ is falling, $\qq$ must be both rising and falling, i.e., balanced.
Thus, we have shown that $\qq$ is always balanced.

Next, let us show that our minimum-length copair-subpath $\qq$ must start with an NE-step. Indeed, if it started with an SE-step, then we could shorten it by removing this step and obtain an even shorter copair-subpath (it will still start at $p_i$ or earlier, since the step from $p_{i}$ to $p_{i+1}$ is an NE-step and thus not the first step of $\qq$); but this would contradict the minimum-length property of $\qq$.
Thus, $\qq$ must start with an NE-step.
For similar reasons, $\qq$ must end with an NE-step.

Now, consider the factor $w'$ of $w$ corresponding to the subpath $\qq$ of $\pp$.
This factor $w'$ is balanced (since $\qq$ is balanced) and starts and ends with a $U$ (since our subpath $\qq$ starts with and ends with NE-steps). Thus we can apply Lemma \ref{lem.balanced.0} to this factor $w'$, and factor it as $w' = pq$, where $p$ is a balanced word starting with a $U$, and where $q$ is a balanced word starting with a $D$.
Hence, the full word $w$ factors as $w = upqv$, where $u$ is the prefix of $w$ coming before this factor $w'$, whereas $v$ is the suffix of $w$ coming after $w'$.
\end{proof}

\begin{remark}
The decomposition of the word $w$ in Lemma \ref{lem.upnorm.1} may not be unique. See Figure \ref{fig.convert.standard}.
\end{remark}
\end{vershort}

\begin{verlong}
\begin{lemma}
\label{lem.balanced.0}
Let $w\in\M$ be a balanced word that starts
with a $U$ and ends with a $U$.
Then, we can write $w$ as a concatenation $w=pq$, where $p$ is a balanced word starting with a $U$, and where $q$ is a balanced word starting with a $D$.
\end{lemma}

\begin{proof}
[Proof of Lemma \ref{lem.balanced.0}.] Write $w$ as $w=w_{1}w_{2}\cdots
w_{\ell}$, where $w_{1},w_{2},\ldots,w_{\ell}\in\left\{  U,D\right\}  $ are
the letters of $w$. For each $k\in\left\{  0,1,\ldots,\ell\right\}  $, define
$w_{:k}:=w_{1}w_{2}\cdots w_{k}$ to be the factor of $w$ consisting of the
first $k$ letters of $w$, and define the number%
\[
h_{k}:=\left(  \text{\# of }U\text{'s in }w_{:k}\right)  -\left(  \text{\# of
}D\text{'s in }w_{:k}\right)  .
\]
(Note that if $w$ is the reading word $\w\left(  \pp\right)$ of a
diagonal path $\pp = \tup{p_0, p_1, \ldots, p_\ell}$ that starts on the x-axis, then $h_{k}$ is the
height of $p_k$. Thus the notation $h_{k}$.
Note also that $h_0 = 0$, since $w_{:0}$ is an empty word.)

It is clear that each $k\in\left\{  1,2,\ldots,\ell\right\}  $ satisfies%
\begin{equation}
h_{k}-h_{k-1}=%
\begin{cases}
1, & \text{if }w_{k}=U;\\
-1, & \text{if }w_{k}=D.
\end{cases}
\label{pf.lem.upnorm.1.h-h}%
\end{equation}
More generally, for any factor $w_{i}w_{i+1}\cdots w_{j}$ of $w$, we have%
\begin{equation}
h_{j}-h_{i-1}=\left(  \text{\# of }U\text{'s in }w_{i}w_{i+1}\cdots
w_{j}\right)  -\left(  \text{\# of }D\text{'s in }w_{i}w_{i+1}\cdots
w_{j}\right)  .\label{pf.lem.upnorm.1.h-hgen}%
\end{equation}
In particular, a factor $w_{i}w_{i+1}\cdots w_{j}$ of $w$ is balanced if and
only if $h_{i-1}=h_{j}$. Thus, $h_{0}=h_{\ell}$ (since $w=w_{1}w_{2}\cdots
w_{\ell}$ is balanced).

The word $w$ ends with a $U$. In other words, $w_{\ell}=U$. Hence,
(\ref{pf.lem.upnorm.1.h-h}) shows that $h_{\ell}-h_{\ell-1}=1$, so that
$h_\ell - 1 = h_{\ell-1}$.

Pick the largest $c\in\left\{  0,1,\ldots,\ell-1\right\}  $ such that
$h_{c}\geq h_{0}$. (Such a $c$ exists, since $h_{c}\geq h_{0}$ holds
for $c=0$.) Then, $c\neq\ell-1$ (since $h_c \geq h_0 = h_\ell > h_\ell - 1 = h_{\ell-1}$ and thus $h_c \neq h_{\ell - 1}$), so that
$c\in\left\{  0,1,\ldots,\ell-2\right\}  $ and therefore $c+1\in\left\{
0,1,\ldots,\ell-1\right\}  $. Therefore, we cannot have $h_{c+1}\geq h_{0}$
(since this would contradict the maximality of $c$). In other words, we have
$h_{c+1}<h_{0}$. Therefore, $h_{c+1}\leq h_{0}-1$ (since $h_{c+1}$ and $h_{0}$
are integers).
But (\ref{pf.lem.upnorm.1.h-h}) (applied to $k = c+1$) shows that $h_{c+1}-h_{c}%
=\pm1\geq-1$, so that $h_{c+1}\geq h_{c}-1$ and therefore $h_{c}-1\leq
h_{c+1}\leq h_{0}-1$. Thus, $h_{c}\leq h_{0}$. Combined with $h_{c}\geq h_{0}%
$, this yields $h_{c}=h_{0}$. By (\ref{pf.lem.upnorm.1.h-hgen}), this shows
that the factor $w_{1}w_{2}\cdots w_{c}$ of $w$ is balanced. Moreover,
combining $h_{c}=h_{0}$ with $h_{0}=h_{\ell}$, we obtain $h_{c}=h_{\ell}$, and
this shows that the factor $w_{c+1}w_{c+2}\cdots w_{\ell}$ of $w$ is balanced
(again by (\ref{pf.lem.upnorm.1.h-hgen})).

If we had $w_{c+1}=U$, then we would have $h_{c+1}-h_{c}=1$ (by
(\ref{pf.lem.upnorm.1.h-h})), whence we would obtain $h_{c+1}=h_{c}%
+1>h_{c}=h_{0}$, which would contradict the fact that we cannot have
$h_{c+1}\geq h_{0}$. Thus, we cannot have $w_{c+1}=U$. Hence, $w_{c+1}=D$. In
other words, the word $w_{c+1}w_{c+2}\cdots w_{\ell}$ starts with a $D$.

But the word $w$ starts with a $U$. Thus, $w_{1}=U\neq D=w_{c+1}$. Therefore,
$1\neq c+1$, so that $c\neq0$. The word $w_{1}w_{2}\cdots w_{c}$ is thus
nonempty. Moreover, this word starts with a $U$ (since $w_{1}=U$).

Now, we have
\[
w=w_{1}w_{2}\cdots w_{\ell}=\underbrace{\left(  w_{1}w_{2}\cdots w_{c}\right)
}_{\substack{\text{a balanced word}\\\text{starting with a }U}%
}\underbrace{\left(  w_{c+1}w_{c+2}\cdots w_{\ell}\right)  }%
_{\substack{\text{a balanced word}\\\text{starting with a }D}}
\]

Hence, we can write $w$ as a concatenation $w=pq$, where $p$ is a balanced
word starting with a $U$, and where $q$ is a balanced word starting with a $D$
(namely, $p=w_{1}w_{2}\cdots w_{c}$ and $q=w_{c+1}w_{c+2}\cdots w_{\ell}$).
This proves Lemma \ref{lem.balanced.0}.
\end{proof}

\begin{lemma}
\label{lem.upnorm.1}Let $w\in\M$ be a rising word that is not
up-normal. Then, we can write $w$ in the form $w=upqv$, where $u$ and $v$ are
two words, where $p$ is a balanced word starting with a $U$, and where $q$ is
a balanced word starting with a $D$.
\end{lemma}

\begin{proof}
[Proof of Lemma \ref{lem.upnorm.1}.]Write $w$ as $w=w_{1}w_{2}\cdots w_{\ell}%
$, where $w_{1},w_{2},\ldots,w_{\ell}\in\left\{  U,D\right\}  $ are the
letters of $w$. For each $k\in\left\{  0,1,\ldots,\ell\right\}  $, define
$w_{:k}:=w_{1}w_{2}\cdots w_{k}$ to be the factor of $w$ consisting of the
first $k$ letters of $w$, and define the number%
\[
h_{k}:=\left(  \text{\# of }U\text{'s in }w_{:k}\right)  -\left(  \text{\# of
}D\text{'s in }w_{:k}\right)  .
\]
Clearly, the equalities (\ref{pf.lem.upnorm.1.h-h}) and
(\ref{pf.lem.upnorm.1.h-hgen}) hold, just as in the proof of Lemma
\ref{lem.balanced.0}. In particular, from (\ref{pf.lem.upnorm.1.h-hgen}), we
obtain $h_{0}\leq h_{\ell}$, since $w$ is rising.

Moreover, the word $w$ is not up-normal, so that $w$ contains a down-zig as a
factor. Let $w_{i}w_{i+1}\cdots w_{j}$ be this factor.
Thus, by the definition of a down-zig, we have $i<j-2$ (since a
down-zig must always have length $\geq4$) and $w_{i}=U$ and $w_{i+1}%
=w_{i+2}=\cdots=w_{j-1}=D$ and $w_{j}=U$.
From $w_{i}=U$, we obtain
$h_{i}-h_{i-1}=1$ (by (\ref{pf.lem.upnorm.1.h-h})). From $w_{j}=U$, we obtain
$h_{j}-h_{j-1}=1$ (by (\ref{pf.lem.upnorm.1.h-h})). Moreover, the factor
$w_{i}w_{i+1}\cdots w_{j}$ is a down-zig and thus contains at least as many
$D$'s as it contains $U$'s; therefore, $h_{j}-h_{i-1}\leq0$ (by
(\ref{pf.lem.upnorm.1.h-hgen})), so that $h_{i-1}\geq h_{j}$.

A \emph{copair} means a pair $\left(  a,b\right)  $ of elements of
$\left\{  0,1,\ldots,\ell\right\}  $ satisfying%
\[
a\leq i-1\text{ and }j\leq b\text{ and }h_{a}\leq h_{b}.
\]
The \emph{span} of a copair $\left(  a,b\right)  $ will mean the difference
$b-a$. Note that $\left(  0,\ell\right)  $ is a copair (since $h_{0}\leq
h_{\ell}$), so that there exists at least one copair.

Pick a copair $\left(  a,b\right)  $ with minimum span. Thus, $a,b\in\left\{
0,1,\ldots,\ell\right\}  $ and $a\leq i-1$ and $j\leq b$ and $h_{a}\leq h_{b}$.

If we had $w_{a+1}=D$, then we would have $a+1\leq i-1$ (since $w_{a+1}=D\neq
U=w_{i}$ would entail $a+1\neq i$, so that $a\neq i-1$, and therefore the
inequality $a\leq i-1$ could be improved to $a<i-1$, so that $a \leq \tup{i-1}-1$ and thus $a+1\leq i-1$)
and $h_{a+1}\leq h_{b}$ (since (\ref{pf.lem.upnorm.1.h-h}) would show that
$h_{a+1}-h_{a}=-1$ (because of $w_{a+1}=D$), thus $h_{a+1}=h_{a}-1<h_{a}\leq
h_{b}$). Therefore, $\left(  a+1,b\right)  $ would again be a copair. This
copair $\left(  a+1,b\right)  $ would have a smaller span than $\left(
a,b\right)  $ (since $b-\tup{a+1} < b-a$), which is impossible since $\left(  a,b\right)  $ was chosen to
have minimum span. Thus, we cannot have $w_{a+1}=D$. Hence, $w_{a+1}=U$.
Therefore, (\ref{pf.lem.upnorm.1.h-h}) yields $h_{a+1}-h_{a}=1$.

If we had $w_{b}=D$, then we would have $j\leq b-1$ (since $w_{b}=D\neq
U=w_{j}$ would entail $b\neq j$, so that the inequality $j\leq b$ could be
improved to $j<b$, and this would entail $j\leq b-1$) and $h_{a}\leq h_{b-1}$
(since (\ref{pf.lem.upnorm.1.h-h}) would show that $h_{b}-h_{b-1}=-1$ (because
of $w_{b}=D$), thus $h_b = h_{b-1} - 1 < h_{b-1}$ and therefore $h_{a}\leq
h_{b}<h_{b-1}$). Therefore, $\left(  a,b-1\right)  $ would again be a copair.
This copair $\left(  a,b-1\right)  $ would have a smaller span than $\left(
a,b\right)  $ (since $\tup{b-1}-a < b-a$), which is impossible since $\left(  a,b\right)  $ was chosen to
have minimum span. Thus, we cannot have $w_{b}=D$. Hence, $w_{b}=U$.
Therefore, (\ref{pf.lem.upnorm.1.h-h}) yields $h_{b}-h_{b-1}=1$.

Now, we show that $h_{a}=h_{b}$. To prove this, we assume the contrary.
Thus, $h_{a}\neq h_{b}$, so that $h_{a}<h_{b}$ (since $h_{a}\leq h_{b}$).
Therefore, $h_{a}\leq h_{b}-1$ (since $h_{a}$ and $h_{b}$ are integers). If we
had $a=i-1$ and $b=j$, then we could rewrite $h_{a}<h_{b}$ as $h_{i-1}<h_{j}$,
which would contradict $h_{i-1}\geq h_{j}$. Thus, we cannot have $a=i-1$ and
$b=j$. Hence, we are in one (or both) of the following two cases:

\textit{Case 1:} We have $a\neq i-1$.

\textit{Case 2:} We have $b\neq j$.

Let us first consider Case 1. In this case, we have $a\neq i-1$. Hence,
$a<i-1$ (since $a\leq i-1$). Therefore, $a+1\leq i-1$. Moreover,
(\ref{pf.lem.upnorm.1.h-h}) yields $h_{a+1}-h_{a}\leq1$, so that $h_{a+1}\leq
h_{a}+1\leq h_{b}$ (since $h_{a}\leq h_{b}-1$). Consequently, $\left(
a+1,b\right)  $ is a copair (since $a+1\leq i-1$ and $j\leq b$). This copair
$\left(  a+1,b\right)  $ has smaller span than $\left(  a,b\right)  $, but
this contradicts the fact that $\left(  a,b\right)  $ was chosen to have
minimum span. Thus, we have found a contradiction in Case 1.

Let us next consider Case 2. In this case, we have $b\neq j$. Hence, $j<b$
(since $j\leq b$). Therefore, $j\leq b-1$. Moreover,
(\ref{pf.lem.upnorm.1.h-h}) yields $h_{b}-h_{b-1}\leq1$, so that $h_{b}-1\leq
h_{b-1}$. Now, $h_{a}\leq h_{b}-1\leq h_{b-1}$. Consequently, $\left(
a,b-1\right)  $ is a copair (since $a\leq i-1$ and $j\leq b-1$). This copair
$\left(  a,b-1\right)  $ has smaller span than $\left(  a,b\right)  $, but
this contradicts the fact that $\left(  a,b\right)  $ was chosen to have
minimum span. Thus, we have found a contradiction in Case 2.

We have now found contradictions in both Cases 1 and 2. Hence, our assumption
was false, and $h_{a}=h_{b}$ is proved.

Because of (\ref{pf.lem.upnorm.1.h-hgen}), this equality $h_{a}=h_{b}$ shows
that the word $w_{a+1}w_{a+2}\cdots w_{b}$ is balanced.
This balanced word is
furthermore nonempty (since $a \leq i-1 < i < j-2 < j \leq b$)
and starts with a $U$ (since $w_{a+1}=U$) and ends with a $U$ (since
$w_{b}=U$). Thus, by Lemma \ref{lem.balanced.0} (applied to $w_{a+1}%
w_{a+2}\cdots w_{b}$ instead of $w$), we can write $w_{a+1}w_{a+2}\cdots
w_{b}$ as a concatenation $w_{a+1}w_{a+2}\cdots w_{b}=pq$, where $p$ is a
balanced word starting with a $U$, and where $q$ is a balanced word starting
with a $D$. Consider these $p$ and $q$.

Let us furthermore set $u:=w_{1}w_{2}\cdots w_{a}$ and $v:=w_{b+1}%
w_{b+2}\cdots w_{\ell}$. Then,%
\[
w = w_1 w_2 \cdots w_\ell = \underbrace{\left(  w_{1}w_{2}\cdots w_{a}\right)  }_{=u}\underbrace{\left(
w_{a+1}w_{a+2}\cdots w_{b}\right)  }_{=pq}\underbrace{\left(  w_{b+1}%
w_{b+2}\cdots w_{\ell}\right)  }_{=v}=upqv.
\]
Hence, we have written $w$ in the form $w=upqv$, where $u$ and $v$ are two
words, where $p$ is a balanced word starting with a $U$, and where $q$ is a
balanced word starting with a $D$. This proves Lemma \ref{lem.upnorm.1}.
\end{proof}
\end{verlong}

\begin{proof}
[Proof of Proposition \ref{prop.upnorm.nf}.]
We equip the set $\M$ with the lexicographic order, where $D<U$.
If the word $w$ is not yet up-normal, then by Lemma~\ref{lem.upnorm.1} we can write it as $w = upqv$, where $p$ and $q$ are balanced, $p$ starts with $U$, and $q$ starts with $D$.
We then perform a balanced commutation to obtain the word $w' = uqpv$, which is lexicographically smaller than $w$ (since the first letter of $p$, which was $U$, has been replaced by the first letter of $q$, which is $D$).
We have $w' \balsim w$ by construction.
This procedure can be iterated until we end up with an up-normal word $t$ that satisfies $t \balsim w$. Indeed, the procedure cannot go on forever, since each balanced commutation makes our word lexicographically smaller while preserving its length.
Moreover, the word remains rising throughout this procedure, since a balanced commutation does not change the total numbers of $U$'s and $D$'s in the word.

Thus, we have proved the existence of an up-normal word $t \in \M$ such that $t \balsim w$.
It remains to prove its uniqueness.

The condition $t \balsim w$ ensures that the words $t$ and $w$ have the same final height (since balanced commutations do not change the numbers of $U$'s and $D$'s, and thus -- by \eqref{eq.U-D0} -- leave the final height unchanged as well) and the same height polynomial (since Lemma~\ref{lem:invariant} shows that balanced commutations do not change the height polynomial).
By Lemma~\ref{lem:unique}, the up-normal word $t$ is thus uniquely determined.
\end{proof}

\section{Proofs of the main results}\label{sec:mainproofs}

\begin{vershort}
\subsection{Proof of Theorem \ref{thm.phiomega}}
We first establish a lemma that combines some results of the previous sections.

\begin{lemma}
\label{lem.4to5}
Let $\pp$ and $\qq$ be two diagonal paths with
the same initial height and the same final height. Assume that $H\left(
\pp,z\right)  =H\left(  \qq,z\right)  $. Then,
$\w\tup{\pp} \balsim \w\tup{\qq}$.
\end{lemma}

\begin{proof}
Set $u=\w\left(  \pp\right)  $ and $v=\w%
\left(  \qq\right)  $.

By assumption, the paths $\pp$ and $\qq$ have the same initial
height and the same final height. Call these two heights $i$ and $f$.
Then, their two reading words $u$ and $v$ have the same final height
(namely, $f-i$).
Moreover, Lemma \ref{lem.H.shift} yields
$H\tup{\pp, z} = z^i H\tup{u, z}$ and $H\tup{\qq, z} = z^i H\tup{v, z}$.
Thus, from our assumption $H\tup{\pp, z} = H\tup{\qq, z}$,
we obtain $H\tup{u, z} = H\tup{v, z}$.

Without loss of generality, we assume $f\geq i$.
(Indeed, the case $f<i$ can be either treated analogously, or
reduced to the $f \geq i$ case by reflecting both paths $\pp$ and $\qq$ across
a vertical line and observing that the words $u$ and $v$ are thus transformed by
the map $\omega$.
The latter argument uses Proposition \ref{prop.bal-symm}.)

Hence, the words $u$ and $v$ are rising.
Thus, Proposition~\ref{prop.upnorm.nf} shows that there exist unique up-normal words
$t_{u}$ and $t_{v}$ such that $t_{u} \balsim u$ and
$t_{v} \balsim v$. Lemma \ref{lem:invariant} then shows that $H\left(  t_{u},z\right)
=H\left(  u,z\right)  $
and $H\left(  t_{v},z\right)  =H\left(  v,z\right)  $.
Thus, $H\tup{t_u, z} = H\tup{t_v, z}$
(since $H\tup{u, z} = H\tup{v, z}$).
In other words, the two words $t_u$ and $t_v$ have the same
height polynomial.

Recall that balanced commutations do not change the \# of
$U$'s and the \# of $D$'s in a word. Thus, they do not change
its final height either. Hence, the words $t_u$ and $t_v$
have the same final height as the words $u$ and $v$, which as we know is $f-i$.

Now, we know that the two up-normal words $t_{u}$ and $t_{v}$ have the same
final height and the same height polynomial.
Hence, Lemma \ref{lem:unique} shows that they
must be equal. That is, $t_{u}=t_{v}$.
From $t_{u} \balsim u$ and $t_{v} \balsim v$, we thus obtain
$u \balsim t_{u} = t_{v} \balsim v$.
In other words, $\w\left(  \pp\right)
\balsim \w\left(  \qq\right)
$.
\end{proof}

\begin{proof}[Proof of Theorem~\ref{thm.phiomega}.]
Let $v=\omega\left(  u\right)  $.
Then, the word $v$ is balanced (since $u$ is balanced).
Let $\pp$ and $\qq$ be the standard paths of $u$ and $v$.  
The paths $\pp$ and $\qq$ both start and end at height $0$,
so they have the same initial height and the same final height.

The $k$-th letter of $v$
from the left is the toggle-image\footnote{The \emph{toggle-image} of a letter is
defined as follows: The toggle-image of $U$ is $D$; the toggle-image of $D$ is $U$.} of
the $k$-th letter of $u$ from the right. Therefore, the $k$-th step of
$\qq$ from the left is the toggle-image
of the $k$-th step of
$\pp$ from the right (since $\w\left(  \pp\right)  =u$ and
$\w\left(  \qq\right)  =v$). Hence, the path $\qq$ is the
reflection of the path $\pp$ across a vertical axis (since both paths
start and end at height $0$).
Thus, the paths $\pp$ and $\qq$
have the same multiset of heights of vertices (although the order in which
these heights appear in $\qq$ is opposite from the order in
$\pp$).
Hence, the paths $\pp$ and $\qq$ have the same height polynomial (since the height polynomial of a diagonal path encodes the heights of its vertices).
In other words, $H\tup{\pp, z} = H\tup{\qq, z}$.
We can thus apply Lemma~\ref{lem.4to5} and conclude that $\w\tup{\pp} \balsim \w\tup{\qq}$.
In other words, $u \balsim v$.
Therefore, Lemma~\ref{lem.bal.weyl} yields
$\phi\tup{u} = \phi\tup{v}$.
Since $v = \omega\tup{u}$, these two relations
yield the claims of Theorem~\ref{thm.phiomega}.
\end{proof}
\end{vershort}

\begin{verlong}
\subsection{A lemma}

We are getting close to the proofs of the main results (Theorems~\ref{thm.DU-eqs2} and~\ref{thm.phiomega}).
First, we show a lemma that combines some results of the previous sections:

\begin{lemma}
\label{lem.4to5}
Let $\pp$ and $\qq$ be two diagonal paths with
the same initial height and the same final height. Assume that $H\left(
\pp,z\right)  =H\left(  \qq,z\right)  $. Then,
$\w\tup{\pp} \balsim \w\tup{\qq}$.
\end{lemma}

\begin{proof}
Set $u=\w\left(  \pp\right)  $ and $v=\w\left(  \qq\right)  $.

By assumption, the paths $\pp$ and $\qq$ have the same initial
height and the same final height. Call these two heights $i$ and $f$. We are
in one of the following two cases:

\textit{Case 1:} We have $f\geq i$.

\textit{Case 2:} We have $f<i$.

Consider Case 1 first. In this case, $f\geq i$. Hence, the reading words
$\w\left(  \pp\right)  $ and $\w\left(
\qq\right)  $ are rising (by (\ref{eq.U-D})). In other words, the words
$u$ and $v$ are rising (since $u=\w\left(  \pp\right)  $
and $v=\w\left(  \qq\right)  $). Thus,
Proposition~\ref{prop.upnorm.nf} shows that there exist unique up-normal words
$t_{u}$ and $t_{v}$ such that $t_{u} \balsim u$ and
$t_{v} \balsim v$. Consider these $t_{u}$ and
$t_{v}$. Lemma \ref{lem:invariant} shows that $H\left(  t_{u},z\right)
=H\left(  u,z\right)  $ (since $t_{u} \balsim u$)
and $H\left(  t_{v},z\right)  =H\left(  v,z\right)  $ (since $t_{v} \balsim v$).
Moreover, the relation $t_{u} \balsim u$ shows that the words $t_{u}$ and $u$
have the same \# of $U$'s (since balanced commutations do not change the \# of
$U$'s). In other words, $\left(  \text{\# of }U\text{'s in }t_{u}\right)
=\left(  \text{\# of }U\text{'s in }u\right)  $. Similarly, $\left(  \text{\#
of }D\text{'s in }t_{u}\right)  =\left(  \text{\# of }D\text{'s in }u\right)
$.

However, the path $\pp$ has initial height $i$ and reading word
$\w\left(  \pp\right)  =u$. Thus, Lemma
\ref{lem.H.shift} (applied to $\rr=\pp$ and $j=i$ and $w=u$)
shows that $H\left(  \pp,z\right)  =z^{i}H\left(  u,z\right)  $.
Similarly, $H\left(  \qq,z\right)  =z^{i}H\left(  v,z\right)  $. Hence,
$z^{i}H\left(  u,z\right)  =H\left(  \pp,z\right)  =H\left(
\qq,z\right)  =z^{i}H\left(  v,z\right)  $. Cancelling $z^{i}$, we
obtain $H\left(  u,z\right)  =H\left(  v,z\right)  $. In view of $H\left(
t_{u},z\right)  =H\left(  u,z\right)  $ and $H\left(  t_{v},z\right)
=H\left(  v,z\right)  $, we can rewrite this as $H\left(  t_{u},z\right)
=H\left(  t_{v},z\right)  $. Furthermore, (\ref{eq.U-D0}) yields
\begin{align*}
\left(  \text{final height of }t_{u}\right)    & =\underbrace{\left(  \text{\#
of }U\text{'s in }t_{u}\right)  }_{=\left(  \text{\# of }U\text{'s in
}u\right)  }-\underbrace{\left(  \text{\# of }D\text{'s in }t_{u}\right)
}_{=\left(  \text{\# of }D\text{'s in }u\right)  }\\
& =\left(  \text{\# of }U\text{'s in }u\right)  -\left(  \text{\# of
}D\text{'s in }u\right)  \\
& =\left(  \text{\# of }U\text{'s in }\w\left(  \pp%
\right)  \right)  -\left(  \text{\# of }D\text{'s in }\w\left(
\pp\right)  \right)  \ \ \ \ \ \ \ \ \ \ \left(  \text{since
}u=\w\left(  \pp\right)  \right)  \\
& =\underbrace{\left(  \text{final height of }\pp\right)  }%
_{=f}-\underbrace{\left(  \text{initial height of }\pp\right)  }%
_{=i}\ \ \ \ \ \ \ \ \ \ \left(  \text{by (\ref{eq.U-D})}\right)  \\
& =f-i
\end{align*}
and similarly $\left(  \text{final height of }t_{v}\right)  =f-i$. Comparing
these two equalities, we find $\left(  \text{final height of }t_{u}\right)
=\left(  \text{final height of }t_{v}\right)  $. In other words, the two words
$t_{u}$ and $t_{v}$ have the same final height.

Now, we know that the two up-normal words $t_{u}$ and $t_{v}$ have the same
final height and the same height polynomial (since $H\left(  t_{u},z\right)
=H\left(  t_{v},z\right)  $). Hence, Lemma \ref{lem:unique} shows that they
must be equal. That is, $t_{u}=t_{v}$.
From $t_{u} \balsim u$ and $t_{v} \balsim v$, we thus obtain
$u \balsim t_{u} = t_{v} \balsim v$.
In other words, $\w\left(  \pp\right)
\balsim \w\left(  \qq\right)
$ (since $u=\w\left(  \pp\right)  $ and
$v=\w\left(  \qq\right)  $). Thus, Lemma \ref{lem.4to5}
is proved in Case 1.

Let us now consider Case 2. In this case, $f<i$.

Let $\pp^{\prime}$ and $\qq^{\prime}$ be the reflections of the
diagonal paths $\pp$ and $\qq$ across a vertical line. Then, the
initial heights of $\pp^{\prime}$ and $\qq^{\prime}$ are the
final heights of $\pp$ and $\qq$, which (as we know) are $f$.
Likewise, the final heights of $\pp^{\prime}$ and $\qq^{\prime}$
are $i$. Obviously, from $f<i$, we obtain $i>f$, thus $i\geq f$.

Now, we claim that $\w\left(  \pp^{\prime}\right)
=\omega\left(  \w\left(  \pp\right)  \right)  $. Indeed,
the path $\pp^{\prime}$ is the reflection of $\pp$ across a
vertical line; thus, its steps are the toggle-images\footnote{The \emph{toggle-image} of
an edge of the diagonal lattice is defined as follows: The toggle-image of an
NE-step is an SE-step; the toggle-image of an SE-step is an NE-step.} of the steps
of $\pp$ read in the reverse order. But this means precisely that
$\w\left(  \pp^{\prime}\right)  =\omega\left(
\w\left(  \pp\right)  \right)  $ (since the
anti-automorphism $\omega$ of $\M$ sends $U$ to $D$ and $D$ to $U$
and reverses the order of letters in a word). In view of $\w%
\left(  \pp\right)  =u$, we can rewrite this as $\w%
\left(  \pp^{\prime}\right)  =\omega\left(  u\right)  $. Similarly,
$\w\left(  \qq^{\prime}\right)  =\omega\left(  v\right)
$.

Furthermore, when we reflect a diagonal path across a vertical line, the
heights of its vertices are preserved, and thus its height polynomial remains
unchanged. Hence, $H\left(  \pp^{\prime},z\right)  =H\left(
\pp,z\right)  $ and $H\left(  \qq^{\prime},z\right)  =H\left(
\qq,z\right)  $. Thus, $H\left(  \pp,z\right)  =H\left(
\qq,z\right)  $ rewrites as $H\left(  \pp^{\prime},z\right)
=H\left(  \qq^{\prime},z\right)  $.

Now, we know that $\pp^{\prime}$ and $\qq^{\prime}$ are two
diagonal paths with the same initial height $f$ and the same final height $i$,
and that $H\left(  \pp^{\prime},z\right)  =H\left(  \qq^{\prime
},z\right)  $. Moreover, $i\geq f$. Hence, the claim of Lemma \ref{lem.4to5}
in Case 1 (which we have already proved above) can be applied to
$\pp^{\prime}$, $\qq^{\prime}$, $i$ and $f$ instead of
$\pp$, $\qq$, $f$ and $i$. As a result, we obtain
$\w\left(  \pp^{\prime}\right)
\balsim \w\left(  \qq%
^{\prime}\right)  $. In other words, $\omega\left(  u\right)
\balsim \omega\left(  v\right)  $ (since
$\w\left(  \pp^{\prime}\right)  =\omega\left(  u\right)
$ and $\w\left(  \qq^{\prime}\right)  =\omega\left(
v\right)  $). By Proposition \ref{prop.bal-symm}, this entails
$u \balsim v$. In other words, $\w%
\left(  \pp\right)  \balsim
\w\left(  \qq\right)  $ (since $u=\w%
\left(  \pp\right)  $ and $v=\w\left(  \qq%
\right)  $). Thus, Lemma \ref{lem.4to5} is proved in Case 2.

Now, Lemma \ref{lem.4to5} is proved in both cases.
\end{proof}

\subsection{Proof of Theorem \ref{thm.phiomega}}

Now we can prove the second of our main results:

\begin{proof}[Proof of Theorem~\ref{thm.phiomega}.]
Let $v=\omega\left(  u\right)  $.
Then, the word $v$ is balanced (since $u$ is balanced).

Let $\pp$ and $\qq$ be the standard paths of $u$ and $v$.
Thus, $\pp$ and $\qq$ are diagonal paths starting at $\left(
0,0\right)  $ whose reading words are $\w\tup{\pp} = u$ and
$\w\tup{\qq} = v$. The path $\pp$ ends at the same
height as it starts (since its reading word $\w\left(  \pp\right)  =u$
is balanced), and thus ends at height $0$ (since it starts at height $0$).
Similarly, the same holds for $\qq$. Thus, the paths $\pp$ and
$\qq$ both have final height $0$. Of course, they also
have initial height $0$.

But recall that $v=\omega\left(  u\right)  $. Hence, the $k$-th letter of $v$
from the left is the toggle-image\footnote{The \emph{toggle-image} of a letter is
defined as follows: The toggle-image of $U$ is $D$; the toggle-image of $D$ is $U$.} of
the $k$-th letter of $u$ from the right. Therefore, the $k$-th step of
$\qq$ from the left is the toggle-image
of the $k$-th step of
$\pp$ from the right (since $\w\left(  \pp\right)  =u$ and
$\w\left(  \qq\right)  =v$). Hence, the path $\qq$ is the
reflection of the path $\pp$ across a vertical axis (since both paths
start and end at height $0$). Clearly, reflecting a point across a vertical
axis does not change the height of this point.
Thus, the paths $\pp$ and $\qq$
have the same multiset of heights of vertices (although the order in which
these heights appear in $\qq$ is opposite from the order in
$\pp$).
Hence, the paths $\pp$ and $\qq$ have the same height polynomial (since the height polynomial of a diagonal path encodes the heights of its vertices).
In other words, $H\tup{\pp, z} = H\tup{\qq, z}$.
Since the paths $\pp$ and $\qq$ have the same initial height (namely, $0$) and the same final height (namely, $0$),
we can thus apply Lemma~\ref{lem.4to5} and conclude that $\w\tup{\pp} \balsim \w\tup{\qq}$.
In other words, $u \balsim \omega\tup{u}$ (since $\w\tup{\pp} = u$ and $\w\tup{\qq} = v = \omega\tup{u}$).
Hence, Lemma~\ref{lem.bal.weyl} (applied to $\omega\tup{u}$ instead of $v$) yields
$\phi\tup{u} = \phi\tup{\omega\tup{u}}$.
This completes the proof of Theorem~\ref{thm.phiomega}.
%
\end{proof}
\end{verlong}

\subsection{Two more lemmas}

For the proof of Theorem \ref{thm.DU-eqs2}, we need two more lemmas:

\begin{lemma}
\label{lem.5to6}
If two words $u,v \in \M$ satisfy $u \balsim v$,
then $u \flipsim v$.
\end{lemma}

\begin{vershort}
\begin{proof}
Recall that $\flipsim$ is an equivalence relation.
Hence, it suffices to show that if a word $u$ is obtained from a word $v$ by
a balanced commutation, then $u \flipsim v$.
So let us show this.

Let a word $u$ be obtained from a word $v$ by a balanced commutation. Thus, we
can write $u$ and $v$ as $u=pxyq$ and $v=pyxq$, where $p,q\in\M$ are
two words and where $x,y\in\M$ are two balanced words (by the
definition of a balanced commutation). Consider these $p,q,x,y$.

Since a concatenation of balanced words is still balanced, the concatenation
$yx$ of $y$ and $x$ is balanced. Thus, the word $\omega\left(  yx\right)  $ is also balanced (since
applying $\omega$ to a balanced word yields a balanced word).

Moreover, since $\omega$ is a monoid anti-morphism, we have $\omega\left(
yx\right)  =\omega\left(  x\right)  \omega\left(  y\right)  $. Thus, the word
$\omega\left(  x\right)  \omega\left(  y\right)  $ is balanced. Furthermore, since $\omega
\circ\omega = \id$, we have $\omega\left(  \omega\left(
yx\right)  \right)  =yx$. In other words, $\omega\left(  \omega\left(
x\right)  \omega\left(  y\right)  \right)  =yx$ (since $\omega\left(
yx\right)  =\omega\left(  x\right)  \omega\left(  y\right)  $).

Now, we can apply a balanced flip to the word $u=pxyq$, in which we apply
$\omega$ to the balanced factor $x$. Thus we obtain a new word $u^{\prime
}=p\omega\left(  x\right)  yq$. To this new word $u^{\prime}
=p\omega\left(  x\right)  yq$, we then
apply a further balanced flip, in which we
apply $\omega$ to the balanced factor $y$. Thus we obtain a new word
$u^{\prime\prime}=p\omega\left(  x\right)  \omega\left(  y\right)  q$.
Finally, we apply one last balanced flip to this new word $u^{\prime\prime
}=p\omega\left(  x\right)  \omega\left(  y\right)  q$, in which we apply
$\omega$ to the balanced factor $\omega\left(  x\right)  \omega\left(
y\right)  $. This produces the new word $u^{\prime\prime\prime}%
=p\underbrace{\omega\left(  \omega\left(  x\right)  \omega\left(  y\right)
\right)  }_{=yx}q=pyxq=v$.

Thus we have obtained $v$ from $u$ by a sequence of three balanced flips (via
$u^{\prime}$ and $u^{\prime\prime}$). Hence, $u \flipsim v$. As we said, this proves Lemma \ref{lem.5to6}.
\end{proof}
\end{vershort}

\begin{verlong}
\begin{proof}
Recall that $\flipsim$ is an equivalence relation.
Hence, it suffices to show that if a word $u$ is obtained from a word $v$ by
a balanced commutation, then $u \flipsim v$.
So let us show this.

Let a word $u$ be obtained from a word $v$ by a balanced commutation. Thus, we
can write $u$ and $v$ as $u=pxyq$ and $v=pyxq$, where $p,q\in\M$ are
two words and where $x,y\in\M$ are two balanced words (by the
definition of a balanced commutation). Consider these $p,q,x,y$.

The words $y$ and $x$ are balanced. Hence, their concatenation $yx$ is
balanced as well  (since a concatenation of balanced words is always
balanced). Thus, the word $\omega\left(  yx\right)  $ is also balanced (since
applying $\omega$ to a balanced word yields a balanced word).

Moreover, since $\omega$ is a monoid anti-morphism, we have $\omega\left(
yx\right)  =\omega\left(  x\right)  \omega\left(  y\right)  $. Thus, the word
$\omega\left(  x\right)  \omega\left(  y\right)  $ is balanced (since
$\omega\left(  yx\right)  $ is balanced). Furthermore, since $\omega
\circ\omega = \id$, we have $\omega\left(  \omega\left(
yx\right)  \right)  =yx$. In other words, $\omega\left(  \omega\left(
x\right)  \omega\left(  y\right)  \right)  =yx$ (since $\omega\left(
yx\right)  =\omega\left(  x\right)  \omega\left(  y\right)  $).

Now, we can apply a balanced flip to the word $u=pxyq$, in which we apply
$\omega$ to the balanced factor $x$. Thus we obtain a new word $u^{\prime
}=p\omega\left(  x\right)  yq$. To this new word $u^{\prime}
=p\omega\left(  x\right)  yq$, we then
apply a further balanced flip, in which we
apply $\omega$ to the balanced factor $y$. Thus we obtain a new word
$u^{\prime\prime}=p\omega\left(  x\right)  \omega\left(  y\right)  q$.
Finally, we apply one last balanced flip to this new word $u^{\prime\prime
}=p\omega\left(  x\right)  \omega\left(  y\right)  q$, in which we apply
$\omega$ to the balanced factor $\omega\left(  x\right)  \omega\left(
y\right)  $. This produces the new word $u^{\prime\prime\prime}%
=p\underbrace{\omega\left(  \omega\left(  x\right)  \omega\left(  y\right)
\right)  }_{=yx}q=pyxq=v$.

Thus we have obtained $v$ from $u$ by a sequence of three balanced flips (via
$u^{\prime}$ and $u^{\prime\prime}$). Hence, $u \flipsim v$. As we said, this proves Lemma \ref{lem.5to6}.
\end{proof}
\end{verlong}

\begin{lemma}
\label{lem.6to1}
If two words $u,v \in \M$ satisfy $u \flipsim v$, then
$\phi\tup{u} = \phi\tup{v}$.
\end{lemma}

\begin{proof}
It clearly suffices to show that if a word $u$ is obtained from a word $v$ by
a balanced flip, then $\phi\left(  u\right)  =\phi\left(  v\right)  $.

So let us prove this. Let a word $u$ be obtained from a word $v$ by a balanced
flip. Thus, we can write $u$ and $v$ as $u=pxq$ and $v=p\omega\left(
x\right)  q$, where $p,q\in\M$ are two words and where $x\in
\M$ is a balanced word. Consider these $p,q,x$.

Theorem \ref{thm.phiomega} (applied to $x$ instead of $u$) yields $\phi\left(
x\right)  =\phi\left(  \omega\left(  x\right)  \right)  $ and
$x \balsim \omega\left(  x\right)  $.

Since $\phi$ is a monoid morphism, we have 
\[
\phi\left(  pxq\right)
=\phi\left(  p\right)  \phi\left(  x\right)  \phi\left(  q\right)
\qquad \text{ and } \qquad
\phi\left(  p\omega\left(  x\right)  q\right)  =\phi\left(  p\right)
\phi\left(  \omega\left(  x\right)  \right)  \phi\left(  q\right).
\]
The
right hand sides of these two equalities are equal (since $\phi\left(
x\right)  =\phi\left(  \omega\left(  x\right)  \right)  $). Hence, so are
their left hand sides. In other words, $\phi\left(  pxq\right)  =\phi\left(
p\omega\left(  x\right)  q\right)  $. But this can be rewritten as
$\phi\left(  u\right)  =\phi\left(  v\right)  $ (since $u=pxq$ and
$v=p\omega\left(  x\right)  q$). This proves Lemma \ref{lem.6to1}.
\end{proof}

\subsection{Proof of Theorem \ref{thm.DU-eqs2}}

We are now ready to prove Theorem \ref{thm.DU-eqs2}:

\begin{proof}[Proof of Theorem \ref{thm.DU-eqs2}.]
Let $\pp$ and $\qq$ be the standard paths of $u$ and $v$.
Then, $\pp$ and $\qq$ are diagonal paths starting in $\tup{0,0}$ and having reading words $\w\tup{\pp} = u$ and $\w\tup{\qq} = v$.
In particular, their initial heights are $0$.
Furthermore,
the final heights of the words $u$ and $v$ are (by their definitions)
the final heights of the paths $\pp$ and $\qq$.
Moreover, the height polynomials $H\tup{u,z}$ and $H\tup{v,z}$ are (by
their definitions) the height polynomials $H\left(  \pp,z\right)  $ and
$H\left(  \qq,z\right)  $, and likewise the NE-height polynomials
$H_{\NE}\left(  u,z\right)  $ and $H_{\NE}\left(  v,z\right)  $ are
the NE-height polynomials $H_{\NE}\left(  \pp,z\right)  $ and $H_{\NE}\left(  \qq,z\right)  $.

Having said this, let us now prove the equivalences. It suffices to show that $\mathcal{S}_1 \Longrightarrow \mathcal{S}_2$ and $\mathcal{S}_2 \Longrightarrow \mathcal{S}_1$ and $\mathcal{S}_1 \Longrightarrow \mathcal{S}_3 \Longrightarrow \mathcal{S}_4 \Longrightarrow \mathcal{S}_5 \Longrightarrow \mathcal{S}_6 \Longrightarrow \mathcal{S}_1$ and $\mathcal{S}_4 \Longleftrightarrow \mathcal{S}'_3$.
\medskip

$\mathcal{S}_{1}\Longrightarrow\mathcal{S}_{2}$: Trivial.
\medskip

$\mathcal{S}_{2}\Longrightarrow\mathcal{S}_{1}$: True since the action of
$\W$ on $\kk\left[  x\right]  $ is faithful.
\medskip

$\mathcal{S}_{1}\Longrightarrow\mathcal{S}_{3}$: Assume that $\mathcal{S}_1$ holds. Thus, $\phi\left(u\right) = \phi\left(v\right)$. In other words, $\phi\tup{\w\tup{\pp}} = \phi\tup{\w\tup{\qq}}$ (since $u = \w\tup{\pp}$ and $v = \w\tup{\qq}$).
Moreover, the paths $\pp$ and $\qq$ have the same initial height (since they both start at $\tup{0,0}$).
Thus, Proposition \ref{prop.DU-eqlett} yields that the final heights of $\pp$ and $\qq$ are equal, and that we have
\begin{align*}
&  \left\{  \htt\left(  p_{i}\right)  \ \mid\ p_{i}\text{ is an
NE-step of }\pp\right\}  _{\multiset}\\
&  =\left\{  \htt\left(  q_{i}\right)  \ \mid\ q_{i}\text{ is an
NE-step of }\qq\right\}  _{\multiset}.
\end{align*}
The latter equality says that the paths $\pp$ and $\qq$ have the same multiset of heights of NE-steps.
Equivalently, $H_{\NE}\tup{\pp, z} = H_{\NE}\tup{\qq, z}$ (since the NE-height polynomial of a diagonal path contains the same information as its multiset of heights of NE-steps).
In other words, $H_{\NE}\tup{u, z} = H_{\NE}\tup{v, z}$
(since the NE-height polynomials
$H_{\NE}\left(  u,z\right)  $ and $H_{\NE}\left(  v,z\right)  $ are
the NE-height polynomials $H_{\NE}\left(  \pp,z\right)  $ and $H_{\NE}\left(  \qq,z\right)  $).
Moreover, the final heights of the words $u$ and $v$ are
the final heights of the paths $\pp$ and $\qq$, and thus are equal (since the final heights of $\pp$ and $\qq$ are equal).
Thus, statement $\mathcal{S}_3$ holds.
We have now proved the implication $\mathcal{S}_1 \Longrightarrow \mathcal{S}_3$.
\medskip

$\mathcal{S}_{3}\Longrightarrow\mathcal{S}_{4}$: Assume that $\mathcal{S}_3$ holds.
That is, the words $u$ and $v$ have the same final
height and satisfy $H_{\NE}\left(  u,z\right)
=H_{\NE}\left(  v,z\right)  $. Let $b$ be the final height of $u$ and $v$.
The equality \eqref{eq.lem.laurent-w.NE} from Lemma~\ref{lem:laurent-w} yields
\[
H(u,z) = (1+z)H_{\NE}(u,z) + \sum_{j \geq b} z^j - \sum_{j \geq 1} z^j
\]
and
\[
H(v,z) = (1+z)H_{\NE}(v,z) + \sum_{j \geq b} z^j - \sum_{j \geq 1} z^j .
\]
The right hand sides of these two equalities are equal (since $H_{\NE}\left(  u,z\right)
=H_{\NE}\left(  v,z\right)  $). Hence, so are the left hand sides.
In other words, $H\tup{u, z} = H\tup{v, z}$.
Since we also know that the words $u$ and $v$ have the same final
height, we thus conclude that statement $\mathcal{S}_4$ holds.
Thus we have proved $\mathcal{S}_{3}\Longrightarrow\mathcal{S}_{4}$.
\medskip

$\mathcal{S}_{4}\Longrightarrow\mathcal{S}_{5}$:
Assume that statement $\mathcal{S}_4$ holds.
In other words, the words $u$ and $v$ have the same final
height and satisfy $H\tup{u, z} = H\tup{v, z}$.

The final heights of the words $u$ and $v$ are the final heights of their standard paths $\pp$ and $\qq$ (by definition).
Thus, the final heights of $\pp$ and $\qq$ are equal (since the final heights of $u$ and $v$ are equal).

Recall that the height polynomials $H\tup{u,z}$ and $H\tup{v,z}$ are
the height polynomials $H\left(  \pp,z\right)  $ and
$H\left(  \qq,z\right)  $.
Hence, $H\tup{\pp, z} = H\tup{\qq, z}$
(since $H\tup{u, z} = H\tup{v, z}$).
By Lemma \ref{lem.4to5}, this entails $\w\tup{\pp} \balsim \w\tup{\qq}$
(since the paths $\pp$ and $\qq$ have the same initial height and the same final height).
This can be rewritten as $u \balsim v$ (since $u = \w\tup{\pp}$ and $v = \w\tup{\qq}$).
But this is exactly $\mathcal{S}_5$.
Thus, the implication $\mathcal{S}_4 \Longrightarrow \mathcal{S}_5$ is proved.
\medskip


$\mathcal{S}_{5}\Longrightarrow \mathcal{S}_{6}$: This is Lemma~\ref{lem.5to6}.
\medskip

$\mathcal{S}_{6}\Longrightarrow \mathcal{S}_{1}$: This is Lemma~\ref{lem.6to1}.
\medskip

$\mathcal{S}_4 \Longleftrightarrow \mathcal{S}'_3$:
Next, we show the equivalence
$\mathcal{S}_4 \Longleftrightarrow \mathcal{S}'_3$.
This is tantamount to showing the equivalence of the two
equalities $H\tup{u, z} = H\tup{v, z}$
and $H_{\SE} \tup{u, z} = H_{\SE} \tup{v, z}$
under the assumption that the words $u$ and $v$ have the same final height.

So let us assume that the words $u$ and $v$ have the same final height.
Let $b$ be this final height.
The equality \eqref{eq.lem.laurent-w.SE} from Lemma~\ref{lem:laurent-w} yields
\[
H(u,z) = (1+z^{-1})H_{\SE}(u,z) + \sum_{j \geq 0} z^j - \sum_{j \geq b+1} z^j.
\]
and
\[
H(v,z) = (1+z^{-1})H_{\SE}(v,z) + \sum_{j \geq 0} z^j - \sum_{j \geq b+1} z^j.
\]
Clearly, the left hand sides of these two equalities are equal
if and only if $H\tup{u, z} = H\tup{v, z}$, whereas the right
hand sides are equal if and only if
$H_{\SE} \tup{u, z} = H_{\SE} \tup{v, z}$ (because the Laurent
polynomial $1 + z^{-1}$ is not a zero-divisor and thus can be cancelled).
Thus, the equalities $H\tup{u, z} = H\tup{v, z}$
and $H_{\SE} \tup{u, z} = H_{\SE} \tup{v, z}$ are equivalent.
As we said, this proves the equivalence
$\mathcal{S}_4 \Longleftrightarrow \mathcal{S}'_3$.
\end{proof}

\section{Enumeration}\label{sec:enum}

Two words $u, v \in \M$ will be called \emph{Weyl-equivalent} if $\phi\tup{u} = \phi\tup{v}$. Obviously, Weyl equivalence is an equivalence relation.
Theorem~\ref{thm.DU-eqs2} (and, later, Theorem~\ref{thm.DU-eqs3}) provides some necessary and sufficient criteria for Weyl equivalence.
In particular, the $\mathcal{S}_1 \Longleftrightarrow \mathcal{S}_5$ part of Theorem~\ref{thm.DU-eqs2} shows that Weyl equivalence is precisely the relation $\balsim$.

In this section, we prove several enumerative results regarding the equivalence classes of Weyl equivalence (henceforth just called ``equivalence classes'').

\subsection{Counting equivalence classes by numbers of $D$'s and $U$'s}

First, we consider equivalence classes of words with a given number of $D$'s and $U$'s. For $0 \leq k \leq n$, let $a(n,k)$ be the number of equivalence classes of words with $k$ many $D$'s and $n-k$ many $U$'s.
One of the simplest properties of these numbers is the following symmetry:

\begin{proposition}
    \label{prop.a.sym}
    We have $a(n, k) = a(n, n-k)$ for any integers $0 \leq k \leq n$.
\end{proposition}

\begin{proof}
There are many easy ways to see this.
For instance, Proposition~\ref{prop.bal-symm} shows that the monoid anti-automorphism $\omega : \M \to \M$ sends equivalence classes to equivalence classes (since Weyl equivalence is the relation $\balsim$).
But $\omega$ turns $D$'s into $U$'s and vice versa.
Thus, the proposition follows.
\end{proof}

In particular, $a(n,0) = a(n,n) = 1$ (the only words in these cases are $UU\cdots U$ and $DD\cdots D$ respectively).
\medskip

Our first real result about the $a(n,k)$ is the following recursion.

\begin{lemma}
\label{lemma:recursion}
For $n > 2k \geq 0$, we have
\[a(n, k) = a(n-1, k) + a(n-2, k-1).\]
Here, $a(n-2,-1)$ is interpreted as $0$ when $k=0$.
\end{lemma}

\begin{proof}
Recall that $a(n, k)$ counts the equivalence classes of words with $k$ many $D$'s and $n-k$ many $U$'s.
Such words always have more $U$'s than $D$'s (since $n > 2k$), and thus (in particular) are rising.
Thus, any such equivalence class has a unique up-normal representative (by Proposition~\ref{prop.upnorm.nf}).
Therefore, $a(n, k)$ counts the up-normal words $w$ with $k$ many $D$'s and $n-k$ many $U$'s. Recall that every up-normal word has the form
\begin{equation}\label{eq:unw_shape}
D^a\, (UD)^{r_1}U\, (UD)^{r_2}U\, \cdots\, (UD)^{r_h}U\, D^b
\end{equation}
for nonnegative integers $a,b,r_1,\ldots,r_h$ (see Proposition~\ref{rem.shape_normal}). 
An up-normal rising word $w$ can be of the following two types:
\begin{enumerate}[itemindent=.7cm,label=\textbf{Type \arabic*:}]
\item The standard path corresponding to $w$ has only one vertex of maximum height. Equivalently, $r_h = 0$ in~\eqref{eq:unw_shape}. In this case, we can 
remove the last $U$ from $w$ to obtain the word
\[D^a\, (UD)^{r_1}U\, (UD)^{r_2}U\, \cdots\, (UD)^{r_{h-1}}U\, D^b\]
of length $n-1$ consisting of $k$ many $D$'s and $n-k-1$ $U$'s. This word is still up-normal. Since $n-1 \geq 2k$, it is also still rising. It is clear that one can reverse the procedure: given a rising up-normal word of length $n-1$ with $k$ many $D$'s and $n-k-1$ many $U$'s, insert a $U$ just before the final run of $D$'s if its last letter is $D$, and at the end otherwise. This gives us a bijection between equivalence classes counted by $a(n-1,k)$ and the equivalence classes of the first type. See Figure~\ref{fig.bij.type.1} for an illustration.
\item The standard path corresponding to $w$ has at least two distinct vertices of maximum height. Equivalently, $r_h > 0$ in~\eqref{eq:unw_shape}. In this case, we can remove the last $U$ and the $D$ right before it to obtain the word
\[D^a\, (UD)^{r_1}U\, (UD)^{r_2}U\, \cdots\, (UD)^{r_h-1}U\, D^b,\]
which is still up-normal and rising. Its length is $n-2$, and it has $k-1$ many $D$'s. 
Again, the process is easily reversed by inserting $DU$ either before the final run of $D$'s if the last letter is $D$, or at the end.
This gives us a bijection between equivalence classes counted by $a(n-2,k-1)$ and the equivalence classes of the second type. See Figure~\ref{fig.bij.type.2} for an illustration.
\end{enumerate}
Combining the two types, we obtain the desired recursion.
\end{proof}

\begin{figure}[htbp]
\begin{center}
\begin{tikzpicture}[scale=0.7]
\draw[help lines] (-0.2, -3.2) grid (17.2, 3.2);
\draw [line width=2pt] (-0.2, 0) -- (17.2, 0);
\draw
[line width=2pt,-stealth] (0, 0) edge (1, -1) (1, -1) edge (2, -2) (2, -2) edge (3, -3) (3, -3) edge (4, -2) (4, -2) edge (5, -3) (5, -3) edge (6, -2) (6, -2) edge (7, -1) (7, -1) edge (8, 0) (8, 0) edge (9, 1) (9, 1) edge (10,0) (10,0) edge (11,1) (11,1) edge (12,2) (12,2) edge (13,1) (13,1) -- (14, 2);
\draw
[line width=2pt,-stealth,dashed,red] (14, 2) -- (15, 3);
\draw
[line width=1pt,dotted,blue,-stealth] (15, 3) edge (16, 2) (16, 2) -- (17, 1);
\draw
[line width=1pt,dotted,blue,-stealth] (14, 2) edge (15, 1) (15, 1) -- (16, 0);
\node at (8.5, -4) {$DDDUDUUUUDUUDU\red{U}DD \leftrightarrow DDDUDUUUUDUUDUDD$};
\end{tikzpicture}
\end{center}
\caption{The bijection for words of type 1. The arc that is removed/inserted is indicated by a dashed red line. Arcs that are shifted by the procedure are indicated by dotted blue lines.}
\label{fig.bij.type.1}
\end{figure}

\begin{figure}[htbp]
\begin{center}
\begin{tikzpicture}[scale=0.7]
\draw[help lines] (-0.2, -3.2) grid (17.2, 3.2);
\draw [line width=2pt] (-0.2, 0) -- (17.2, 0);
\draw
[line width=2pt,-stealth] (0, 0) edge (1, -1) (1, -1) edge (2, -2) (2, -2) edge (3, -3) (3, -3) edge (4, -2) (4, -2) edge (5, -3) (5, -3) edge (6, -2) (6, -2) edge (7, -1) (7, -1) edge (8, 0) (8, 0) edge (9, 1) (9, 1) edge (10,0) (10,0) edge (11,1) (11,1) edge (12,2) (12,2) edge (13,3);
\draw
[line width=2pt,-stealth,dashed,red] (13,3) edge (14,2) (14, 2) -- (15, 3);
\draw
[line width=1pt,dotted,blue,-stealth] (15, 3) edge (16, 2) (16, 2) -- (17, 1);
\draw
[line width=1pt,dotted,blue,-stealth] (13, 3) edge (14, 2) (14, 2) -- (15, 1);
\node at (8.5, -4) {$DDDUDUUUUDUUU\red{DU}DD \leftrightarrow DDDUDUUUUDUUUDD$};
\end{tikzpicture}
\end{center}
\caption{The bijection for words of type 2. The arcs that are removed/inserted are indicated by dashed red lines. Arcs that are shifted by the procedure are indicated by dotted blue lines.}
\label{fig.bij.type.2}
\end{figure}

Second, we prove explicit formulas in two special cases, which together with the previous lemma characterize the numbers $a(n,k)$ for all $n$ and $k$.

\begin{lemma}
\label{lemma:explicit}
For $k > 0$, we have
\[ a(2k, k) = (k + 3)2^{k-2} \quad\text{and}\quad a(2k + 1, k) = (k + 2)2^{k-1}.\]
\end{lemma}

\begin{proof}
We start with the first formula.
Recall (from Theorem~\ref{thm.DU-eqs2}, equivalence $\mathcal{S}_1 \Longleftrightarrow \mathcal{S}_3$) that the equivalence class of a word (with given numbers of $U$'s and $D$'s) is uniquely determined by the multiset of heights of NE-steps in its standard path.
Let the minimum and maximum heights of NE-steps be $-s$ and $t$ respectively (with $s \geq 0$ and $t \geq -1$, where the case $t = -1$ means that the path has no NE-steps).
Let $h_j$ be the number of NE-steps of height $j$ (for each $-s \leq j \leq t$).
Since our words are balanced (they have $k$ many $D$'s and $k$ many $U$'s), the equivalence class has an up-normal representative (by Proposition~\ref{prop.bal-symm}), and therefore all $h_j$'s need to be strictly positive.
Thus, $\tup{h_{-s}, h_{-s+1}, \ldots, h_t}$ is a composition of $k$ into $s+t+1$ positive integers. For each composition of $k$ of length $\ell$ (of which there are $\binom{k-1}{\ell-1}$), there are $\ell+1$ possibilities for the pair $(s,t)$ ($s$ can be any integer from $0$ to $\ell$, and $t = \ell-s-1$). This gives us a total of
\[
\sum_{\ell=1}^k \binom{k-1}{\ell-1} (\ell+1) = (k+3)2^{k-2}
\]
possibilities.

For the second formula, we can use induction combined with the recursion in Lemma~\ref{lemma:recursion}, which gives us
\[a(2k+1,k) = a(2k,k) + a(2k-1,k-1),\]
or we can apply a similar combinatorial argument with compositions of $k+1$ (the only difference is the fact that a composition of length $\ell$ only gives rise to $\ell$ possibilities, since the maximum height $t$ of an NE-step can no longer be $-1$).
\end{proof}

The combination of these two lemmas yields an explicit formula for the bivariate generating function of $a(n,k)$.

\begin{theorem}
\label{thm:generatingfunction}
We have
\begin{equation}\label{eq:firstgf}
\sum_{n \geq 0} \sum_{0 \leq k \leq n} a(n,k)t^k x^n = \frac{(1-3tx^2+t^2x^4)(1-tx^2)^2}{(1-tx-tx^2)(1-x-tx^2)(1-2tx^2)^2}
\end{equation}
or equivalently
\begin{equation}\label{eq:secondgf}
\sum_{n \geq 0} \sum_{0 \leq k \leq n/2} a(n,k)t^k x^n = \frac{(1-tx^2)^3}{(1-x-tx^2)(1-2tx^2)^2}.
\end{equation}
\end{theorem}

\begin{proof}
We start with the second identity~\eqref{eq:secondgf}. Let us write $A(x,t)$ for the bivariate generating function on the left. Multiplying the recursion in Lemma~\ref{lemma:recursion} by $t^k x^n$ and summing over all $n$ and $k$ with $n > 2k \geq 0$, we obtain
\[
\sum_{n \geq 1} \sum_{0 \leq k < n/2} a(n,k)t^k x^n = \sum_{n \geq 1} \sum_{0 \leq k < n/2} a(n-1,k)t^k x^n + \sum_{n \geq 2} \sum_{0 \leq k < n/2} a(n-2,k-1)t^k x^n.
\]
In terms of the generating function $A(x,t)$, this becomes
\[
A(x,t) - \sum_{k \geq 0} a(2k,k) t^k x^{2k} = x A(x,t) + tx^2 A(x,t) - \sum_{k \geq 0} a(2k,k) t^{k+1} x^{2k+2}.
\]
Solving for $A(x,t)$ now yields
\begin{equation}\label{eq:Axt-intermediate}
A(x,t) = \frac{1-tx^2}{1-x-tx^2} \sum_{k \geq 0} a(2k,k) t^k x^{2k}.
\end{equation}
In view of Lemma~\ref{lemma:explicit} (noting also that $a(0,0) = 1$), the sum evaluates to 
\begin{equation}\label{eq:central_terms}
    \sum_{k \geq 0} a(2k,k)t^kx^{2k} = \frac{(1-tx^2)^2}{(1-2tx^2)^2},
\end{equation}
which completes the proof of~\eqref{eq:secondgf}. To prove the first identity, we write
\begin{align*}
\sum_{n \geq 0} \sum_{0 \leq k \leq n} a(n,k)t^k x^n &=
\sum_{n \geq 0} \sum_{0 \leq k \leq n/2} a(n,k)t^k x^n + \sum_{n \geq 0} \sum_{0 \leq k \leq n/2} a(n,n-k)t^{n-k} x^n \\
&\qquad\quad - \sum_{k \geq 0} a(2k,k) t^k x^{2k}.
\end{align*}
By the symmetry property $a(n,n-k) = a(n,k)$, the first two terms are $A(x,t)$ and $A(tx,1/t)$ respectively. The final term is precisely~\eqref{eq:central_terms} again. Now identity~\eqref{eq:firstgf} follows upon simplification.
\end{proof}

\begin{corollary}\label{cor:sumformula}
For all $n$ and $k$ with $0 \leq k \leq n/2$, we have
\[
a(n,k) = \sum_{j=0}^k (k-j+1) \binom{n-k-1}{j}.
\]
\end{corollary}

\begin{proof}
Let $\Sigma(n,k)$ be the sum on the right side of the equation. It is easy to verify that $\Sigma(n,0) = a(n,0) = 1$ and $\Sigma(2k,k) = a(2k,k) = (k + 3)2^{k-2}$ as well as $\Sigma(2k+1,k) = a(2k+1,k) = (k + 2)2^{k-1}$. Since these values together with the recursion in Lemma~\ref{lemma:recursion} characterize $a(n,k)$ uniquely, it suffices to verify that
\[
\Sigma(n,k) = \Sigma(n-1,k) + \Sigma(n-2,k-1).
\]
The latter is a simple consequence of the recursion for the binomial coefficients, since
\[
\Sigma(n-1,k) = \sum_{j = 0}^k  (k - j + 1) \binom{n - k - 2}{j}
\]
and
\[
\Sigma(n-2,k-1) = \sum_{j = 0}^{k-1} (k-j)\binom{n-k-2}{j} = \sum_{j = 1}^{k} (k-j+1)\binom{n-k-2}{j-1}.
\]
\end{proof}

\begin{table}[ht]
\centering
\begin{tabular}{ |c|ccccccccccc|} 
 \hline
 \diagbox[width=10mm]{$n$}{$k$} & 0 & 1 & 2 & 3 & 4 & 5 & 6 & 7 & 8 & 9 & 10 \\ 
 \hline
$0$ & $1$ & & & & & & & & & & \\
$1$ & $1$ & $1$ & & & & & & & & & \\
$2$ & $1$ & $2$ & $1$ & & & & & & & & \\
$3$ & $1$ & $3$ & $3$ & $1$ & & & & & & & \\
$4$ & $1$ & $4$ & $5$ & $4$ & $1$ & & & & & & \\
$5$ & $1$ & $5$ & $8$ & $8$ & $5$ & $1$ & & & & & \\
$6$ & $1$ & $6$ & $12$ & $12$ & $12$ & $6$ & $1$ & & & & \\ 
$7$ & $1$ & $7$ & $17$ & $20$ & $20$ & $17$ & $7$ & $1$ & & & \\
$8$ & $1$ & $8$ & $23$ & $32$ & $28$ & $32$ & $23$ & $8$ & $1$ & & \\ 
$9$ & $1$ & $9$ & $30$ & $49$ & $48$ & $48$ & $49$ & $30$ & $9$ & $1$ & \\
$10$ & $1$ & $10$ & $38$ & $72$ & $80$ & $64$ & $80$ & $72$ & $38$ & $10$ & $1$ \\
\hline
\end{tabular}
\caption{Table of the values of $a(n,k)$.}\label{tab:a}
\end{table}

\begin{remark}\label{rem:comb_int}
The terms in Corollary~\ref{cor:sumformula} (see Table~\ref{tab:a} for some explicit values) have a simple combinatorial interpretation. Recall that the up-normal representative of any equivalence class can be written as $D^a\, (UD)^{r_1}U\, (UD)^{r_2}U\, \cdots \, (UD)^{r_h}U\, D^b$ for some nonnegative integers $a,b,h$ and $r_1,r_2,\ldots,r_h$ (see~\eqref{eq.lem.unique.w=}). Counting $U$'s and $D$'s, we find that $a+b+r_1+r_2+\cdots+r_h = k$ and $(r_1+1)+(r_2+1)+\cdots+(r_h+1) = n-k$. From these, one obtains $a+b = 2k+h-n$. Given $h$ (which is the number of distinct heights of NE-steps in the corresponding standard path), there are thus $2k+h-n+1$ possibilities for $a$ and $b$. 
Moreover, $r_1+1,r_2+1,\ldots,r_h+1$ is a composition of $n-k$ into $h$ positive integers, for which there are $\binom{n-k-1}{h-1}$ possibilities. Thus the total number of equivalence classes must be
\begin{equation*}
\sum_{h=n-2k}^{n-k} (2k+h-n+1) \binom{n-k-1}{h-1} = \sum_{j=0}^k (k-j+1) \binom{n-k-1}{j},
\end{equation*}
where the second expression is obtained from the first by the simple substitution $h = n-k-j$. The numbers $r_1+1,r_2+1,\ldots,r_h+1$ are the multiplicities in the multiset of heights of NE-steps, while $a$ and $b$ determine at which $h$ consecutive heights the NE-steps occur. 
\end{remark}

\subsection{Counting all equivalence classes of a given length}

\begin{corollary}\label{cor:total_number}
    The total number of equivalence classes of words of length $n > 0$ is
    \[
	\sum_{k=0}^n a(n,k)
	= 2F_{n+4} - \begin{cases} (3n+42)2^{n/2-3}, & \text{if } n \text{ even,} \\ (n+15)2^{(n-3)/2}, & \text{if } n \text{ odd,} \end{cases}
	\]
    where $F_n$ is the $n$-th Fibonacci number. See Table~\ref{tab:a_tot}.
\end{corollary}

\begin{table}[ht]
\centering
\begin{tabular}{ |c|ccccccccccc|} 
 \hline
$n$ & 0 & 1 & 2 & 3 & 4 & 5 & 6 & 7 & 8 & 9 & 10 \\ 
\hline
$\sum_k a(n,k)$ & 1 & 2 & 4 & 8 & 15 & 28 & 50 & 90 & 156 & 274 & 466 \\
\hline
\end{tabular}
\caption{Total number of equivalence classes for $n \leq 10$.}\label{tab:a_tot}
\end{table}

\begin{proof}
\begin{verlong}
We simply plug $t=1$ into~\eqref{eq:firstgf} to obtain the generating function for the total number of equivalence classes, which is
\begin{align*}
\sum_{n \geq 0} \Big( \sum_{0 \leq k \leq n} a(n,k) \Big) x^n &= \frac{(1+x-x^2)(1-x^2)^2}{(1-x-x^2)(1-2x^2)^2} \\
&= \frac14 + \frac{2 (3 + 2 x)}{1-x-x^2} - \frac{9+14x}{2(1-2x^2)} - \frac{3+4x}{4(1-2x^2)^2}.
\end{align*}
Reading off coefficients from the generating function yields the formula.
\end{verlong}
\begin{vershort}
Plug $t = 1$ into~\eqref{eq:firstgf} and compare coefficients using standard generating function techniques.
\end{vershort}
\end{proof}

\subsection{$c$-Dyck words}

Let us now turn our attention to restricted words. Fix a constant $c > 0$, and consider words with the property that every prefix has at least $c$ times as many $U$'s as $D$'s. These words form a submonoid $\M_c$ of $\M$. Again, we will be interested in the number of equivalence classes of words of length $n$ in $\M_c$. Note here that not all words equivalent to a word in $\M_c$ are necessarily also in $\M_c$. To give a simple example, the word $UUDUUD$ is in $\M_2$ while the equivalent word $UUUDDU$ is not.

Let $a_c(n,k)$ be the number of equivalence classes of words in the submonoid $\M_c$ that consist of $k$ $D$'s and $n-k$ $U$'s. Note that we must have $n-k \geq ck$, or equivalently $n \geq (c+1)k$. We first show that we can again focus on up-normal words.

\begin{lemma}\label{lem:normal_form_only}
    Let $c \geq 1$ be a real constant. An equivalence class of words in $\M$ contains words in $\M_c$ if and only if it consists of rising words and the unique up-normal representative is in $\M_c$.
\end{lemma}

\begin{proof}
A word lies in $\M_c$ if and only if the associated diagonal path stays above\footnote{``Above'' means ``weakly above''; i.e., the path is allowed to touch this line.} the line $y = \frac{c-1}{c+1} x$. This is because the part of the path corresponding to a prefix with $a$ $U$'s and $b$ $D$'s ends at $(a+b,a-b)$. The condition $a \geq cb$ translates to $a-b \geq \frac{c-1}{c+1} (a+b)$. In particular, for $c \geq 1$, all elements of $\M_c$ are rising by definition. We know that for every equivalence class of rising words, there is a unique up-normal representative (Proposition~\ref{prop.upnorm.nf}).

Now let an equivalence class $C$ be given. All diagonal paths that correspond to a word in $C$ have to end at the same point $(a,b)$. If $b < \frac{c-1}{c+1} a$, then $\M_c$ cannot contain any elements of $C$, and there is nothing to prove. 

So assume that $b \geq \frac{c-1}{c+1} a$, and assume also that there is a word $w \in C$ that lies in $\M_c$. Let $\pp$ be the corresponding standard path, and consider any nonnegative integer $s < b$. Every vertex of $\pp$ whose $y$-coordinate is less than or equal to $s$ has to have $x$-coordinate less than or equal to $\frac{c+1}{c-1} s$ (we can interpret this as $\infty$ if $c = 1$). Therefore, there are at most $\frac{c+1}{c-1}s + 1$ arcs in $\pp$ that have at least one end at height $s$ or less. Since the multiset of step heights is the same for all words in $C$ (by Theorem~\ref{thm.DU-eqs2}), this also holds for the diagonal path $\qq$ that corresponds to the up-normal representative. By construction, all steps with an end at height $s$ or less occur before all others in $\qq$, and it follows that the rightmost vertex of $\qq$ whose height is $s$ has an $x$-coordinate of at most $\frac{c+1}{c-1}s$, which means that it lies above or on the line $y = \frac{c-1}{c+1} x$. For every vertex whose height is greater than or equal to $b$, this is automatically true since the final vertex $(a,b)$ of $\qq$ has this property. We have thus shown that the entire path $\qq$ lies above the line $y = \frac{c-1}{c+1} x$, so the up-normal representative lies in $\M_c$. This completes the proof.
\end{proof}

In analogy to Lemma~\ref{lemma:recursion}, the following lemma holds.

\begin{lemma}
\label{lemma:recursion2}
For every real constant $c \geq 1$ and every pair $(n,k)$ of positive integers with $n-1 \geq (c+1)k$, we have
\begin{equation}\label{eq:general_rec}
a_c(n, k) = a_c(n-1, k) + a_c(n-2, k-1).
\end{equation}
\end{lemma}

\begin{proof}
Again, we use the characterization that the standard paths corresponding to words in the submonoid $\M_c$ have to stay entirely above the line $y = \frac{c-1}{c+1} x$. By Lemma~\ref{lem:normal_form_only}, it suffices to consider the unique up-normal representative of any equivalence class that is counted by $a_c(n,k)$. The same bijections as in the proof of Lemma~\ref{lemma:recursion} apply; we only need to check that removing the last $U$ (Type 1) or the last $DU$ (Type 2) yields a new word that is still in in $\M_c$. Let us consider the two different types:
\begin{enumerate}[itemindent=.7cm,label=\textbf{Type \arabic*:}]
\item The only change in the standard path is the final descent, which is shifted one unit to the left and one unit down (or one unit to the right and one unit up in the inverse operation). The final descent is entirely above the line $y = \frac{c-1}{c+1} x$ provided that its final vertex is. This is guaranteed in both directions by the assumption that $n-1 \geq (c+1)k$. The procedure is shown in Figure~\ref{fig.bij2.type.1}. 
\item In this case, the only change in the standard path is that the final descent is moved two units to the left (or to the right when the inverse is applied). As for Type 1, it stays entirely above the line $y = \frac{c-1}{c+1} x$ in both directions because of the assumption that $n-1 \geq (c+1)k$. See Figure~\ref{fig.bij2.type.2} for an illustration.
\end{enumerate}
Combining the two types, we obtain the desired recursion again.
\end{proof}

\begin{figure}[htbp]
\begin{center}
\begin{tikzpicture}[scale=0.7]
\draw[help lines] (-0.2, -0.2) grid (13.2, 7.2);
\draw [line width=2pt] (-0.2, -0.066666) -- (13.2, 4.4);
\draw
[line width=2pt,-stealth] (0, 0) edge (1, 1) (1, 1) edge (2, 2) (2, 2) edge (3, 1) (3, 1) edge (4, 2) (4, 2) edge (5, 3) (5, 3) edge (6, 4) (6, 4) edge (7, 5) (7, 5) edge (8, 4) (8, 4) edge (9, 5) (9, 5) -- (10, 6);
\draw
[line width=2pt,-stealth,dashed,red] (10, 6) -- (11, 7);
\draw
[line width=1pt,dotted,blue,-stealth] (11, 7) edge (12, 6) (12, 6) -- (13, 5);
\draw
[line width=1pt,dotted,blue,-stealth] (10, 6) edge (11, 5) (11, 5) -- (12, 4);
\node at (6.5, -1) {$UUDUUUUDUU\red{U}DD \leftrightarrow UUDUUUUDUUDD$};
\end{tikzpicture}
\end{center}
\caption{The bijection for words of type 1. The arc that is removed/inserted is indicated by a dashed red line. Arcs that are shifted by the procedure are indicated by dotted blue lines. The line $y = \frac{c-1}{c+1} x$ is shown as well: in this example, $c=2$.}
\label{fig.bij2.type.1}
\end{figure}

\begin{figure}[htbp]
\begin{center}
\begin{tikzpicture}[scale=0.7]
\draw[help lines] (-0.2, -0.2) grid (13.2, 7.2);
\draw [line width=2pt] (-0.2, -0.066666) -- (13.2, 4.4);
\draw
[line width=2pt,-stealth] (0, 0) edge (1, 1) (1, 1) edge (2, 2) (2, 2) edge (3, 1) (3, 1) edge (4, 2) (4, 2) edge (5, 3) (5, 3) edge (6, 4) (6, 4) edge (7, 5) (7, 5) edge (8, 6) (8, 6) -- (9, 7);
\draw
[line width=2pt,-stealth,dashed,red] (9, 7) edge (10, 6) 
(10, 6) -- (11, 7);
\draw
[line width=1pt,dotted,blue,-stealth]  (11, 7) edge (12, 6) (12, 6) -- (13, 5);
\draw
[line width=1pt,dotted,blue,-stealth]  (9, 7) edge (10, 6) (10, 6) -- (11, 5);
\node at (6.5, -1)
{$UUDUUUUUU\red{DU}DD \leftrightarrow UUDUUUUUUDD$};
\end{tikzpicture}
\end{center}
\caption{The bijection for words of type 2. The arcs that are removed/inserted are indicated by dashed red lines. Arcs that are shifted by the procedure are indicated by dotted blue lines. The line $y = \frac{c-1}{c+1} x$ is shown as well: in this example, $c=2$.}
\label{fig.bij2.type.2}
\end{figure}

\begin{remark}
The recursion~\eqref{eq:general_rec} is in general false for $c < 1$. As a counterexample, note that $a_{1/2}(4,2) = 3$ (the three elements $UUDD$, $UDUD$ and $UDDU$ of $\M_{1/2}$ belong to three distinct equivalence classes), while $a_{1/2}(3,2) = 1$ (the only element is $UDD$) and $a_{1/2}(2,1) = 1$ (the only element is $UD$).
\end{remark}

For positive integer values of $c$, there is in fact a fairly simple explicit formula for $a_c(n,k)$.

\begin{theorem}
If $c$ is a positive integer and $n,k$ are positive integers with $n \geq (c+1)k$, then we have
\begin{equation}\label{eq:explicit_general_c}
a_c(n,k) = \binom{n-k-1}{k} - (c-2) \sum_{j=0}^{k-1} \binom{n-k-1}{j}.
\end{equation}
\end{theorem}

\begin{proof}
The recursion~\eqref{eq:general_rec} determines all values of $a_c(n,k)$ except for the boundary cases where $n = (c+1)k$. 
However, if $n = (c+1)k$, then the last letter of every valid word in $\M_c$ has to be a $D$ (otherwise, the condition of the submonoid $\M_c$ is not satisfied for the prefix obtained by removing the last letter). This immediately implies that
\begin{equation}\label{eq:boundary}
a_c((c+1)k,k) = a_c((c+1)k-1,k-1).
\end{equation}
Together with~\eqref{eq:general_rec} and the trivial initial value $a_c(0,0) = 1$, this determines $a_c(n,k)$ uniquely for all values of $n$ and $k$ (with $n \geq (c+1)k \geq 0$), so it suffices to verify that the expression on the right side of~\eqref{eq:explicit_general_c} satisfies the recursion~\eqref{eq:general_rec} as well as~\eqref{eq:boundary}. The former is a simple consequence of the recursion for the binomial coefficients. The latter is (after some trivial cancellations) equivalent to
\[\binom{ck-1}{k} = (c-1) \binom{ck-1}{k-1},\]
which is readily verified. The theorem follows immediately by induction.
\end{proof}

\begin{table}[ht]
\centering
\begin{tabular}{ |c|cccccc|c|} 
 \hline
 \diagbox[width=10mm]{$n$}{$k$} & 0 & 1 & 2 & 3 & 4 & 5 & $\sum$ \\ 
 \hline
$1$ & $1$ & & & & & & $1$ \\
$2$ & $1$ & $1$ & & & & & $2$ \\
$3$ & $1$ & $2$ & & & & & $3$ \\
$4$ & $1$ & $3$ & $2$ & & & & $6$ \\
$5$ & $1$ & $4$ & $4$ & & & & $9$ \\
$6$ & $1$ & $5$ & $7$ & $4$ & & & $17$ \\ 
$7$ & $1$ & $6$ & $11$ & $8$ & & & $26$ \\
$8$ & $1$ & $7$ & $16$ & $15$ & $8$ & & $47$ \\ 
$9$ & $1$ & $8$ & $22$ & $26$ & $16$ & & $73$ \\
$10$ & $1$ & $9$ & $29$ & $42$ & $31$ & $16$ & $128$ \\
\hline
\end{tabular}
\caption{Table of the values of $a_1(n,k)$. The final column gives the total number $\sum_k a_1(n,k)$.}\label{tab:a1}
\end{table}

\begin{remark}
In the special case $c = 1$, we obtain
\[
a_1(n,k) = \sum_{j=0}^k \binom{n-k-1}{j},
\]
see Table~\ref{tab:a1}. These partial sums of binomial coefficients are the entries of \emph{Bernoulli's triangle} (see \cite[A008949]{OEIS}) and appear famously as numbers of regions in a general-position hyperplane arrangement \cite[Proposition 2.4]{Stanley-IHA}. The terms in the sum have a combinatorial interpretation again (compare Remark~\ref{rem:comb_int}). The only difference to the unrestricted case is that the up-normal representative cannot have an initial segment of $D$'s and can thus be written as $(UD)^{r_1}U\, (UD)^{r_2}U\, \cdots\, (UD)^{r_h}U\, D^b$ for some nonnegative integers $b,h$ and $r_1,r_2,\ldots,r_h$. The numbers $r_1+1,r_2+1,\ldots,r_h+1$, i.e., the multiplicities in the multiset of heights of NE-steps, 
form a composition of $n-k$ into $h$ positive integers. Since there are $\binom{n-k-1}{h-1}$ such compositions, the total number of equivalence classes must be
\begin{equation*}
\sum_{h=n-2k}^{n-k} \binom{n-k-1}{h-1} = \sum_{j=0}^k \binom{n-k-1}{j}.
\end{equation*}
In this special case, we also have a simple generating function that is similar to~\eqref{eq:secondgf}. We now have
\begin{equation}\label{eq:secondgf2}
A_1(x,t) = \sum_{n \geq 0} \sum_{0 \leq k \leq n/2} a_1(n,k)t^k x^n = \frac{(1-tx^2)^2}{(1-x-tx^2)(1-2tx^2)}.
\end{equation}
This is proved in the same way as~\eqref{eq:secondgf}. First, since $a_1(n,k)$ satisfies the same recursion as $a(n,k)$, one obtains
\begin{equation*}
A_1(x,t) = \frac{1-tx^2}{1-x-tx^2} \sum_{k \geq 0} a_1(2k,k) t^k x^{2k}
\end{equation*}
in the same way as~\eqref{eq:Axt-intermediate}. Now,
\begin{equation*}
    a_1(2k,k) = \sum_{j=0}^k \binom{k-1}{j} = 2^{k-1}
\end{equation*}
for $k \geq 1$ and $a_1(0,0) = 1$, thus
\begin{equation*}
\sum_{k \geq 0} a_1(2k,k) t^k x^{2k} = \frac{1-tx^2}{1-2tx^2},
\end{equation*}
and~\eqref{eq:secondgf2} follows.

Furthermore, we also have an explicit formula for the total number of equivalence classes in this case: the number of equivalence classes of words of length $n > 0$ for which every prefix has at least as many $U$'s as $D$'s is precisely
    \[F_{n+2} - 2^{\floor{(n-1)/2}},\]
which is \cite[A079289]{OEIS}. The final column in Table~\ref{tab:a1} gives the values of this sequence up to $n=10$. This follows e.g.~by plugging $t=1$ into the generating function, in the same way as the formula in Corollary~\ref{cor:total_number}.
\end{remark}


\begin{remark}
In the special case $c=2$, the sum disappears from~\eqref{eq:explicit_general_c}, and we obtain the remarkably simple formula
\[a_2(n,k) = \binom{n-k-1}{k}.\]
There is a connection to the famous ballot problem: consider any up-normal word in the submonoid $\M_2$ that consists of $k$ $D$'s and $n-k$ $U$'s. Remove the last $U$ and all $D$'s that follow. Moreover, replace every occurrence of $UD$ by a single $D$. The result is a word in $\M_1$, i.e., a $1$-Dyck word (every prefix contains at least as many $U$'s as $D$'s) of length $n-k-1$ with at most $k$ $D$'s. Conversely, any $1$-Dyck word of length $n-k-1$ with at most $k$ many $D$'s can be turned into an up-normal word in $\M_2$ with $k$ $D$'s and $n-k$ $U$'s: letting $r$ denote the number of $D$'s ($r \leq k$), replace every $D$ by $UD$ and add $UD^{k-r}$ at the end. So we have a bijection between equivalence classes in $\M_2$ and $1$-Dyck words of length $n-k-1$ with at most $k$ $D$'s. Since it is well-known that there are $\binom{\ell}{j} - \binom{\ell}{j-1}$ many $1$-Dyck words of length $\ell$ with exactly $j$ many $D$'s ($1 \leq j \leq \frac{\ell}2$), it follows that
\[
a_2(n,k) = 1 + \sum_{j=1}^k \Big( \binom{n-k-1}{j} - \binom{n-k-1}{j-1} \Big) = \binom{n-k-1}{k}.
\]
See Table~\ref{tab:a2} for some values of $a_2(n,k)$. 
\end{remark}

\begin{table}[ht]
\centering
\begin{tabular}{ |c|cccc|c|} 
 \hline
 \diagbox[width=10mm]{$n$}{$k$} & 0 & 1 & 2 & 3 & $\sum$ \\ 
 \hline
$1$ & $1$ & & & & $1$ \\
$2$ & $1$ & & & & $1$ \\
$3$ & $1$ & $1$ & & & $2$ \\
$4$ & $1$ & $2$ & & & $3$ \\
$5$ & $1$ & $3$ & & & $4$ \\
$6$ & $1$ & $4$ & $3$ & & $8$ \\ 
$7$ & $1$ & $5$ & $6$ & & $12$ \\
$8$ & $1$ & $6$ & $10$ & & $17$ \\ 
$9$ & $1$ & $7$ & $15$ & $10$ & $33$ \\
$10$ & $1$ & $8$ & $21$ & $20$ & $50$ \\
\hline
\end{tabular}
\caption{Table of the values of $a_2(n,k)$. The final column gives the total number $\sum_k a_2(n,k)$.}\label{tab:a2}
\end{table}

\subsection{The size of an equivalence class}

Given an equivalence relation on a finite set, its equivalence classes are not
the only thing that can be counted. One can also ask how large the equivalence
classes are. The following theorem answers this question.

\begin{theorem}
Let $w\in\M$ be a word with NE-height polynomial $H_{\NE}\tup{w,z} = \sum\limits_{i\in \ZZ} a_{i}z^{i}$ and SE-height polynomial $H_{\SE}\tup{w,z} = \sum\limits_{i\in \ZZ} b_{i}z^{i}$.
(Note that $a_i = b_i = 0$ for all but finitely many $i$.)
Then, the size of the equivalence class containing $w$ (that is, the number of words $u\in\M$ satisfying $u \balsim w$) is
\begin{align}
&\prod_{i \geq 0} \binom{a_i + b_{i+2} - 1}{b_{i+2}} \binom{b_{-i} + a_{-i-2} - 1}{a_{-i-2}} \nonumber \\
&\qquad\times \begin{cases} \binom{a_0+b_0}{a_0}, & \text{ if $w$ is balanced;} \\
\binom{a_0+b_0-1}{b_0}, & \text{ if $w$ is rising and non-balanced;} \\
\binom{a_0+b_0-1}{a_0}, & \text{ if $w$ is falling and non-balanced.} \\
\end{cases}
\label{eq:size-of-eqclass}
\end{align}
\end{theorem}

\begin{proof}
We prove the formula by considering the associated standard paths. Then $a_i$ is the number of NE-arcs from height $i$ to height $i+1$, and $b_i$ is the number of SE-arcs from height $i$ to height $i-1$. Our aim is to show that the number of standard paths, given all $a_i$ and $b_i$, is given by the formula~\eqref{eq:size-of-eqclass}.

We first consider the special case that $w$ is a $1$-Dyck word, i.e., a word whose prefixes are all rising, so that the standard path $\pp$ stays above the $x$-axis. In this case, $a_i = 0$ whenever $i < 0$, and $b_i = 0$ whenever $i \leq 0$, so the formula reduces to 
\[
\prod_{i \geq 0} \binom{a_i + b_{i+2} - 1}{b_{i+2}}.
\]
We use induction on the maximum height $d$ of vertices in the standard path. For $d=0$, the word and its associated path are empty, so the statement becomes trivial.

Now we proceed with the induction step. Consider only the part $\pp'$ of the path that lies above the line $y=1$ (after removing all gaps, see Figure~\ref{fig:eq-class-1}). By the induction hypothesis, the number of possibilities for this path is
\[
\prod_{i \geq 1} \binom{a_i + b_{i+2} - 1}{b_{i+2}}.
\]
Given $\pp'$, in order to obtain a feasible path $\pp$, we always have to add an NE-step at the beginning, an SE-step at the end if $a_0 = b_1$ (so that the path ends at height $0$), and insert a total of $a_0-1$ copies of an SE-step followed by an NE-step at vertices of $\pp'$ that lie at height $0$. There are $b_2+1$ such places (at the beginning and after each of the $b_2$ SE-steps that end at height $0$), so the possibilities for $\pp$, given $\pp'$, correspond to the weak compositions of $a_0-1$ into $b_2+1$ nonnegative integers. Since there are $\binom{a_0+b_2-1}{b_2}$ such compositions, the desired formula follows, completing the induction.

Now we consider the general case: every diagonal path can be decomposed into the part above the $x$-axis and the part below the $x$-axis (see Figure~\ref{fig:eq-class-2}). The number of possibilities for these two parts is
\[
\prod_{i \geq 0} \binom{a_i + b_{i+2} - 1}{b_{i+2}} \binom{b_{-i} + a_{-i-2} - 1}{a_{-i-2}}
\]
in view of what has already been shown. It remains to multiply by the number of ways to combine them: each time the path is at height $0$ (but not completed yet), we have to decide whether to continue with an NE-step (thus a piece of the path that lies above the $x$-axis) or an SE-step (thus a piece of the path that lies below the $x$-axis). The only exception is the final return to the $x$-axis in the non-balanced case: if the path ends above the $x$-axis, the final step from the $x$-axis must be an NE-step; if it ends below the $x$-axis, it must be an SE-step.

The total number of steps to be chosen in this way is $a_0+b_0$. In the balanced case, we have $\binom{a_0+b_0}{a_0}$ possibilities. In the non-balanced case, the number of choices is $\binom{a_0+b_0-1}{b_0}$ (rising) or $\binom{a_0+b_0-1}{a_0}$ (falling),
respectively. Combining this with the number of possibilities for the two parts above and below the $x$-axis, we reach the desired formula.
\end{proof}

\begin{figure}[htbp]
\begin{center}
\begin{tikzpicture}[scale=0.6]
\draw[help lines] (-0.2, -0.2) grid (12.2, 3.2);
\draw[help lines] (13.8, 0.8) grid (20.2, 3.2);
\draw [line width=2pt] (-0.2, 0) -- (12.2, 0);
\draw [line width=2pt] (-0.2, 1) -- (12.2, 1);
\draw
[line width=2pt,-stealth,dashed,red] (0, 0) edge (1, 1) (1, 1) edge (2, 0) (2, 0) -- (3, 1);
\draw
[line width=2pt,-stealth,blue] (3, 1) edge (4, 2) (4, 2) edge (5, 3) (5, 3) edge (6, 2) (6, 2) -- (7, 1);
\draw
[line width=2pt,-stealth,dashed,red] (7, 1) edge (8, 0) (8, 0) -- (9, 1);
\draw
[line width=2pt,-stealth,blue] (9, 1) edge (10, 2) (10, 2) -- (11, 1);
\draw
[line width=2pt,-stealth,dashed,red] (11, 1) -- (12, 0);
\draw
[line width=2pt,-stealth,blue] (14, 1) edge (15, 2) (15, 2) edge (16, 3) (16, 3) edge (17, 2) (17, 2) edge (18, 1) (18, 1) edge (19, 2) (19, 2) -- (20, 1);
\node at (17,0) {$\pp'$};
\end{tikzpicture}
\end{center}
\caption{Decomposition into the part above the line $y=1$ (blue, solid) and the part between $y=0$ and $y=1$ (red, dashed).}
\label{fig:eq-class-1}
\end{figure}

\begin{figure}[htbp]
\begin{center}
\begin{tikzpicture}[scale=0.6]
\draw[help lines] (-0.2, -2.2) grid (14.2, 2.2);
\draw[help lines] (15.8, 0.8) grid (22.2, 3.2);
\draw[help lines] (15.8, -0.8) grid (24.2, -3.2);
\draw [line width=2pt] (-0.2, 0) -- (14.2, 0);
\draw
[line width=2pt,-stealth,dashed,red] (0, 0) edge (1, -1) (1, -1) -- (2, 0);
\draw
[line width=2pt,-stealth,blue] (2, 0) edge (3, 1) (3, 1) edge (4, 2) (4, 2) edge (5, 1) (5, 1) -- (6, 0);
\draw
[line width=2pt,-stealth,dashed,red] (6, 0) edge (7, -1) (7, -1) edge (8, -2) (8, -2) edge (9, -1) (9, -1) edge (10, 0) (10, 0) edge (11,-1) (11,-1) -- (12, 0);
\draw
[line width=2pt,-stealth,blue] (12, 0) edge (13, 1) (13, 1) -- (14, 0);
\draw
[line width=2pt,-stealth,blue] (16, 1) edge (17, 2) (17, 2) edge (18, 3) (18, 3) edge (19, 2) (19, 2) edge (20, 1) (20, 1) edge (21, 2) (21, 2) -- (22, 1);
\draw
[line width=2pt,-stealth,dashed,red]
(16, -1) edge (17, -2) (17, -2) edge (18, -1) (18, -1) edge (19, -2) (19, -2) edge (20, -3) (20, -3) edge (21, -2) (21, -2) edge (22,-1) (22,-1) edge (23,-2) (23,-2) -- (24, -1);
\end{tikzpicture}
\end{center}
\caption{Decomposition into the part above (blue, solid) and below (red, dashed) the $x$-axis.}
\label{fig:eq-class-2}
\end{figure}

\begin{remark}
The special case of balanced $1$-Dyck words (balanced words for which every prefix is rising) appears in different places in the literature. See \cite[Theorem 6]{GJW75} (in the context of rook theory) or \cite[Proposition 3A and 3B]{Flajol80} (in the context of continued fractions and lattice paths).    
\end{remark}

\section{Bond percolation}\label{sec:bondperc}

The diagonal lattice interpretation of the Weyl algebra brings forward an intriguing connection with bond percolation on a directed square lattice.
In this section we explore this connection in depth.

\begin{figure}[htbp]
\begin{center}
\begin{tikzpicture}

\draw[fill=black] (0,0) circle (.3ex);
\draw[fill=black] (1,1) circle (.3ex);
\draw[fill=black] (1,-1) circle (.3ex);
\draw[fill=black] (2,2) circle (.3ex);
\draw[fill=black] (2,-2) circle (.3ex);
\draw[fill=black] (2,0) circle (.3ex);
\draw[fill=black] (4,0) circle (.3ex);
\draw[fill=black] (0,2) circle (.3ex);
\draw[fill=black] (0,4) circle (.3ex);
\draw[fill=black] (3,1) circle (.3ex);
\draw[fill=black] (3,-1) circle (.3ex);

\node at (-0.3,-0.3) {(0,0)};
\node at (2,-0.3) {2};
\node at (4,-0.3) {4};
\node at (-0.3,2) {2};
\node at (-0.3,4) {4};
\node at (6.1, -0.2) {$t$};
\node at (-0.2, 5.9) {$x$};

\draw[thick, ->-] (0, 0)--(1, 1);
\draw[thick, ->-] (0, 0)--(1, -1);
\draw[thick, ->-] (1, -1)--(2, 0);
\draw[thick, ->-] (1, -1)--(2, -2);
\draw[thick, ->-] (1, 1)--(2, 2);
\draw[thick, ->-] (2, 2)--(3, 3);
\draw[thick, ->-] (2, 2)--(3, 1);
\draw[thick, ->-] (3, 1)--(4, 0);
\draw[thick, ->-] (1, 1)--(2, 0);
\draw[thick, ->-] (2, 0)--(3, 1);
\draw[thick, ->-] (2, 0)--(3, -1);
\draw[thick, ->-] (3, 1)--(4, 2);
\draw[thick, ->-] (2, -2)--(3, -1);
\draw[thick, ->-] (2, -2)--(3, -3);
\draw[thick, ->-] (3, -1)--(4, 0);
\draw[thick, ->-] (3, -1)--(4, -2);
\draw[thick, ->-] (4, 0)--(5, 1);
\draw[thick, ->-] (4, 0)--(5, -1);
\draw[thick, ->] (0, 0)--(6, 0);
\draw[thick, ->] (0, 0)--(0, 6);

\end{tikzpicture}
\end{center}
\caption{Acyclic directed square lattice.}\label{fig:illustration1}
\end{figure}

Percolation is one of the fundamental problems in statistical physics \cite{Grimme99} \cite{StaAha94}, and is of great theoretical interest in its own right as well as being applicable to a wide variety of
problems in physics, biology, chemistry, and many other areas of science. Bond percolation, the phenomenon of interest here, was introduced in the mathematics literature by Broadbent and Hammersley in 1957 \cite{BroHam57}, and has been studied extensively by mathematicians since then.

The prototype setting for bond percolation is a directed square lattice, whose vertices (called \emph{sites} in this context) are the points in the Cartesian $t$-$x$-plane with integer coordinates such that $t \geq 0$ and $t+x$ is even. See Figure \ref{fig:illustration1}.
Here $t$ is commonly thought of as the time (or stage) of the percolation process.
We regard the lattice as originally consisting of dry sites except for the origin, which is the source of fluid and wet at stage $0$.
There are two \emph{bonds} (i.e., arcs) leading from each site $(t, x)$; they terminate at the sites $(t+1, x+1)$ and $(t+1, x-1)$. (In our language, they are the NE-arcs and the SE-arcs.)
All bonds have probability $p$ of being open to the passage of fluid and probability $1-p$ of being closed.
Fluid flows from a wet site along an unblocked bond to wet another site (in the forward direction, i.e., from source to target).
Thus a site is wetted if there is a path of unblocked directed bonds (and wet sites) from the origin to that site.
See Figure \ref{fig:illustration2} for an illustration of all possible scenarios of the percolation process from the origin $(0, 0)$ to the site $(2, 0)$ in two time steps.
\emph{Clusters} are sets of connected bonds, where two bonds are said to be \emph{adjacent} if they have a vertex in common.
Sometimes, a \emph{wall} parallel to the $t$-axis (``growth direction'') is present to restrict the lateral growth of the percolation clusters.
In particular, we consider the situation where the $t$-axis itself is the wall, so that the bonds leading to sites with $x<0$ are always closed \cite{EGJT96}. See Figure \ref{fig:illustration3}.

\begin{figure}[htbp]
\begin{center}
\begin{tikzpicture}
\draw[fill=black] (0,0) circle (.3ex);
\draw[fill=black] (1,1) circle (.3ex);
\draw[fill=black] (1,-1) circle (.3ex);
\draw[fill=black] (2,0) circle (.3ex);

\draw[thick,  ->-] (0, 0)--(1, 1);
\draw[thick,  ->-] (0, 0)--(1, -1);
\draw[thick,  ->-] (1, 1)--(2, 0);
\draw[thick,  ->-] (1, -1)--(2, 0);
\end{tikzpicture}
\qquad
\begin{tikzpicture}
\draw[fill=black] (0,0) circle (.3ex);
\draw (1,1) circle (.3ex);
\draw[fill=black] (1,-1) circle (.3ex);
\draw[fill=black] (2,0) circle (.3ex);

\draw[thick, dashed] (0, 0)--(1, 1);
\draw[thick,  ->-] (0, 0)--(1, -1);
\draw[thick,  ->-] (1, 1)--(2, 0);
\draw[thick,  ->-] (1, -1)--(2, 0);
\end{tikzpicture}
\qquad
\begin{tikzpicture}
\draw[fill=black] (0,0) circle (.3ex);
\draw[fill=black] (1,1) circle (.3ex);
\draw[fill=black] (1,-1) circle (.3ex);
\draw[fill=black] (2,0) circle (.3ex);

\draw[thick,  ->-] (0, 0)--(1, 1);
\draw[thick,  ->-] (0, 0)--(1, -1);
\draw[thick, dashed] (1, 1)--(2, 0);
\draw[thick,  ->-] (1, -1)--(2, 0);
\end{tikzpicture}
\qquad
\begin{tikzpicture}
\draw[fill=black] (0,0) circle (.3ex);
\draw[fill=black] (1,1) circle (.3ex);
\draw[fill=black] (1,-1) circle (.3ex);
\draw[fill=black] (2,0) circle (.3ex);

\draw[thick,  ->-] (0, 0)--(1, 1);
\draw[thick,  ->-] (0, 0)--(1, -1);
\draw[thick,  ->-] (1, 1)--(2, 0);
\draw[thick, dashed] (1, -1)--(2, 0);
\end{tikzpicture}

\vskip.1truein

\begin{tikzpicture}
\draw[fill=black] (0,0) circle (.3ex);
\draw[fill=black] (1,1) circle (.3ex);
\draw (1,-1) circle (.3ex);
\draw[fill=black] (2,0) circle (.3ex);

\draw[thick,  ->-] (0, 0)--(1, 1);
\draw[thick, dashed] (0, 0)--(1, -1);
\draw[thick,  ->-] (1, 1)--(2, 0);
\draw[thick,  ->-] (1, -1)--(2, 0);
\end{tikzpicture}
\qquad
\begin{tikzpicture}
\draw[fill=black] (0,0) circle (.3ex);
\draw (1,1) circle (.3ex);
\draw (1,-1) circle (.3ex);
\draw (2,0) circle (.3ex);

\draw[thick, dashed] (0, 0)--(1, 1);
\draw[thick, dashed] (0, 0)--(1, -1);
\draw[thick,  ->-] (1, 1)--(2, 0);
\draw[thick,  ->-] (1, -1)--(2, 0);
\end{tikzpicture}
\qquad
\begin{tikzpicture}
\draw[fill=black] (0,0) circle (.3ex);
\draw (1,1) circle (.3ex);
\draw[fill=black] (1,-1) circle (.3ex);
\draw[fill=black] (2,0) circle (.3ex);

\draw[thick, dashed] (0, 0)--(1, 1);
\draw[thick,  ->-] (0, 0)--(1, -1);
\draw[thick, dashed] (1, 1)--(2, 0);
\draw[thick,  ->-] (1, -1)--(2, 0);
\end{tikzpicture}
\qquad
\begin{tikzpicture}
\draw[fill=black] (0,0) circle (.3ex);
\draw[fill=black] (1,1) circle (.3ex);
\draw[fill=black] (1,-1) circle (.3ex);
\draw (2,0) circle (.3ex);

\draw[thick,  ->-] (0, 0)--(1, 1);
\draw[thick,  ->-] (0, 0)--(1, -1);
\draw[thick, dashed] (1, 1)--(2, 0);
\draw[thick, dashed] (1, -1)--(2, 0);
\end{tikzpicture}

\vskip.1truein

\begin{tikzpicture}
\draw[fill=black] (0,0) circle (.3ex);
\draw[fill=black] (1,1) circle (.3ex);
\draw (1,-1) circle (.3ex);
\draw[fill=black] (2,0) circle (.3ex);

\draw[thick,  ->-] (0, 0)--(1, 1);
\draw[thick, dashed] (0, 0)--(1, -1);
\draw[thick,  ->-] (1, 1)--(2, 0);
\draw[thick, dashed] (1, -1)--(2, 0);
\end{tikzpicture}
\qquad
\begin{tikzpicture}
\draw[fill=black] (0,0) circle (.3ex);
\draw (1,1) circle (.3ex);
\draw[fill=black] (1,-1) circle (.3ex);
\draw (2,0) circle (.3ex);

\draw[thick, dashed] (0, 0)--(1, 1);
\draw[thick,  ->-] (0, 0)--(1, -1);
\draw[thick,  ->-] (1, 1)--(2, 0);
\draw[thick, dashed] (1, -1)--(2, 0);
\end{tikzpicture}
\qquad
\begin{tikzpicture}
\draw[fill=black] (0,0) circle (.3ex);
\draw[fill=black] (1,1) circle (.3ex);
\draw (1,-1) circle (.3ex);
\draw (2,0) circle (.3ex);

\draw[thick,  ->-] (0, 0)--(1, 1);
\draw[thick, dashed] (0, 0)--(1, -1);
\draw[thick, dashed] (1, 1)--(2, 0);
\draw[thick,  ->-] (1, -1)--(2, 0);
\end{tikzpicture}
\qquad
\begin{tikzpicture}
\draw[fill=black] (0,0) circle (.3ex);
\draw[fill=black] (1,1) circle (.3ex);
\draw (1,-1) circle (.3ex);
\draw (2,0) circle (.3ex);

\draw[thick,  ->-] (0, 0)--(1, 1);
\draw[thick, dashed] (0, 0)--(1, -1);
\draw[thick, dashed] (1, 1)--(2, 0);
\draw[thick, dashed] (1, -1)--(2, 0);
\end{tikzpicture}

\vskip.1truein

\begin{tikzpicture}
\draw[fill=black] (0,0) circle (.3ex);
\draw (1,1) circle (.3ex);
\draw (1,-1) circle (.3ex);
\draw (2,0) circle (.3ex);

\draw[thick, dashed] (0, 0)--(1, 1);
\draw[thick, dashed] (0, 0)--(1, -1);
\draw[thick,  ->-] (1, 1)--(2, 0);
\draw[thick, dashed] (1, -1)--(2, 0);
\end{tikzpicture}
\qquad
\begin{tikzpicture}
\draw[fill=black] (0,0) circle (.3ex);
\draw (1,1) circle (.3ex);
\draw (1,-1) circle (.3ex);
\draw (2,0) circle (.3ex);

\draw[thick, dashed] (0, 0)--(1, 1);
\draw[thick, dashed] (0, 0)--(1, -1);
\draw[thick, dashed] (1, 1)--(2, 0);
\draw[thick,  ->-] (1, -1)--(2, 0);
\end{tikzpicture}
\qquad
\begin{tikzpicture}
\draw[fill=black] (0,0) circle (.3ex);
\draw (1,1) circle (.3ex);
\draw[fill=black] (1,-1) circle (.3ex);
\draw (2,0) circle (.3ex);

\draw[thick, dashed] (0, 0)--(1, 1);
\draw[thick,  ->-] (0, 0)--(1, -1);
\draw[thick, dashed] (1, 1)--(2, 0);
\draw[thick, dashed] (1, -1)--(2, 0);
\end{tikzpicture}
\qquad
\begin{tikzpicture}
\draw[fill=black] (0,0) circle (.3ex);
\draw (1,1) circle (.3ex);
\draw (1,-1) circle (.3ex);
\draw (2,0) circle (.3ex);

\draw[thick, dashed] (0, 0)--(1, 1);
\draw[thick, dashed] (0, 0)--(1, -1);
\draw[thick, dashed] (1, 1)--(2, 0);
\draw[thick, dashed] (1, -1)--(2, 0);
\end{tikzpicture}
\end{center}
\caption{Directed bond percolation from the origin to the site $(2, 0)$. Open (closed) bonds are indicated by solid (dashed) lines. Filled (hollow) circles denote wet (dry) sites.}\label{fig:illustration2}
\end{figure}

\begin{figure}[htbp]
\begin{center}
\begin{tikzpicture}

\draw[fill=black] (0,0) circle (.3ex);
\draw[fill=black] (1,1) circle (.3ex);
\draw[fill=black] (2,2) circle (.3ex);
\draw[fill=black] (2,0) circle (.3ex);
\draw[fill=black] (4,0) circle (.3ex);
\draw[fill=black] (0,2) circle (.3ex);
\draw[fill=black] (0,4) circle (.3ex);
\draw[fill=black] (3,1) circle (.3ex);

\node at (-0.3,-0.3) {(0,0)};
\node at (2,-0.3) {2};
\node at (4,-0.3) {4};
\node at (-0.3,2) {2};
\node at (-0.3,4) {4};
\node at (6.1, -0.2) {$t$};
\node at (-0.2, 5.9) {$x$};

\draw[thick, ->-] (0, 0)--(1, 1);
\draw[thick, ->-] (1, 1)--(2, 2);
\draw[thick, ->-] (2, 2)--(3, 3);
\draw[thick, ->-] (2, 2)--(3, 1);
\draw[thick, ->-] (3, 1)--(4, 0);
\draw[thick, ->-] (1, 1)--(2, 0);
\draw[thick, ->-] (2, 0)--(3, 1);
\draw[thick, ->-] (3, 1)--(4, 2);
\draw[thick, ->-] (4, 0)--(5, 1);
\draw[red, thick, ->] (0, 0)--(6, 0);
\draw[thick, ->] (0, 0)--(0, 6);

\end{tikzpicture}
\end{center}
\caption{Acyclic directed square lattice with a wall at $x=0$.}\label{fig:illustration3}
\end{figure}

The mean size $S(p)$ of the clusters is a quantity that has captured a lot of interest:
\begin{equation*}
S(p) = \sum_{\text{sites }(t, x)} C(t, x; p),
\end{equation*}
where $C(t, x; p)$ is the probability that there is an open path from the origin to the site $(t, x)$. For example, we may readily calculate that $C(2, 0; p)=2p^2-p^4$ for percolation (without a wall) from Figure \ref{fig:illustration2}. We may also easily calculate that $C(2, 0; p)=p^2$ for percolation (with a wall). For large $t$ and $x$, however, calculating the probability $C(t, x; p)$ becomes a tedious matter and is usually done with the help of a computer. There is a vast body of literature in statistical physics regarding the implementation of the computational procedure, commonly referred to as a \emph{transfer matrix method}. See \cite{Ble77} for the setup in physics. The main idea is the following: The state of time step $t$ is a specification of which sites in column $t$ of the directed square lattice are wet and which sites are dry. Essentially the state vector of a given column is completely determined by that of the previous column and only one state vector need be held in the computer at any stage and all other state vectors overwritten, although some care is necessary for the execution.

The low-density series expansion of $S(p)$ for bond percolation (both with and without a wall) may be performed to order $p^n$ (for varying values of $n$) by calculating $C(t, x; p)$ to order $p^n$ of all sites which may be reached in a walk of $n$ or fewer steps from the origin, i.e., summing up $C(t, x; p)$ of all those reachable sites before or at column $n$. So for example, when a wall is not present, to obtain $S(p)$ to order $p$, we compute
\begin{equation*}
C(0, 0; p)+C(1, 1; p)+C(1, -1; p)=1+p+p=1+2p.
\end{equation*}
To obtain $S(p)$ to order $p^2$, we compute
\begin{multline*}
C(0, 0;p)+C(1, 1; p)+C(1, -1; p)+C(2, 2; p)+C(2, 0; p)+C(2, -2; p)\\
=1+p+p+p^2+(2p^2-p^4)+p^2=1+2p+4p^2-p^4,
\end{multline*}
which gives $1+2p+4p^2$ when kept to order $p^2$. And when a wall is present, to obtain $S(p)$ to order $p$, we compute
\begin{equation*}
C(0, 0; p)+C(1, 1; p)=1+p.
\end{equation*}
To obtain $S(p)$ to order $p^2$, we compute
\begin{equation*}
C(0, 0;p)+C(1, 1; p)+C(2, 2; p)+C(2, 0; p)=1+p+p^2+p^2=1+p+2p^2.
\end{equation*}
The coefficients of the series expansion of $S(p)$ (without a wall/with a wall) are respectively given in \cite[A006727]{OEIS} and \cite[A056532]{OEIS}. We note in particular that negative terms appear starting from $n=50$ for \cite[A006727]{OEIS} and from $n=39$ for \cite[A056532]{OEIS}, as for large $n$, the negative higher-order terms from columns before column $n$ might dominate the positive $p^n$ term from column $n$.

Compared with our results from earlier, we see that $a(n):=\sum_k a(n, k)$ (see Table \ref{tab:a_tot}) gives the total number of equivalence classes of diagonal paths from the origin to all sites in column $n$ on a directed square lattice without a wall in Figure \ref{fig:illustration1} while $a_1(n):=\sum_k a_1(n, k)$ (see Table \ref{tab:a1}) gives the analogous number for a directed square lattice with a wall in Figure \ref{fig:illustration3}.
It is thus not entirely surprising that $a(n)$ agrees with \cite[A006727]{OEIS} up to $n=11$, and that $a_1(n)$ agrees with \cite[A056532]{OEIS} up to $n=8$. After all, recall from Lemma \ref{lemma:recursion} that the $a(n,k)$ satisfy the recursion $a(n, k)=a(n-1, k)+a(n-2, k-1)$. Due to the transfer matrix method that was explained briefly earlier, this recursive feature appears again in calculating the bond percolation on a directed square lattice. To derive $C(t, x; p)$, the probability that there is an open path from the origin to the site $(t, x)$, we only need to keep track of the corresponding $C$ values on the one square to the left of the site $(t, x)$, consisting of sites $(t-2, x), (t-1, x+1), (t-1, x-1), (t, x)$.

Nevertheless, the Weyl algebra problem and the bond percolation problem are different in nature: One is deterministic while the other is probabilistic, and more importantly, the contribution to the $n$th term only comes from the $n$th column for one but might involve some columns before column $n$ for the other. In detail, the coefficient of $p^n$ term in the low-density series expansion of $S(p)$ is calculated by summing up $C(t, x; p)$ of all those reachable sites before or at column $n$, not just at column $n$. (In our earlier calculation for $S(p)$, summing up column $2$ sites contributes a higher-order term $-p^4$ in addition to $4p^2$. For larger $n$, the difference is more evident.)

\section{The rook theory connection}\label{sec:rook}

Theorem \ref{thm.DU-eqs2} classifies equalities between products of $D$'s and $U$'s in the Weyl algebra
$\W$. Another approach to this classification problem is to expand
any such product in one of the bases $\left(  D^{i}%
U^{j}\right)  _{i,j\in\NN}$ and $\left(  U^{j}D^{i}\right)
_{i,j\in\NN}$ of $\W$; the uniqueness of this expansion then
allows us to compare two such products by comparing their
respective coefficients.

It turns out that this expansion can be done in explicit combinatorial terms
using \emph{rook theory}. We do not give a detailed introduction to this
subject (see \cite{BCHR11} and \cite[\S 2.4.4]{ManSch16} for that), but
quickly recall the basics we need.

A \emph{cell}  means a pair $\left(  i,j\right)  $ of two positive
integers. Each cell $\left(  i,j\right)  $ will be drawn as a $1\times
1$-square, situated in the Cartesian plane with center at the point $\left(
i,j\right)  $, with its sides parallel to the axes. A \emph{board} means
a finite set of cells. For instance, the board $\left\{  \left(  1,1\right)
,\ \left(  2,2\right)  ,\ \left(  3,1\right)  ,\ \left(  4,2\right)
,\ \left(  6,1\right)  ,\ \left(  6,2\right)  \right\}  $ looks as
follows:
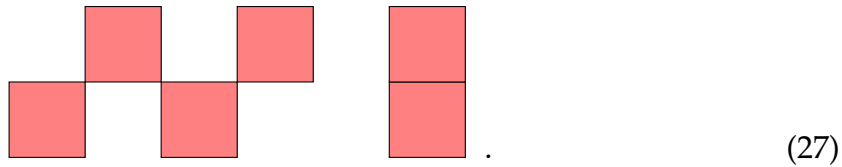
\begin{equation}
\begin{tikzpicture}
\draw[fill=red!50] (0, 0) rectangle (1, 1);
\draw[fill=red!50] (1, 1) rectangle (2, 2);
\draw[fill=red!50] (2, 0) rectangle (3, 1);
\draw[fill=red!50] (3, 1) rectangle (4, 2);
\draw[fill=red!50] (5, 0) rectangle (6, 1);
\draw[fill=red!50] (5, 1) rectangle (6, 2);
\end{tikzpicture}
\ \ .\label{eq.rook.board1}
\end{equation}

A \emph{rook placement} of a board $B$ is a subset $S$ of $B$ such that no two
cells in $S$ lie in the same row or column. If $B$ is a board, and if
$k\in\NN$, then the \emph{rook number} $r_{k}\left(  B\right)  $ is the
number of $k$-element rook placements of $B$. For instance, if $B$ is the board in
(\ref{eq.rook.board1}), then its rook numbers are $r_{0}\left(  B\right)  =1$
and $r_{1}\left(  B\right)  =\left\vert B\right\vert =6$ and $r_{2}\left(
B\right)  =8$ and $r_{k}\left(  B\right)  =0$ for all $k>2$.

Two boards $B$ and $C$ are said to be \emph{rook-equivalent} if they share the
same rook numbers (i.e., if $r_{k}\left(  B\right)  =r_{k}\left(  C\right)  $
for all $k\in\NN$).

If $w$ is a word in $\M$, then the \emph{Ferrers board} $B_{w}$ is a
special board defined as follows: It is a contiguous set of cells, whose
bottom and right boundaries are straight lines, whereas the rest of its
boundary is a jagged path that (when walked from southwest to northeast) takes
a north-step for each $D$ in $w$ and an east-step for each $U$ in $w$ (reading
the word $w$ from left to right). For instance, if $w=UDDUDUUDUD$, then $B_{w}$
looks as follows:
\[
\begin{tikzpicture}
\draw[fill=red!50] (1, 0) rectangle (2, 1);
\draw[fill=red!50] (1, 1) rectangle (2, 2);
\draw[fill=red!50] (2, 0) rectangle (3, 1);
\draw[fill=red!50] (2, 1) rectangle (3, 2);
\draw[fill=red!50] (2, 2) rectangle (3, 3);
\draw[fill=red!50] (3, 0) rectangle (4, 1);
\draw[fill=red!50] (3, 1) rectangle (4, 2);
\draw[fill=red!50] (3, 2) rectangle (4, 3);
\draw[fill=red!50] (4, 0) rectangle (5, 1);
\draw[fill=red!50] (4, 1) rectangle (5, 2);
\draw[fill=red!50] (4, 2) rectangle (5, 3);
\draw[fill=red!50] (4, 3) rectangle (5, 4);
\draw (0, 0) -- (1, 0);
\draw (5, 4) -- (5, 5);
\node(x1) at (0.5, -0.25) {$U$};
\node(x2) at (0.75, 0.5) {$D$};
\node(x3) at (0.75, 1.5) {$D$};
\node(x4) at (1.5, 1.75) {$U$};
\node(x5) at (1.75, 2.5) {$D$};
\node(x6) at (2.5, 2.75) {$U$};
\node(x7) at (3.5, 2.75) {$U$};
\node(x8) at (3.75, 3.5) {$D$};
\node(x9) at (4.5, 3.75) {$U$};
\node(x10) at (4.75, 4.5) {$D$};
\end{tikzpicture}
\]
(where the $D$ and $U$ labels are signaling the correspondence between
the letters of $w$ and the steps of the jagged boundary).\footnote{Note that the $U$ at the beginning of $w$, and the $D$ at the end, do not affect the Ferrers board.}

Now a classical result of Navon (originally \cite[\S 2]{Navon73}, but see
\cite[Theorem 20]{BCHR11} or \cite[Theorem 6.11 for $h=1$]{ManSch16} or \cite[Theorem 3.2]{Varvak} for a
modern treatment\footnote{See also \cite{GoF09} for a brief introduction with physics in mind. The sources differ slightly in their notation, but are easily seen to be equivalent (e.g., by reflecting the Ferrers boards across a diagonal line).}) says the following:

\begin{theorem}
\label{thm.navon}
Let $w\in\M$ be any word that contains $n$ many
$D$'s and $m$ many $U$'s. Then, in $\W$, we have
\[
\phi\left(  w\right)
= \sum_{k=0}^{\min\left\{  m,n\right\}  }r_{k}\left(
B_{w}\right)  U^{m-k}D^{n-k}.
\]

\end{theorem}

As a consequence, we obtain the following:

\begin{theorem}
\label{thm.DU-eqs3}
Let $u$ and $v$ be two words in $\M$ that have the
same number of $D$'s and the same number of $U$'s. Then, the following
statements are equivalent:
the statements $\mathcal{S}_{1}$, $\mathcal{S}_{2}$, $\mathcal{S}_{3}$, $\mathcal{S}'_3$, $\mathcal{S}_{4}$, $\mathcal{S}_{5}$ and $\mathcal{S}_6$ from Theorem \ref{thm.DU-eqs2},
and the additional statement

\begin{itemize}
\item $\mathcal{R}_1$: The boards $B_{u}$ and $B_{v}$ are rook-equivalent.
\end{itemize}
\end{theorem}

\begin{proof}
It suffices to prove the equivalence $\mathcal{S}_{1}\Longleftrightarrow
\mathcal{R}_1$.

Let $n$ be the \# of $D$'s in $u$ (or, equivalently, in $v$), and let $m$ be
the \# of $U$'s in $u$ (or, equivalently, in $v$). Then, the board $B_{u}$ has
$\leq n$ nonempty rows (since $u$ has $n$ many $D$'s). Hence, any rook placement of
$B_{u}$ has size $\leq n$ (since otherwise, it would contain two cells in the
same row, by the pigeonhole principle). In other words, $r_{k}\left(
B_{u}\right)  =0$ for all $k>n$. Similarly, $r_{k}\left(  B_{u}\right)  =0$
for all $k>m$. Combining these two observations, we obtain%
\begin{equation}
r_{k}\left(  B_{u}\right)  =0\ \ \ \ \ \ \ \ \ \ \text{for all }k>\min\left\{
m,n\right\}  .\label{pf.thm.DU-eqs3.1}%
\end{equation}
Similarly,%
\begin{equation}
r_{k}\left(  B_{v}\right)  =0\ \ \ \ \ \ \ \ \ \ \text{for all }k>\min\left\{
m,n\right\}  .\label{pf.thm.DU-eqs3.2}%
\end{equation}
Comparing these two equalities, we conclude that%
\begin{equation}
r_{k}\left(  B_{u}\right)  =r_{k}\left(  B_{v}\right)
\ \ \ \ \ \ \ \ \ \ \text{for all }k>\min\left\{  m,n\right\}
.\label{pf.thm.DU-eqs3.1=2}%
\end{equation}

Recall that the family $\left(  U^{j}D^{i}\right)  _{i,j\in\NN}$ is a
basis of $\W$ (by \cite[Proposition 2.7.1 (i)]{Etingof}), therefore
$\kk$-linearly independent.

Now, we have the following chain of equivalences:%
\begin{align*}
\mathcal{S}_{1}\   & \Longleftrightarrow\ \left(  \phi\left(  u\right)
=\phi\left(  v\right)  \right)  \\
& \Longleftrightarrow\ \left(  \sum_{k=0}^{\min\left\{  m,n\right\}  }%
r_{k}\left(  B_{u}\right)  U^{m-k}D^{n-k}=\sum_{k=0}^{\min\left\{
m,n\right\}  }r_{k}\left(  B_{v}\right)  U^{m-k}D^{n-k}\right)  \\
& \ \ \ \ \ \ \ \ \ \ \ \ \ \ \ \ \ \ \ \ \left(  \text{here, we rewrote }%
\phi\left(  u\right)  \text{ and }\phi\left(  v\right)  \text{ using Theorem
\ref{thm.navon}}\right)  \\
& \Longleftrightarrow\ \left(  r_{k}\left(  B_{u}\right)  =r_{k}\left(
B_{v}\right)  \text{ for all }k\in\left\{  0,1,\ldots,\min\left\{
m,n\right\}  \right\}  \right)  \\
& \ \ \ \ \ \ \ \ \ \ \ \ \ \ \ \ \ \ \ \ \left(  \text{since the family
}\left(  U^{j}D^{i}\right)  _{i,j\in\NN}\text{ is }\kk%
\text{-linearly independent}\right)  \\
& \Longleftrightarrow\ \left(  r_{k}\left(  B_{u}\right)  =r_{k}\left(
B_{v}\right)  \text{ for all }k\in\NN\right)
\ \ \ \ \ \ \ \ \ \ \left(  \text{by (\ref{pf.thm.DU-eqs3.1=2})}\right)  \\
& \Longleftrightarrow\ \left(  \text{the boards }B_{u}\text{ and }B_{v}\text{
are rook-equivalent}\right)  \ \Longleftrightarrow\ \mathcal{R}_1.
\end{align*}
This completes the proof of $\mathcal{S}_{1}\Longleftrightarrow\mathcal{R}_1$ and thus the proof of Theorem \ref{thm.DU-eqs3}.
\end{proof}

The implication $\mathcal{R}_1\Longrightarrow\mathcal{S}_{1}$ in Theorem
\ref{thm.DU-eqs3} has been implicitly observed in \cite[bottom of p.
40]{BCHR11}. Note that this implication really requires the assumption about
equal numbers of $D$'s and of $U$'s in Theorem \ref{thm.DU-eqs3}, since (e.g.)
the Ferrers boards $B_{DUUDU}$ and $B_{DDUU}$ are rook-equivalent without
$\phi\left(  DUUDU\right)  $ equalling $\phi\left(  DDUU\right)  $. For an
even starker example, if $w\in\M$ is any word, then the Ferrers
boards $B_{w}$ and $B_{\omega\left(  w\right)  }$ are rook-equivalent (being
each other's reflection across a diagonal), but $\phi\left(  w\right)  $ is
usually not $\phi\left(  \omega\left(  w\right)  \right)  $. We note that this reasoning leads to a new proof of Theorem \ref{thm.phiomega}.

\begin{noncompile}
\begin{proof}[Second proof of Theorem \ref{thm.phiomega}.]
The word $u$ is balanced, i.e., has the same number of $U$'s and of $D$'s. Let
$n$ be this number. Then, the word $\omega\left(  u\right)  $ also has $n$
many $D$'s and many $U$'s. Hence, Theorem \ref{thm.DU-eqs3} (applied to
$v=\omega\left(  u\right)  $) shows that the statements $\mathcal{S}_{1}$,
$\mathcal{S}_{2}$, $\mathcal{S}_{5}$, $\mathcal{S}_6$ and $\mathcal{R}_1$ (for $v=\omega
\left(  u\right)  $) are equivalent. But the Ferrers boards $B_{u}$ and
$B_{\omega\left(  u\right)  }$ are rook-equivalent, since one can be obtained
from the other by a reflection across a diagonal. Thus, statement
$\mathcal{R}_1$ (for $v=\omega\left(  u\right)  $) holds. Hence, statements
$\mathcal{S}_{1}$ and $\mathcal{S}_{5}$ also hold (since the statements
$\mathcal{S}_{1}$, $\mathcal{S}_{2}$, $\mathcal{S}_{5}$, $\mathcal{S}_6$ and $\mathcal{R}_1$
are equivalent). In other words, $\phi\left(  u\right)  =\phi\left(
\omega\left(  u\right)  \right)  $ and $u \balsim \omega\left(  u\right)  $. This proves Theorem \ref{thm.phiomega}.
\end{proof}
\end{noncompile}

Theorem \ref{thm.DU-eqs3} connects our study of the kernel of $\phi$ to a
classical question, namely: When are two Ferrers boards rook-equivalent? A
classical result of Foata and Sch\"{u}tzenberger (\cite[Theorem 6]{FoaSch70}
or \cite[Theorem 7]{BCHR11}) shows that each Ferrers board is rook-equivalent
to a unique \textquotedblleft increasing Ferrers board\textquotedblright.
These \textquotedblleft increasing Ferrers boards\textquotedblright\ are
somewhat similar to our up-normal words, but not quite in bijection, since (as
we said) the rook equivalence of $B_{u}$ and $B_{v}$ implies $\phi\left(
u\right)  =\phi\left(  v\right)  $ only when we know that $u$ and $v$ have the
same number of $U$'s and the same number of $D$'s.

Interestingly, Foata and Sch\"{u}tzenberger have their own kind of moves that
they use to normalize a Ferrers board modulo rook equivalence: the
\textquotedblleft$\left(  k,k^{\prime}\right)  $-transforms\textquotedblright%
\ (see \cite[Definition 8 bis on page 9]{FoaSch70}). These appear to be
close relatives of our balanced flips. \medskip

A recent preprint by Cotardo, Gruica and Ravagnani \cite{CoGrRa23} proves
another set of equivalent criteria for the rook-equivalence of two Ferrers
boards \cite[Corollary 3.2]{CoGrRa23}.
It lists six equivalent conditions, one of which (condition 6) is
rook-equivalence, whereas another (condition 1) is equivalent to our
statement $\mathcal{S}_3$ (albeit this equivalence takes some work to
prove).
The other four conditions are not found in our lists so far.
In the following, we  state two of these four conditions (4 and 5),
as they are rather surprising and reveal an unexpected connection to
the theory of finite fields. (Arguably, at least one of them has been
foreseen, to some extent, in Haglund's \cite{Haglun98}.)

First, we introduce the necessary notations.
For any finite field $F$, any nonnegative integers $n$ and $k$, and any board
$B \subseteq \set{1, 2, \ldots, n}^2$, we define $P_k \tup{B / F}$
to be the number of $n \times n$-matrices $A \in F^{n \times n}$ of
rank $k$ such that all entries of $A$ in cells outside of $B$ are zero.
This is called $P_k \tup{B}$ in \cite[Definition 1]{Haglun98}, where
$F = \mathbb{F}_q$; but we include $F$ in the notation since
$P_k \tup{B / F}$ depends on $F$. It is easy to see, however, that
$P_k \tup{B / F}$ does not depend on $n$,
as long as $n$ is large enough that $B \subseteq \set{1, 2, \ldots, n}^2$.
In \cite{CoGrRa23}, our $P_k \tup{B / F}$ is called
$W_k\tup{\operatorname{Mat}^{n\times m}_q\tup{B}}$, where $F = \mathbb{F}_q$,
and where $n$ and $m$ are chosen large enough that
$B \subseteq \set{1,2,\ldots,n} \times \set{1,2,\ldots,m}$.
Now, we claim the following:

\begin{theorem}
\label{thm.DU-eqs4}
Let $u$ and $v$ be two words in $\M$ that have the
same number of $D$'s and the same number of $U$'s. Then, the following
statements are equivalent:
the statements $\mathcal{S}_{1}$, $\mathcal{S}_{2}$, $\mathcal{S}_{3}$, $\mathcal{S}'_3$, $\mathcal{S}_{4}$, $\mathcal{S}_{5}$ and $\mathcal{S}_6$ from Theorem \ref{thm.DU-eqs2},
the statement $\mathcal{R}_1$ from Theorem \ref{thm.DU-eqs3},
and the following two additional statements:

\begin{itemize}
\item $\mathcal{R}_2$: For any finite field $F$ and any $k\in\NN$, we
have $P_k \tup{B_u / F} = P_k \tup{B_v / F}$.

\item $\mathcal{R}_3$: For any finite field $F$, we have $P_1 \tup{B_u / F} = P_1 \tup{B_v / F}$.
\end{itemize}
\end{theorem}

\begin{vershort}
\begin{proof}
Let $\mathcal{F}$ and $\mathcal{F}'$ be the Ferrers boards $B_u$ and $B_v$.
Then, our statements $\mathcal{R}_1$, $\mathcal{R}_2$ and $\mathcal{R}_3$ are (respectively) the conditions 6, 5 and 4 of \cite[Corollary 3.2]{CoGrRa23}.
Thus, the former three statements are equivalent (since \cite[Corollary 3.2]{CoGrRa23} shows that the latter three conditions are equivalent).
Combined with Theorem~\ref{thm.DU-eqs3}, this proves Theorem~\ref{thm.DU-eqs4}.
\end{proof}
\end{vershort}

\begin{verlong}
\begin{proof}
Let $\mathcal{F}$ and $\mathcal{F}'$ be the Ferrers boards $B_u$ and $B_v$.
Then, our statements $\mathcal{R}_1$, $\mathcal{R}_2$ and $\mathcal{R}_3$ are (respectively) the conditions 6, 5 and 4 of \cite[Corollary 3.2]{CoGrRa23}.
Thus, the former three statements are equivalent (since \cite[Corollary 3.2]{CoGrRa23} shows that the latter three conditions are equivalent).
In other words,
$\mathcal{R}_1 \Longleftrightarrow \mathcal{R}_2 \Longleftrightarrow \mathcal{R}_3$.
Combined with Theorem~\ref{thm.DU-eqs3}, this proves Theorem~\ref{thm.DU-eqs4}.

However, let us also give an alternative proof of the equivalence
$\mathcal{R}_1\Longleftrightarrow\mathcal{R}_2$ using some results of
Haglund \cite{Haglun98} and Garsia and Remmel
\cite{GarRem86}. We should warn that Haglund's paper \cite{Haglun98} differs
from us (and from \cite{GarRem86}) in how it defines Ferrers boards: Haglund's
Ferrers boards are the reflections of ours across a horizontal axis.
Fortunately, this reflection does not affect the theory in any significant way
(since rook placements and matrices can be reflected along with the boards),
and thus any result can be easily translated from one convention to another.

Any Ferrers board $B$ and any number $k \in \NN$ give rise to a $q$\emph{-rook polynomial} (aka $q$\emph{-rook
number}) $R_{k}\left(  B,q\right)  \in\ZZ\left[  q\right]  $,
which is defined in \cite[(I.4)]{GarRem86} or in \cite[(1)]{Haglun98}.
(In \cite[(1)]{Haglun98}, it is denoted by $R_{k}\left(  B\right)  $, and the
variable $q$ is renamed $x$.) A result of Garsia and Remmel (\cite[last
paragraph of \S 1]{GarRem86}, \cite[between (2) and (3)]{Haglun98}) says that
the two Ferrers boards $B_{u}$ and $B_{v}$ have the same rook numbers (i.e.,
are rook-equivalent) if and only if they have the same $q$-rook numbers (i.e.,
if $R_{k}\left(  B_{u},q\right)  =R_{k}\left(  B_{v},q\right)  $ for all
$k\in\NN$). In other words, the statement $\mathcal{R}_1$ is
equivalent to the following statement:

\begin{itemize}
\item $\mathcal{R}_5$: We have $R_{k}\left(  B_{u},q\right)  =R_{k}\left(
B_{v},q\right)  $ for all $k\in\NN$.
\end{itemize}

Next we recall a result of Haglund \cite[Theorem 1]{Haglun98}, which says that
any Ferrers board $B$, any finite
field $F$ and any $k\in\NN$ satisfy
\begin{equation}
P_{k}\left(  B/F\right)  =\left(  \left\vert F\right\vert -1\right)
^{k}\left\vert F\right\vert ^{\left\vert B\right\vert -k}R_{k}\left(
B,\ \left\vert F\right\vert ^{-1}\right)  .\label{pf.thm.DU-eqs4.hag}%
\end{equation}
(Note that \cite{Haglun98} denotes $\left\vert F\right\vert $ by $q$ and
denotes $\left\vert B\right\vert $ by $\operatorname*{Area}\left(  B\right)
$, and also abbreviates $R_{k}\left(  B,q^{-1}\right)  $ by $R_{k}\left(
q^{-1}\right)  $.)

We are now ready to prove the equivalence $\mathcal{R}_1\Longleftrightarrow
\mathcal{R}_2$:

$\mathcal{R}_1\Longrightarrow\mathcal{R}_2$:
Assume that $\mathcal{R}_1$ holds. Thus, $r_{k}\left(  B_{u}\right)  =r_{k}\left(  B_{v}\right)  $
for all $k\in\NN$. In particular, $r_{1}\left(  B_{u}\right)
=r_{1}\left(  B_{v}\right)  $. In other words,
\begin{equation}
\left\vert B_{u}\right\vert =\left\vert B_{v}\right\vert
\label{pf.thm.DU-eqs4.1}%
\end{equation}
(since $r_{1}\left(  B\right)  =\left\vert B\right\vert $ for any board $B$).
Moreover, we have assumed that $\mathcal{R}_1$ holds; thus, $\mathcal{R}_5$ holds as well
(since we have already observed that $\mathcal{R}_1$ is
equivalent to $\mathcal{R}_5$). In other words,%
\begin{equation}
R_{k}\left(  B_{u},q\right)  =R_{k}\left(  B_{v},q\right)
\label{pf.thm.DU-eqs4.2}%
\end{equation}
for all $k\in\NN$. Now, if $F$ is any finite field, and if
$k\in\NN$, then%
\begin{align*}
P_{k}\left(  B_{u}/F\right)    & =\left(  \left\vert F\right\vert -1\right)
^{k}\left\vert F\right\vert ^{\left\vert B_{u}\right\vert -k}R_{k}\left(
B_{u},\ \left\vert F\right\vert ^{-1}\right)  \ \ \ \ \ \ \ \ \ \ \left(
\text{by (\ref{pf.thm.DU-eqs4.hag})}\right)  \\
& =\left(  \left\vert F\right\vert -1\right)  ^{k}\left\vert F\right\vert
^{\left\vert B_{v}\right\vert -k}R_{k}\left(  B_{v},\ \left\vert F\right\vert
^{-1}\right)  \ \ \ \ \ \ \ \ \ \ \left(  \text{by (\ref{pf.thm.DU-eqs4.1})
and (\ref{pf.thm.DU-eqs4.2})}\right)  \\
& =P_{k}\left(  B_{v}/F\right)  \ \ \ \ \ \ \ \ \ \ \left(  \text{by
(\ref{pf.thm.DU-eqs4.hag})}\right)  .
\end{align*}
In other words, statement $\mathcal{R}_2$ holds. This proves the implication
$\mathcal{R}_1\Longrightarrow\mathcal{R}_2$.

$\mathcal{R}_2\Longrightarrow\mathcal{R}_1$: Assume that $\mathcal{R}_2$ holds.
In other words, for any finite field $F$ and any $k\in\NN$, we have
\begin{equation}
P_{k}\left(  B_{u}/F\right)  =P_{k}\left(  B_{v}/F\right)
.\label{pf.thm.DU-eqs4.5}%
\end{equation}
Summing this equality over all $k\in\NN$, we obtain%
\begin{equation}
\sum_{k\in\NN} P_{k}\left(  B_{u}/F\right)
=
\sum_{k\in\NN} P_{k}\left(  B_{v}/F\right)
\label{pf.thm.DU-eqs4.6}
\end{equation}
for any finite field $F$. However, if $F$ is a finite field and $B\subseteq
\left\{  1,2,\ldots,n\right\}  ^{2}$ is any board, then
$\sum\limits_{k\in\NN}P_{k}\left(  B/F\right)  $ is
the number of \textbf{all} $n\times
n$-matrices $A\in F^{n\times n}$ (of all possible ranks) such that all entries
of $A$ in cells outside of $B$ are zero (by the definition of $P_{k}\left(
B/F\right)  $). This number is $\left\vert F\right\vert ^{\left\vert
B\right\vert }$, since the $\left\vert B\right\vert $ entries of such a matrix
$A$ that lie in the cells of $B$ can be chosen freely from the field $F$
(while the remaining entries are determined to be $0$). Thus, for any finite
field $F$ and any board $B\subseteq\left\{  1,2,\ldots,n\right\}  ^{2}$, we
obtain%
\[
\sum_{k\in\NN}P_{k}\left(  B/F\right)
=\left\vert F\right\vert ^{\left\vert B\right\vert }.
\]
Therefore, (\ref{pf.thm.DU-eqs4.6}) can be rewritten as%
\[
\left\vert F\right\vert ^{\left\vert B_{u}\right\vert }=\left\vert
F\right\vert ^{\left\vert B_{v}\right\vert }.
\]
Applying this to $F=\mathbb{F}_{2}$ (in which case $\left\vert F\right\vert
=2$), we obtain $2^{\left\vert B_{u}\right\vert }=2^{\left\vert B_{v}%
\right\vert }$. Therefore, $\left\vert B_{u}\right\vert =\left\vert
B_{v}\right\vert $.

Now, let $k\in\NN$. If $F$ is any finite field, then%
\[
P_{k}\left(  B_{u}/F\right)  =\left(  \left\vert F\right\vert -1\right)
^{k}\left\vert F\right\vert ^{\left\vert B_{u}\right\vert -k}R_{k}\left(
B_{u},\ \left\vert F\right\vert ^{-1}\right)  \ \ \ \ \ \ \ \ \ \ \left(
\text{by (\ref{pf.thm.DU-eqs4.hag})}\right)
\]
and%
\[
P_{k}\left(  B_{v}/F\right)  =\left(  \left\vert F\right\vert -1\right)
^{k}\left\vert F\right\vert ^{\left\vert B_{v}\right\vert -k}R_{k}\left(
B_{v},\ \left\vert F\right\vert ^{-1}\right)  \ \ \ \ \ \ \ \ \ \ \left(
\text{by (\ref{pf.thm.DU-eqs4.hag})}\right)  .
\]
The left hand sides of these two equalities are equal (by
(\ref{pf.thm.DU-eqs4.5})). Thus, so are their right hand sides. In other
words, for any finite field $F$, we have%
\[
\left(  \left\vert F\right\vert -1\right)  ^{k}\left\vert F\right\vert
^{\left\vert B_{u}\right\vert -k}R_{k}\left(  B_{u},\ \left\vert F\right\vert
^{-1}\right)  =\left(  \left\vert F\right\vert -1\right)  ^{k}\left\vert
F\right\vert ^{\left\vert B_{v}\right\vert -k}R_{k}\left(  B_{v},\ \left\vert
F\right\vert ^{-1}\right)  .
\]
The factors $\left(  \left\vert F\right\vert -1\right)  ^{k}\left\vert
F\right\vert ^{\left\vert B_{u}\right\vert -k}$ and $\left(  \left\vert
F\right\vert -1\right)  ^{k}\left\vert F\right\vert ^{\left\vert
B_{v}\right\vert -k}$ in this equality are equal (since $\left\vert
B_{u}\right\vert =\left\vert B_{v}\right\vert $) and nonzero (since
$\left\vert F\right\vert >1$); thus, we can cancel them and obtain
\begin{equation}
R_{k}\left(  B_{u},\ \left\vert F\right\vert ^{-1}\right)  =R_{k}\left(
B_{v},\ \left\vert F\right\vert ^{-1}\right)  .\label{pf.thm.DU-eqs4.8}%
\end{equation}
In particular, for any prime number $p$, we obtain
\begin{equation}
R_{k}\left(  B_{u},\ p^{-1}\right)  =R_{k}\left(
B_{v},\ p^{-1}\right)  .\label{pf.thm.DU-eqs4.9}
\end{equation}
(by applying \eqref{pf.thm.DU-eqs4.8} to $F = \mathbb{F}_p$).
But $R_{k}\left(  B_{u},q\right)  $ and $R_{k}\left(  B_{v},q\right)  $ are
two polynomials in $q$. The equality (\ref{pf.thm.DU-eqs4.9}) shows that these
two polynomials agree at infinitely many rational inputs (namely, at the reciprocals of all prime numbers). Hence, these two
polynomials must be equal. In other words, $R_{k}\left(  B_{u},q\right)
=R_{k}\left(  B_{v},q\right)  $.

Since we have proved this for all $k\in\NN$, we thus conclude that
statement $\mathcal{R}_5$ holds. Hence, $\mathcal{R}_1$ holds as well
(since we have already observed that $\mathcal{R}_1$ is equivalent to
$\mathcal{R}_5$). The implication $\mathcal{R}_2\Longrightarrow
\mathcal{R}_1$ is thus proved.

Combining the implications $\mathcal{R}_1\Longrightarrow\mathcal{R}_2$
and $\mathcal{R}_2\Longrightarrow\mathcal{R}_1$ yields the equivalence
$\mathcal{R}_1\Longleftrightarrow\mathcal{R}_2$. This completes our proof.
\end{proof}
\end{verlong}


\begin{remark}
\label{rmk.rook.as-crit}
Any word $w \in \M$ satisfies $B_w = B_{Uw} = B_{wD}$.
That is, the Ferrers board $B_w$ of a word $w \in \M$ does not change if we insert a $U$ at the beginning of $w$ or a $D$ at the end of $w$.
Thus, we can use Theorem~\ref{thm.DU-eqs4} to tell whether two Ferrers boards are rook-equivalent:
Namely, we write the two Ferrers boards as $B_u$ and $B_v$, where $u$ and $v$ are two words with the same number of $U$'s and the same number of $D$'s (this can be ensured by inserting an appropriate number of $U$'s at the beginning and an appropriate number of $D$'s at the end of either word), and then Theorem~\ref{thm.DU-eqs4} provides us several equivalent criteria for the rook-equivalence of $B_u$ and $B_v$.
\end{remark}

\section{Balanced commutations revisited: irreducible balanced words}\label{sec:irredbal}

In our definition of balanced commutations (which underlay the definition of the equivalence relation $\balsim$), we allowed two arbitrary balanced factors of our word to trade places, as long as they were adjacent in the word.
Now, one may wonder whether we can get by with a smaller set of allowed swaps:
Is there a more restrictive subset of balanced commutations that generates the same equivalence relation $\balsim$ ?

The answer is ``yes'', and in fact there are likely several reasonable choices.
We  here present one, which is not minimal but still far more parsimonious than the set of all balanced commutations.

To define it, we begin with a simple notion:
A balanced word $w$ is said to be \emph{irreducible} if it is nonempty and cannot be written as a concatenation $w = uv$ of two nonempty balanced words $u$ and $v$.
For instance, the balanced word $DDUDUU$ is irreducible, whereas the balanced word $DUUUDD$ is not (since it is the concatenation $DU \cdot UUDD$).
In terms of diagonal paths, this notion can be restated as follows:
Given a nontrivial diagonal path $\pp$ with initial height $0$ and final height $0$, its reading word $\w\tup{\pp}$ is irreducible if and only if $\pp$ intersects the x-axis only in its first and last vertices.

Given two words $v,w\in\M$, we say that $v$ is obtained from
$w$ by an \emph{irreducible balanced commutation} if and only if we can write $v$ and $w$
as $v=pxyq$ and $w=pyxq$, where $p,q\in\M$ are two arbitrary words and where
$x,y\in\M$ are two irreducible balanced words with different first letters.
Clearly, this condition implies that $v$ is obtained from $w$ by a balanced commutation, but the converse is not true.
For example, the word $UUDDDUUD$ is obtained from $UDDUUUDD$ by an irreducible balanced commutation (swapping the $DDUU$ with the $UD$ in the middle), but the word $UUDDUDDU$ is not (indeed, it is obtained by swapping the balanced factors $UDDU$ and $UUDD$, but these factors don't have different first letters, and the first of them is not irreducible).

We define an equivalence relation $\irrsim$ on
$\M$ by stipulating that two words $w,v\in\M$ satisfy
$w \irrsim v$ if and only if $v$ can be obtained
from $w$ by a sequence (possibly empty) of irreducible balanced commutations.
Even though not every balanced commutation is irreducible, we claim the following:

\begin{theorem}
    \label{thm.irr=bal}
    The relations $\irrsim$ and $\balsim$ are the same.
    That is, if a word $v$ can be obtained from a word $w$ by a sequence of balanced commutations, then we can also obtain $v$ from $w$ by a (possibly longer) sequence of irreducible balanced commutations.
\end{theorem}

The proof of this theorem needs a few lemmas. The first one is nearly obvious:

\begin{lemma}
\label{lem.irr.fac}
Any balanced word $w\in\M$ can be decomposed into
a product $w=v_{1}v_{2}\cdots v_{k}$ of irreducible balanced words
$v_{1},v_{2},\ldots,v_{k}\in\M$. (If $w$ is empty, we will have $k=0$ here.)
\end{lemma}

\begin{proof}
This is shown in the same way as the existence of a factorization of a
positive integer into primes:

Let $w\in\M$ be a balanced word. If $w$ is irreducible or empty, then
we are done. If not, then $w$ can be written as a product of two shorter
nonempty balanced words. If these two shorter words are irreducible, then we
are done. If not, then at least one of them can itself be written as a product
of two shorter nonempty balanced words, so that $w$ becomes a product of three
nonempty balanced words. Thus, we obtain longer and longer factorizations of
$w$ into shorter and shorter nonempty balanced words. Obviously, this process
will eventually have to stop, and at that point we will have a factorization
of $w$ into irreducible balanced words in front of us.
\end{proof}

The decomposition in Lemma \ref{lem.irr.fac} is furthermore unique (and this
is easy to see using diagonal paths), but we do not need this.

Next we introduce some shorthand terminology: A \emph{UIB word} will mean an
irreducible balanced word that begins with a $U$. A \emph{DIB word} will mean
an irreducible balanced word that begins with a $D$. Note that any irreducible
balanced word is nonempty, and thus is either UIB or DIB (but not both). The
following is another easy observation:

\begin{lemma}
\label{lem.irr-omega}
\textbf{(a)} Any UIB word ends with a $D$.

\textbf{(b)} Any DIB word ends with a $U$.

\textbf{(c)} If $u\in\M$ is a UIB word, then $\omega\left(  u\right)
$ is a UIB word as well.

\textbf{(d)} If $u\in\M$ is a DIB word, then $\omega\left(  u\right)
$ is a DIB word as well.
\end{lemma}

\begin{proof}
\textbf{(a)} Let $w$ be a UIB word. We need to prove that $w$ ends with a $D$.

\begin{vershort}
Assume the contrary. Thus, $w$ ends with a $U$. But $w$ also starts with a $U$
(since $w$ is UIB) and is balanced (for the same reason). Hence, Lemma
\ref{lem.balanced.0} shows that we can write $w$ as a concatenation $w=pq$,
where $p$ and $q$ are two nonempty balanced words. Hence, $w$ is not
irreducible, despite being a UIB word.
This contradiction shows that our assumption was false.
Lemma \ref{lem.irr-omega} \textbf{(a)} is thus proved.
\end{vershort}

\begin{verlong}
Assume the contrary. Thus, $w$ ends with a $U$. But $w$ also starts with a $U$
(since $w$ is UIB) and is balanced (for the same reason). Hence, Lemma
\ref{lem.balanced.0} shows that we can write $w$ as a concatenation $w=pq$,
where $p$ is a balanced word starting with a $U$, and where $q$ is a balanced
word starting with a $D$. Consider these $p$ and $q$. Thus, both $p$ and $q$
are nonempty balanced words. Hence, $w=pq$ shows that $w$ is not irreducible.
But this contradicts our assumption that $w$ be UIB. This contradiction shows
that our assumption was false. Lemma \ref{lem.irr-omega} \textbf{(a)} is thus proved.
\end{verlong}

\textbf{(b)} This is analogous to part \textbf{(a)}; we just need to
interchange the roles of $U$ and $D$.

\textbf{(c)} Let $u\in\M$ be a UIB word. Then, $u$ ends with a $D$
(by part \textbf{(a)}). Hence, $\omega\left(  u\right)  $ starts with a $U$
(since the anti-automorphism $\omega$ reverses a word and replaces each $D$ by
a $U$ and each $U$ by a $D$). Moreover, $u$ is irreducible balanced (since $u$
is UIB). Therefore, $\omega\left(  u\right)  $ is irreducible balanced as well
(indeed, the irreducibility follows from the fact that any factorization of
$\omega\left(  u\right)  $ into two nonempty balanced factors could be turned
back into a factorization of $u$ by applying $\omega$ again%
\begin{verlong}
\footnote{In more
detail: Assume that $\omega\left(  u\right)  $ is not irreducible. Thus,
$\omega\left(  u\right)  $ can be factored as $\omega\left(  u\right)  =pq$
for two nonempty balanced words $p$ and $q$. Consider these $p$ and $q$. Then,
$\omega\left(  p\right)  $ and $\omega\left(  q\right)  $ are nonempty
balanced words (since $p$ and $q$ are nonempty balanced words). However, from
$\omega\circ\omega=\id$, we obtain $\omega\left(  \omega\left(
u\right)  \right)  =u$. Thus, $u=\omega\left(  \underbrace{\omega\left(
u\right)  }_{=pq}\right)  =\omega\left(  pq\right)  =\omega\left(  q\right)
\omega\left(  p\right)  $. Thus, $u$ is the product of two nonempty balanced
words (namely, $\omega\left(  q\right)  $ and $\omega\left(  p\right)  $). But
this contradicts the irreducibility of $u$. This contradiction shows that our
assumption was false. Hence, $\omega\left(  u\right)  $ is irreducible.}
\end{verlong}
).
Hence, $\omega\left(  u\right)  $ is a UIB word (since $\omega\left(
u\right)  $ starts with a $U$). This proves Lemma \ref{lem.irr-omega}
\textbf{(c)}.

\textbf{(d)} This is analogous to part \textbf{(c)}; we just need to
interchange the roles of $U$ and $D$.
\end{proof}

It is not hard to see that each UIB word has the form $UsD$ where $s$ is a balanced $1$-Dyck word (i.e., a balanced word whose each prefix is rising).
Likewise, each DIB word has the form $DtU$ where $t$ is a balanced anti-$1$-Dyck word (i.e., a balanced word whose each prefix is falling).
Obviously, UIB words can be distinguished from DIB words by their first letter.

Now we claim the following variant of Lemma \ref{lem.balanced.0}:

\begin{lemma}
\label{lem.irr.0} Let $w\in\M$ be a balanced word that starts with a
$U$ and ends with a $U$. Then, we can write $w$ as a concatenation $w=spqt$,
where $s,p,q,t\in\M$ are four balanced words such that $p$ is UIB and
$q$ is DIB.
\end{lemma}

\begin{proof}
Lemma \ref{lem.balanced.0} shows that we can write $w$ as a concatenation
$w=p^{\prime}q^{\prime}$, where $p^{\prime}$ is a balanced word starting with
a $U$, and where $q^{\prime}$ is a balanced word starting with a $D$. Consider
these $p^{\prime}$ and $q^{\prime}$.

Lemma \ref{lem.irr.fac} shows that $p^{\prime}$ can be decomposed into a
product $p^{\prime}=p_{1}p_{2}\cdots p_{k}$ of irreducible balanced words
$p_{1},p_{2},\ldots,p_{k}\in\M$. Likewise, $q^{\prime}$ can be
decomposed into a product $q^{\prime}=p_{k+1}p_{k+2}\cdots p_{\ell}$ of
irreducible balanced words $p_{k+1},p_{k+2},\ldots,p_{\ell}\in\M$.
Consider these decompositions. Thus, $p_{1},p_{2},\ldots,p_{\ell}$ are $\ell$
irreducible balanced words such that $p^{\prime}=p_{1}p_{2}\cdots p_{k}$ and
$q^{\prime}=p_{k+1}p_{k+2}\cdots p_{\ell}$. Hence,%
\[
p^{\prime}q^{\prime}=\left(  p_{1}p_{2}\cdots p_{k}\right)  \left(
p_{k+1}p_{k+2}\cdots p_{\ell}\right)  =p_{1}p_{2}\cdots p_{\ell}.
\]

The word $p_{1}$ is nonempty (since it is irreducible) and is a prefix of
$p^{\prime}$ (since $p^{\prime}=p_{1}p_{2}\cdots p_{k}$). Thus, it starts with
a $U$ (since $p^{\prime}$ starts with a $U$). Similarly, $p_{k+1}$ starts with
a $D$.

Now, consider the \textbf{smallest} number $i\in\left\{  1,2,\ldots
,\ell\right\}  $ for which the word $p_{i}$ starts with a $D$. (Such an $i$
exists, since $p_{k+1}$ starts with a $D$.) Then, $i$ cannot be $1$ (since
$p_{1}$ starts with a $U$, not with a $D$), and thus must be $\geq2$. Hence,
$i-1\in\left\{  1,2,\ldots,\ell\right\}  $. The word $p_{i-1}$ cannot start
with a $D$ (because then, $i$ would not be the \textbf{smallest} number for
which $p_{i}$ starts with a $D$), and thus must start with a $U$. Hence, the
word $p_{i-1}$ is UIB (since it is irreducible balanced). Meanwhile, the word
$p_{i}$ is DIB (since it is irreducible balanced and starts with a $D$). All
the words $p_{1},p_{2},\ldots,p_{\ell}$ are balanced; hence, their
concatenations $p_{1}p_{2}\cdots p_{i-2}$ and $p_{i+1}p_{i+2}\cdots p_{\ell}$
are balanced as well (since any concatenation of balanced words is balanced).

Now,%
\[
w=p^{\prime}q^{\prime}=p_{1}p_{2}\cdots p_{\ell}=\underbrace{\left(
p_{1}p_{2}\cdots p_{i-2}\right)  }_{\text{balanced word}}\underbrace{p_{i-1}%
}_{\text{UIB word}}\underbrace{p_{i}}_{\text{DIB word}}\underbrace{\left(
p_{i+1}p_{i+2}\cdots p_{\ell}\right)  }_{\text{balanced word}}.
\]
Hence, we can write $w$ as a concatenation $w=spqt$, where $s,p,q,t\in
\M$ are four balanced words such that $p$ is UIB and $q$ is DIB
(namely, we take $s=p_{1}p_{2}\cdots p_{i-2}$ and $p=p_{i-1}$ and $q=p_{i}$
and $t=p_{i+1}p_{i+2}\cdots p_{\ell}$). This proves Lemma \ref{lem.irr.0}.
\end{proof}

This, in turn, allows us to improve Lemma \ref{lem.upnorm.1} as follows:

\begin{lemma}
\label{lem.irr.1}Let $w\in\M$ be a rising word that is not up-normal.
Then, we can write $w$ in the form $w=upqv$, where $u$ and $v$ are two words,
where $p$ is a UIB word, and where $q$ is a DIB word.
\end{lemma}

\begin{vershort}
\begin{proof}
Proceed as in the above proof of Lemma \ref{lem.upnorm.1} to find a factor $w'$ of $w$ that is balanced and starts and ends with a $U$.
Then, apply Lemma~\ref{lem.irr.0} to $w'$ instead of $w$.
Thus, $w'$ is written as a concatenation $w' = spqt$,
where $s,p,q,t\in\M$ are four balanced words such that $p$ is UIB and $q$ is DIB.
Let $\widetilde{u}$ be the prefix of $w$ coming before the factor $w'$, and $\widetilde{v}$ be the suffix of $w$ coming after $w'$.
Then, $w = \widetilde{u} w' \widetilde{v} = \underbrace{\widetilde{u} s}_{=:u} p q \underbrace{t \widetilde{v}}_{=:v}$
is the factorization we are looking for.
\end{proof}
\end{vershort}

\begin{verlong}
\begin{proof}
Write $w$ as $w=w_{1}w_{2}\cdots w_{\ell}$, where $w_{1},w_{2},\ldots,w_{\ell
}\in\left\{  U,D\right\}  $ are the letters of $w$. In the proof of Lemma
\ref{lem.upnorm.1}, we have found a balanced factor $w_{a+1}w_{a+2}\cdots
w_{b}$ of $w$ that starts with a $U$ and ends with a $U$. Consider this
balanced factor. Lemma \ref{lem.irr.0} (applied to $w_{a+1}w_{a+2}\cdots
w_{b}$ instead of $w$) then shows that we can write $w_{a+1}w_{a+2}\cdots
w_{b}$ as a concatenation $w_{a+1}w_{a+2}\cdots w_{b}=spqt$, where
$s,p,q,t\in\M$ are four balanced words such that $p$ is UIB and $q$
is DIB. Consider these four words $s,p,q,t$.

Now,%
\begin{align*}
w  & =w_{1}w_{2}\cdots w_{\ell}=\left(  w_{1}w_{2}\cdots w_{a}\right)
\underbrace{\left(  w_{a+1}w_{a+2}\cdots w_{b}\right)  }_{=spqt}\left(
w_{b+1}w_{b+2}\cdots w_{\ell}\right)  \\
& =\left(  w_{1}w_{2}\cdots w_{a}\right)  spqt\left(  w_{b+1}w_{b+2}\cdots
w_{\ell}\right)  .
\end{align*}
Hence, we can write $w$ in the form $w=upqv$, where $u$ and $v$ are two words,
where $p$ is a UIB word, and where $q$ is a DIB word (namely, we set
$u=\left(  w_{1}w_{2}\cdots w_{a}\right)  s$ and $p=p$ and $q=q$ and
$v=t\left(  w_{b+1}w_{b+2}\cdots w_{\ell}\right)  $). This proves Lemma
\ref{lem.irr.1}.
\end{proof}
\end{verlong}

Next, we prove an analogue of Proposition \ref{prop.upnorm.nf}:

\begin{lemma}
\label{lem.irr.nf} Let $w\in\M$ be a rising word. Then, there exists
a unique up-normal word $t\in\M$ such that $t\irrsim w$.
\end{lemma}

\begin{proof}
The existence of $t$ can be proved in the same way as the existence
part of Proposition \ref{prop.upnorm.nf}, but using Lemma \ref{lem.irr.1}
instead of Lemma~\ref{lem.upnorm.1} (since UIB words start with a $U$, while
DIB words start with a $D$).

The uniqueness of $t$ follows from the uniqueness of $t$ in Proposition
\ref{prop.upnorm.nf}, since the relation $t\irrsim w$ implies $t\balsim w$.
\end{proof}

The following is an analogue of Proposition \ref{prop.bal-symm}:

\begin{proposition}
\label{prop.irr-symm} Let $u$ and $v$ be two words in $\M$. Then,
$u\irrsim v$ if and only if $\omega\left(  u\right)  \irrsim\omega\left(
v\right)  $.
\end{proposition}

\begin{proof}
This is similar to the proof of Proposition \ref{prop.bal-symm}, with a slight
twist: We need to show that if $a$ and $b$ are two irreducible balanced words
that have different first letters, then their images $\omega\left(  a\right)
$ and $\omega\left(  b\right)  $ are again two irreducible balanced words that
have different first letters. But this follows from parts \textbf{(c)} and
\textbf{(d)} of Lemma \ref{lem.irr-omega}.
\end{proof}

We have an analogue of Lemma \ref{lem.4to5} as well:

\begin{lemma}
\label{lem.irr.4to5} Let $\pp$ and $\qq$ be two diagonal paths
with the same initial height and the same final height. Assume that $H\left(
\pp,z\right)  =H\left(  \qq,z\right)  $. Then,
$\w\left(  \pp\right)  \irrsim\w\left(
\qq\right)  $.
\end{lemma}

\begin{proof}
Analogous to the proof of Lemma \ref{lem.4to5}, but using Proposition
\ref{prop.irr-symm} and Lemma \ref{lem.irr.nf} instead of Proposition
\ref{prop.bal-symm} and Proposition \ref{prop.upnorm.nf}. (Use the obvious
fact that $u\irrsim v$ implies $u\balsim v$.)
\end{proof}

We can now prove Theorem \ref{thm.irr=bal}:

\begin{proof}
[Proof of Theorem \ref{thm.irr=bal}.] Let us modify Theorem \ref{thm.DU-eqs2}
by adding the following extra statement:

\begin{itemize}
\item $\mathcal{S}_{5}^{\prime}$: We have $u\irrsim v$.
\end{itemize}

Clearly, this statement $\mathcal{S}_{5}^{\prime}$ implies $\mathcal{S}_{5}$,
since $u\irrsim v$ implies $u\balsim v$. But the implication $\mathcal{S}%
_{4}\Longrightarrow\mathcal{S}_{5}^{\prime}$ holds as well, and can be proved
just as the implication $\mathcal{S}_{4}\Longrightarrow\mathcal{S}_{5}$ in
Theorem \ref{thm.DU-eqs2} was proved (but using Lemma \ref{lem.irr.4to5} instead of
Lemma \ref{lem.4to5}). Hence, the statement $\mathcal{S}_{5}^{\prime}$ is
equivalent to all the seven statements
$\mathcal{S}_{1},\mathcal{S}_{2},\mathcal{S}_{3},\mathcal{S}'_3,\mathcal{S}_{4},\mathcal{S}_{5}, \mathcal{S}_6$
from Theorem \ref{thm.DU-eqs2}.
In
particular, $\mathcal{S}_{5}^{\prime}$ is equivalent to $\mathcal{S}_{5}$. In
other words, $u\irrsim v$ is equivalent to $u\balsim v$. In other words, the
relations $\irrsim$ and $\balsim$ are the same. This proves Theorem
\ref{thm.irr=bal}.
\end{proof}

\section{Other algebras}\label{sec:otheralg}

Everything we have done so far concerned the \textquotedblleft rank-$1$%
\textquotedblright\ Weyl algebra
\[
\W=\kk\left\langle
D,U\ \mid\ DU-UD=1\right\rangle.
\]
But the main question we addressed -- to
classify equal products of generators -- can be posed for any $\kk%
$-algebra given by generators and relations. In particular, several analogues
and variants of $\W$ are natural candidates for a
similar study. In this section, we briefly discuss some of them, giving some
answers and posing some questions.
(There are many more -- see, e.g., \cite{Gaddis23} for a recent survey.)

\subsection{Multivariate Weyl algebras}

For any $n\in\NN$, there is an \textquotedblleft$n$-Weyl
algebra\textquotedblright\ $\W_n$, defined as the $\kk$-algebra given by $2n$ generators
$D_{1},D_{2},\ldots,D_{n},U_{1},U_{2},\ldots,U_{n}$ and relations%
\begin{align*}
D_{i}U_{j}  &  =U_{j}D_{i}\ \ \ \ \ \ \ \ \ \ \text{for all }i\neq j;\\
D_{i}U_{i}  &  =U_{i}D_{i}+1\ \ \ \ \ \ \ \ \ \ \text{for all }i;\\
D_{i}D_{j}  &  =D_{j}D_{i}\ \ \ \ \ \ \ \ \ \ \text{for all }i,j;\\
U_{i}U_{j}  &  =U_{j}U_{i}\ \ \ \ \ \ \ \ \ \ \text{for all }i,j.
\end{align*}
It is isomorphic to the $\kk$-algebra of differential
operators on the polynomial ring $\kk\left[  x_{1},x_{2},\ldots
,x_{n}\right]  $.

However, this algebra $\W_n$ can also be seen as the $n$-fold
tensor power\footnote{All tensor products and tensor powers
in this paper are taken over the field $\kk$.} $\W^{\otimes n}$ of the original Weyl algebra
$\W$, via the $\kk$-algebra isomorphism
$\W_n \to \W^{\otimes n}$ that sends each generator $D_{i}$ to
$\underbrace{1\otimes1\otimes\cdots\otimes1}_{i-1\text{ times}}\otimes
\,D\otimes\underbrace{1\otimes1\otimes\cdots\otimes1}_{n-i\text{ times}}$ and
each generator $U_{i}$ to $\underbrace{1\otimes1\otimes\cdots\otimes
1}_{i-1\text{ times}}\otimes\,U\otimes\underbrace{1\otimes1\otimes
\cdots\otimes1}_{n-i\text{ times}}$. From this point of view, products of
generators of $\W_{n}$ are just elements of the form $\phi\left(
w_{1}\right)  \otimes\phi\left(  w_{2}\right)  \otimes\cdots\otimes\phi\left(
w_{n}\right)  \in\W^{\otimes n}$, where $w_{1},w_{2},\ldots,w_{n}%
\in\M$ are some words. Which of these products are equal? The answer
turns out to boil down to the answer for $n=1$ (which we know from Theorems
\ref{thm.DU-eqs2}, \ref{thm.DU-eqs3} and \ref{thm.DU-eqs4}):

\begin{theorem}
\label{thm.Wn.reduce}Let $u_{1},u_{2},\ldots,u_{n},v_{1},v_{2},\ldots,v_{n}$
be $2n$ words in $\M$. Then,%
\[
\phi\left(  u_{1}\right)  \otimes\phi\left(  u_{2}\right)  \otimes
\cdots\otimes\phi\left(  u_{n}\right)  =\phi\left(  v_{1}\right)  \otimes
\phi\left(  v_{2}\right)  \otimes\cdots\otimes\phi\left(  v_{n}\right)  \text{
in }\W^{\otimes n}
\]
if and only if
\[
\text{each }i\in\left\{  1,2,\ldots,n\right\}  \text{ satisfies }\phi\left(
u_{i}\right)  =\phi\left(  v_{i}\right)  .
\]

\end{theorem}

In other words, we don't get any \textquotedblleft new\textquotedblright%
\ equalities by tensoring $n$ copies of $\W$.

The \textquotedblleft
if\textquotedblright\ part of Theorem \ref{thm.Wn.reduce} is obvious, while the \textquotedblleft
only if\textquotedblright\ part of the theorem follows immediately from two lemmas. The first is a
general fact from linear algebra (\cite[Theorem 5.15]{KConrad-tp2}):

\begin{lemma}
\label{lem.tensor.eq}Let $V_{1},V_{2},\ldots,V_{n}$ be any $\kk$-vector
spaces. For each $i\in\left\{  1,2,\ldots,n\right\}  $, let $x_{i}$ and
$y_{i}$ be two nonzero vectors in $V_{i}$. Then,
\[
x_{1}\otimes x_{2}\otimes\cdots\otimes x_{n}=y_{1}\otimes y_{2}\otimes
\cdots\otimes y_{n}\text{ in }V_{1}\otimes V_{2}\otimes\cdots\otimes V_{n}%
\]
if and only if there exist some scalars $\lambda_{1},\lambda_{2}%
,\ldots,\lambda_{n}\in\kk$ such that $\lambda_{1}\lambda_{2}%
\cdots\lambda_{n}=1$ and such that
\[
\text{each }i\in\left\{  1,2,\ldots,n\right\}  \text{ satisfies }x_{i}%
=\lambda_{i}y_{i}.
\]

\end{lemma}

The next lemma ensures that the conditions of Lemma \ref{lem.tensor.eq} are
met in the appropriate case:

\begin{lemma}
\label{lem.Wn.scalar1}\textbf{(a)} For any $w\in\M$, we have
$\phi\left(  w\right)  \neq0$.

\textbf{(b)} If $u,v\in\M$ and $\lambda\in\kk$ satisfy
$\phi\left(  u\right)  =\lambda\phi\left(  v\right)  $, then $\lambda=1$ and
$\phi\left(  u\right)  =\phi\left(  v\right)  $.
\end{lemma}

\begin{proof}[First proof.]
The family $\left(  U^{j}D^{i}\right)  _{i,j\in\NN}$ is a
basis of $\W$ (by \cite[Proposition 2.7.1 (i)]{Etingof}). Hence, each
element $a$ of $\W$ can be uniquely written as a $\kk$-linear
combination $\sum\limits_{i,j\in\NN}a_{i,j}U^{j}D^{i}$ of the elements of this
family. When this $a$ is nonzero, we define the \emph{leading monomial} of $a$
to be the lexicographically highest pair $\left(  i,j\right)  \in
\NN^2$ for which $a_{i,j}\neq0$, and we define the \emph{leading
coefficient} of $a$ to be the coefficient $a_{i,j}$ corresponding to this pair
$\left(  i,j\right)  $. \medskip

\textbf{(a)} Let $w\in\M$. Then, Theorem \ref{thm.navon} yields
\begin{equation}
\phi\left(  w\right)  =\sum_{k=0}^{\min\left\{  m,n\right\}  }r_{k}\left(
B_{w}\right)  U^{m-k}D^{n-k} \label{pf.lem.Wn.scalar1.1}%
\end{equation}
for appropriate $n,m\in\NN$. The elements $U^{m-k}D^{n-k}$ on the right
hand side of this equality are $\kk$-linearly independent (since the
family $\left(  U^{j}D^{i}\right)  _{i,j\in\NN}$ is a basis of
$\W$), and at least one of the coefficients $r_{k}\left(
B_{w}\right)  $ is nonzero (indeed, we have $r_{0}\left(  B_{w}\right)  =1$,
since any board $B$ satisfies $r_{0}\left(  B\right)  =1$). Thus, the entire
right hand side is nonzero. Hence, $\phi\left(  w\right)  \neq0$. This proves
Lemma \ref{lem.Wn.scalar1} \textbf{(a)}. \medskip

\textbf{(b)} Let $u,v\in\M$ and $\lambda\in\kk$ satisfy
$\phi\left(  u\right)  =\lambda\phi\left(  v\right)  $. Recall that for each
$w\in\M$, the element $\phi\left(  w\right)  \in\W$ is
nonzero (by part \textbf{(a)}) and has leading coefficient $1$ (indeed, the
equality (\ref{pf.lem.Wn.scalar1.1}) shows that the leading coefficient of
$\phi\left(  w\right)  $ is $r_{0}\left(  B_{w}\right)  =1$).
\begin{vershort}
Thus, comparing leading coefficients in
$\phi\left(  u\right)  =\lambda\phi\left(  v\right)  $
shows that $\lambda = 1$ and $\phi\left(  u\right)  =\phi\left(  v\right)$.
\end{vershort}
\begin{verlong}
Hence, the
element $\phi\left(  u\right)  $ has leading coefficient $1$. Similarly,
$\phi\left(  v\right)  $ also has leading coefficient $1$. Thus, $\lambda
\phi\left(  v\right)  $ has leading coefficient $\lambda\cdot1=\lambda$. In
other words, $\phi\left(  u\right)  $ has leading coefficient $\lambda$ (since
$\phi\left(  u\right)  =\lambda\phi\left(  v\right)  $). Since we also know
that $\phi\left(  u\right)  $ has leading coefficient $1$, we thus conclude
that $\lambda=1$. Hence, $\phi\left(  u\right)  =\underbrace{\lambda}_{=1}%
\phi\left(  v\right)  =\phi\left(  v\right)  $.
Lemma \ref{lem.Wn.scalar1} \textbf{(b)} is thus proved.
\end{verlong}
\end{proof}

\begin{proof}[Second proof.]
\textbf{(b)} First, we generalize Proposition
\ref{prop.DU-eqlett} by replacing the assumption \textquotedblleft$\phi\left(
\w\left(  \pp\right)  \right)  =\phi\left(
\w\left(  \qq\right)  \right)  $\textquotedblright\ by
\textquotedblleft$\phi\left(  \w\left(  \pp\right)
\right)  =\lambda\phi\left(  \w\left(  \qq\right)
\right)  $\textquotedblright\ and adding the claim \textquotedblleft%
$\lambda=1$\textquotedblright\ to the conclusion. The proof of Proposition
\ref{prop.DU-eqlett} that we gave above still applies to this generalization,
once a few trivial changes are made. In particular, the polynomial identity%
\[
\prod_{p_{i}\text{ is an NE-step of }\pp}\left(  x-h_{i}\right)
=\prod_{q_{i}\text{ is an NE-step of }\qq}\left(  x-g_{i}\right)
\]
must be replaced by%
\[
\prod_{p_{i}\text{ is an NE-step of }\pp}\left(  x-h_{i}\right)
=\lambda\prod_{q_{i}\text{ is an NE-step of }\qq}\left(  x-g_{i}%
\right)  ,
\]
which of course entails not only%
\begin{align*}
&  \left\{  h_{i}\ \mid\ p_{i}\text{ is an NE-step of }\pp\right\}
_{\multiset}\\
&  =\left\{  g_{i}\ \mid\ q_{i}\text{ is an NE-step of }\qq\right\}
_{\multiset}%
\end{align*}
but also $\lambda=1$ (by comparing leading coefficients).

\begin{vershort}
Applying this generalized version of Proposition
\ref{prop.DU-eqlett}
to our situation (setting $\pp$ and $\qq$ to be the
standard paths of $u$ and $v$), we obtain $\lambda = 1$,
and thus $\phi\tup{u} = \phi\tup{v}$.
This proves Lemma \ref{lem.Wn.scalar1} \textbf{(b)}.
\end{vershort}

\begin{verlong}
Now, let us apply this to our situation. Let $u,v\in\M$ and
$\lambda\in\kk$ satisfy $\phi\left(  u\right)  =\lambda\phi\left(
v\right)  $. Let $\pp=\left(  p_{0},p_{1},\ldots,p_{k}\right)  $ be the
diagonal path starting at $\left(  0,0\right)  $ that satisfies
$u=\w\left(  \pp\right)  $. Similarly, let
$\qq=\left(  q_{0},q_{1},\ldots,q_{m}\right)  $ be the diagonal path
starting at $\left(  0,0\right)  $ that satisfies $v=\w\left(
\qq\right)  $. The paths $\pp$ and $\qq$ have the same
initial height (namely, $0$). Moreover, we have $\phi\left(  u\right)
=\lambda\phi\left(  v\right)  $, so that $\phi\left(  \w\left(
\pp\right)  \right)  =\lambda\phi\left(  \w\left(
\qq\right)  \right)  $ (since $u=\w\left(
\pp\right)  $ and $v=\w\left(  \qq\right)  $).
Hence, the generalized version of Proposition \ref{prop.DU-eqlett} that we
have just proposed yields (among other things) that $\lambda=1$. Thus,
$\phi\left(  u\right)  =\underbrace{\lambda}_{=1}\phi\left(  v\right)
=\phi\left(  v\right)  $.
This proves Lemma \ref{lem.Wn.scalar1} \textbf{(b)}.
\end{verlong}
\medskip

\textbf{(a)} Let $w\in\M$. We must prove that $\phi\left(  w\right)
\neq0$. Assume the contrary. Thus, $\phi\left(  w\right)  =0=0\phi\left(
w\right)  $. Hence, part \textbf{(b)} (applied to $u=w$ and $v=w$ and
$\lambda=0$) yields $0=1$ and $\phi\left(  w\right)  =\phi\left(  w\right)  $.
Obviously, $0=1$ is absurd, so we found a contradiction. This proves Lemma
\ref{lem.Wn.scalar1} \textbf{(a)}.
\end{proof}

\begin{verlong}
\begin{proof}
[Proof of Theorem \ref{thm.Wn.reduce}.] The \textquotedblleft
if\textquotedblright\ part is obvious, so let us prove the \textquotedblleft
only if\textquotedblright\ part.

Lemma \ref{lem.Wn.scalar1} \textbf{(a)} shows that the vectors $\phi\left(
u_{i}\right)  $ and $\phi\left(  v_{i}\right)  $ in $\W$ are nonzero
for all $i\in\left\{  1,2,\ldots,n\right\}  $.

Assume that
\[
\phi\left(  u_{1}\right)  \otimes\phi\left(  u_{2}\right)  \otimes
\cdots\otimes\phi\left(  u_{n}\right)  =\phi\left(  v_{1}\right)  \otimes
\phi\left(  v_{2}\right)  \otimes\cdots\otimes\phi\left(  v_{n}\right)  \text{
in }\W^{\otimes n}.
\]
Then, Lemma \ref{lem.tensor.eq} (applied to $V_{i}=\W$ and
$x_{i}=\phi\left(  u_{i}\right)  $ and $y_{i}=\phi\left(  v_{i}\right)  $)
yields that there exist some scalars $\lambda_{1},\lambda_{2},\ldots
,\lambda_{n}\in\kk$ such that $\lambda_{1}\lambda_{2}\cdots\lambda
_{n}=1$ and such that
\[
\text{each }i\in\left\{  1,2,\ldots,n\right\}  \text{ satisfies }\phi\left(
u_{i}\right)  =\lambda_{i}\phi\left(  v_{i}\right)  .
\]
Consider these $\lambda_{1},\lambda_{2},\ldots,\lambda_{n}$. For each
$i\in\left\{  1,2,\ldots,n\right\}  $, we have $\phi\left(  u_{i}\right)
=\lambda_{i}\phi\left(  v_{i}\right)  $ and therefore $\phi\left(
u_{i}\right)  =\phi\left(  v_{i}\right)  $ (by Lemma \ref{lem.Wn.scalar1}
\textbf{(b)}, applied to $u=u_{i}$ and $v=v_{i}$ and $\lambda = \lambda_i$). In other words,
\[
\text{each }i\in\left\{  1,2,\ldots,n\right\}  \text{ satisfies }\phi\left(
u_{i}\right)  =\phi\left(  v_{i}\right)  .
\]
This proves the \textquotedblleft only if\textquotedblright\ part of Theorem
\ref{thm.Wn.reduce}.
\end{proof}
\end{verlong}

\subsection{Characteristic $p$}

We have hitherto assumed that the field $\kk$ has characteristic $0$. If
$\kk$ has characteristic $p\neq0$ instead, things change significantly:
Purely identity-type results such as Proposition \ref{prop.UD-act},
Proposition \ref{prop.DU-act} and Lemma \ref{lem.act-diag} remain valid
(indeed, they hold whenever $\kk$ is merely a commutative ring), but
the action of $\W$ on $\kk\left[  x\right]  $ is no longer
faithful (i.e., Lemma~\ref{lem.weyl.faith} fails), and various other results
that build on the tacit identification of integers with elements of
$\kk$ become false as well (e.g., Proposition \ref{prop.DU-eqlett}).
Lemma~\ref{lem.bal.weyl0} remains true, but the first proof we gave above no longer
works (although it is not hard to derive it from the characteristic-$0$ case,
since it is an identity in the free $\kk$-module $\W$).
Thus, Lemma~\ref{lem.bal.weyl} remains true as well.

The proof of Proposition \ref{prop.DU-eqlett2} also falls flat in
characteristic $p$, but the proposition itself survives. Indeed, it holds for
any nontrivial ring $\kk$, and can be proved by comparing leading terms
in Theorem \ref{thm.navon}.

The most interesting question is when two words $u,v\in\M$ satisfy
$\phi\left(  u\right)  =\phi\left(  v\right)  $ for a field $\kk$ of
characteristic $p\neq0$. The equivalence $\mathcal{S}_{1}\Longleftrightarrow
\mathcal{S}_{5}$ in Theorem \ref{thm.DU-eqs2} no longer holds in this case, as
(e.g.) we have $\phi\left(  U^{p+1}D\right)  =\phi\left(  UDU^{p}\right)  $
(indeed, $U^{p}$ is a central element of $\W$ when
$\charr \kk = p$) but we don't have $U^{p+1} D \balsim UDU^{p}$.
One might try to salvage the
equivalence by loosening the notion of balanced commutations, e.g., by
allowing both $U^{p}$ and $D^{p}$ to be swapped with any (neighboring) factor
of the word; the resulting equivalence relation
$\overset{p-\operatorname*{bal}}{\sim}$ might satisfy the equivalence $\left(
\phi\left(  u\right)  =\phi\left(  v\right)  \right)  \ \Longleftrightarrow
\ \left(  u\overset{p-\operatorname*{bal}}{\sim}v\right)  $, but we don't know
if it does.

\begin{question}
Does it?
\end{question}


The Weyl algebra $\W$ in characteristic $p\neq0$ has some quotients
(unlike in characteristic $0$, where it is famously a simple algebra). Indeed,
both elements $U^{p}$ and $D^{p}$ are known to lie in its center, and thus
generate two-sided ideals $\W U^{p}\W$ and $\W D^{p}\W$ that can be quotiented out. We can thus form the three
quotients%
\begin{align*}
\W^{-} &:= \W/\left(  \W U^{p}\W\right), \ \ \ \ \ \ \ \ \ \ %
\W_{-}  := \W/\left(  \W D^{p}\W\right)  , \\
\W_{-}^{-} &:= \W/\left(  \W U^{p}\W + \W D^{p}\W\right)  .
\end{align*}
The third quotient, $\W_{-}^{-}$, is actually a finite-dimensional
$\kk$-vector space, of dimension $p^{2}$ and with basis $\left(
U^{j}D^{i}\right)  _{i,j\in\left\{  0,1,\ldots,p-1\right\}  }$. It acts
faithfully and densely on the $\kk$-algebra $\kk\left[  x\right]  /\left(
x^{p}\right)  $. The quotient $\W_{-}$ acts faithfully on the full
polynomial ring $\kk\left[  x\right]  $.
The quotients $\W_{-}$ and $\W^{-}$ are isomorphic via the isomorphism sending
$U\mapsto D$ and $D\mapsto-U$.

The question of when two words $u,v\in\M$ give rise to equal
monomials can now be asked not only for $\W$, but also for any of its
quotient algebras $\W^{-}$, $\W_{-}$ and $\W_{-}^{-}$.
For instance, the words $U^{p-1}DUD^{p-1}$ and $U^{p-1}D^{p-1}$ do
not yield equal monomials in $\W$ (by Proposition
\ref{prop.DU-eqlett2}), but yield equal monomials in each of the quotients
$\W^{-}$, $\W_{-}$ and $\W_{-}^{-}$ (since their
difference in $\W$ is $U^{p-1}DUD^{p-1}-U^{p-1}D^{p-1}=U^{p-1}%
\underbrace{\left(  DU-1\right)  }_{=UD}D^{p-1}=U^{p-1}UDD^{p-1}=U^{p}D^{p}$,
which becomes $0$ in each of the quotients). What combinatorial condition is
responsible for this equality? We don't know; there are neither any balanced
commutations that can be applied to them to produce new words, nor any $U^{p}$
or $D^{p}$ factors that can be annihilated. Thus, we ask the following rather
open-ended question (actually three questions in disguise):

\begin{question}
Characterize equalities between monomials in $\W^{-}$, $\W_{-}$ and $\W_{-}^{-}$ combinatorially.
\end{question}

\subsection{Down-up algebras}

We now return to the case when $\kk$ is a field of characteristic $0$.

There are several deformations and other variations of the Weyl algebra
$\W$, and our question about equal monomials can be asked for each of
them. We  here discuss one of the most recent such variations: the
\emph{down-up algebra}, actually a family of algebras depending on three
scalar parameters $\alpha,\beta,\gamma$.

We fix three scalars $\alpha,\beta,\gamma\in\kk$. The \emph{down-up
algebra} $\mathcal{A}\left(  \alpha,\beta,\gamma\right)  $ is defined to be
the $\kk$-algebra with generators $D$ and $U$ and the two relations%
\begin{align*}
D^{2}U  & =\alpha DUD+\beta UD^{2}+\gamma D\ \ \ \ \ \ \ \ \ \ \text{and}\\
DU^{2}  & =\alpha UDU+\beta U^{2}D+\gamma U.
\end{align*}
Clearly, this algebra $\mathcal{A}\left(  \alpha,\beta,\gamma\right)  $ has an
algebra anti-automorphism $\omega$ sending $D$ and $U$ to $U$ and $D$. Also,
it is easy to check that the above two relations of $\mathcal{A}\left(
\alpha,\beta,\gamma\right)  $ are satisfied in the Weyl algebra $\W$
whenever $\alpha+\beta=\gamma-\beta=1$; therefore, the Weyl algebra
$\W$ is a quotient of $\mathcal{A}\left(  \alpha,\beta,\gamma\right)
$ in this case. Down-up algebras originate in \cite[Proposition 3.5]{BenRob98}
and have since found uses in noncommutative algebraic geometry and combinatorics.

We can define a map $\phi : \M \to \mathcal{A}\left(
\alpha,\beta,\gamma\right)  $ in the same way as we defined $\phi
: \M \to \W$, but using the down-up algebra
$\mathcal{A}\left(  \alpha,\beta,\gamma\right)  $ instead of $\W$.
(Thus, $\phi$ is a morphism of multiplicative monoids and sends $D$ and $U$ to
$D$ and $U$.) Surprisingly, we have:

\begin{theorem}
\label{thm.dua.equiv}
Assume that $\alpha+\beta=\gamma-\beta=1$. Then, the
equivalence of the six statements $\mathcal{S}_{1}$, $\mathcal{S}_{3}$,
$\mathcal{S}'_3$, $\mathcal{S}_{4}$, $\mathcal{S}_{5}$ and $\mathcal{S}_6$
in Theorem \ref{thm.DU-eqs2} still holds if we replace $\W$ by
$\mathcal{A} \left(  \alpha,\beta,\gamma\right)  $.
\end{theorem}

\begin{proof}
The statements $\mathcal{S}_{3}$, $\mathcal{S}'_3$, $\mathcal{S}_{4}$,
$\mathcal{S}_{5}$ and $\mathcal{S}_6$
are unchanged from Theorem \ref{thm.DU-eqs2}, so they are
still equivalent. It thus remains to prove that $\mathcal{S}_{1}%
\Longleftrightarrow\mathcal{S}_{5}$.

$\mathcal{S}_{5}\Longrightarrow\mathcal{S}_{1}$: This relies on a version of
Lemma \ref{lem.bal.weyl} for the algebra $\mathcal{A}\left(  \alpha
,\beta,\gamma\right)  $ instead of $\W$.
This version, in turn, relies on a version of Lemma \ref{lem.bal.weyl0} for the algebra $\mathcal{A}\left(  \alpha
,\beta,\gamma\right)  $ instead of $\W$.
But the latter has already been established in \cite{BenRob98}.
In fact, the $\kk$-algebra $\mathcal{A}\left(
\alpha,\beta,\gamma\right)  $ is graded (just like $\W$: the
generators $U$ and $D$ are homogeneous of degrees $1$ and $-1$). Its $0$-th
graded component is commutative, by \cite[Proposition 3.5]{BenRob98}. As in
the second proof of Lemma \ref{lem.bal.weyl0}, this entails the validity of Lemma
\ref{lem.bal.weyl0} for $\mathcal{A}\left(  \alpha
,\beta,\gamma\right)  $, and this in turn yields that Lemma \ref{lem.bal.weyl} holds for $\mathcal{A}\left(  \alpha
,\beta,\gamma\right)  $ as well.
Hence, $\mathcal{S}_{5}$ implies
$\mathcal{S}_{1}$.

$\mathcal{S}_{1}\Longrightarrow\mathcal{S}_{5}$: The condition $\alpha
+\beta=\gamma-\beta=1$ ensures that the Weyl algebra $\W$ is a
quotient of $\mathcal{A}\left(  \alpha,\beta,\gamma\right)  $ (with its
generators $D$ and $U$ being the projections of the generators $D$ and $U$).
Thus, $\phi\left(  u\right)  =\phi\left(  v\right)  $ in $\mathcal{A}\left(
\alpha,\beta,\gamma\right)  $ implies $\phi\left(  u\right)  =\phi\left(
v\right)  $ in $\W$. In other words, our new statement $\mathcal{S}%
_{1}$ implies the old statement $\mathcal{S}_{1}$ from Theorem
\ref{thm.DU-eqs2}. But that old statement implies $\mathcal{S}_{5}$, as we
already know.
Hence, our new $\mathcal{S}_1$ also implies $\mathcal{S}_5$.
\end{proof}

\begin{example}
One of the situations to which Theorem~\ref{thm.dua.equiv} applies is the case when $\alpha = 2$ and $\beta = -1$ and $\gamma = 0$.
In this case, the down-up algebra $\mathcal{A}\left(\alpha,\beta,\gamma\right) = \mathcal{A}\left(2,-1,0\right)$ is the \emph{homogenized Weyl algebra}, since its two defining relations
\begin{align*}
D^{2}U =2 DUD- UD^{2}\ \ \ \ \ \ \ \ \ \ \text{and}\ \ \ \ \ \ \ \ \ \ \ DU^{2} =2 UDU- U^{2}D
\end{align*}
can be rewritten (in terms of commutators $\left[x,y\right] := xy-yx$) as
\[
\left[D,\ DU-UD\right] = 0 \ \ \ \ \ \ \ \ \ \ \text{and}\ \ \ \ \ \ \ \ \ \ \ \left[U,\ DU-UD\right] = 0
\]
and thus mean that the commutator $DU - UD$ is a central element.
This algebra appears in \cite[\S 2.3]{BDSHP10} under the name of $\mathcal{U}\left(\mathcal{L}_{\mathcal{H}}\right)$, with the generators $a$, $a^\dagger$ and $e$ corresponding to our $D$, $U$ and $DU-UD$.
\end{example}

\begin{question}
To what extent can the condition $\alpha+\beta=\gamma-\beta=1$ be lifted in
Theorem \ref{thm.dua.equiv}?
\end{question}

This condition is unnecessary for $\mathcal{S}_{5}\Longrightarrow
\mathcal{S}_{1}$, but needed for $\mathcal{S}_{1}\Longrightarrow
\mathcal{S}_{5}$. For example, $\mathcal{S}_{1}\Longrightarrow\mathcal{S}_{5}$
would fail in the following five cases:
\begin{enumerate}
\item the case $\tup{\alpha, \beta, \gamma} = \tup{0, 1, 0}$ (here, we have $\phi\tup{DU^2} = \phi\tup{U^2D}$);
\item more generally, the case $\alpha = \gamma = 0$ and arbitrary $\beta$ (here we have $\phi\tup{DU^4D} = \phi\tup{U^2D^2U^2}$);
\item the case $\tup{\alpha, \beta} = \tup{0, -1}$ and arbitrary $\gamma$ (here we have $\phi\tup{DU^4} = \phi\tup{U^4D}$);
\item the case $\tup{\alpha, \beta} = \tup{-1, -1}$ and arbitrary $\gamma$ (here we have $\phi\tup{DU^3} = \phi\tup{U^3D}$);
\item the case $\tup{\alpha, \beta} = \tup{1, -1}$ and arbitrary $\gamma$ (here we have $\phi\tup{DU^6} = \phi\tup{U^6D}$).
\end{enumerate}
However, such failures seem to be the
exception, not the rule.
The last three are explained by \cite[Theorem 1.3 (f)]{Zhao99}, and correspond to the only roots of unity that are quadratic over $\QQ$.
There are more exceptions with irrational $\alpha, \beta, \gamma$, for instance $\tup{\alpha, \beta} = \tup{0, i}$ with $i = \sqrt{-1}$, which satisfies $\phi\tup{DU^8} = \phi\tup{U^8D}$.

With some additional work, we can adapt the above proof of Theorem \ref{thm.dua.equiv} to replace the condition ``$\alpha+\beta=\gamma-\beta=1$'' by ``$\alpha + \beta = 1$ and $\left(\gamma = 0\right) \Longleftrightarrow \left(\beta = -1\right)$''.
Indeed, under this condition, we can find a nonzero scalar $\zeta \in \kk$ such that $\gamma = \tup{1+\beta}\zeta$.
Then, there is a $\kk$-algebra morphism from $\mathcal{A}\tup{\alpha, \beta, \gamma}$ to $\W$ that sends $D$ and $U$ to $D$ and $\zeta U$.
This morphism sends equal monomials in $\mathcal{A}\tup{\alpha, \beta, \gamma}$ to proportional monomials in $\W$;
but Lemma~\ref{lem.Wn.scalar1} \textbf{(b)} says that proportional monomials in $\W$ must actually be equal, and so we can argue as in the proof of Theorem \ref{thm.dua.equiv}.
However, the condition $\alpha + \beta = 1$ cannot be lifted in this way.
Instead, we suspect that the linear-recurrence highest weight modules of \cite[Proposition 2.2]{BenRob98} should be used in the general case in lieu of $\kk\ive{x}$ (certainly, this would explain the above exceptions as coming from the periodic linearly recurrent sequences).

\bigskip

\textbf{Acknowledgments.}
We thank Richard P. Stanley for posing a series of problems that became many of the results of this paper.
We further thank J\"orgen Backelin and Joel Brewster Lewis for helpful conversations.


\begin{thebibliography}{99999999}

\bibitem[BCHR11]{BCHR11}Fred Butler, Mahir Can, Jim Haglund, Jeffrey B.
Remmel, \textit{Rook Theory Notes}, 2011.\newline\url{https://mathweb.ucsd.edu/~remmel/files/Book.pdf}

\bibitem[BDSHP10]{BDSHP10}
P. Blasiak, G. H. E. Duchamp, A. I. Solomon, A. Horzela, K. A. Penson, \textit{Combinatorial Algebra for second-quantized Quantum Theory}, Adv. Theor. Math. Phys. \textbf{14} (2011), pp. 1--35. See \href{https://arxiv.org/abs/1001.4964v1}{arXiv:1001.4964v1} for a preprint.

\bibitem[BenRob98]{BenRob98}
\href{https://doi.org/10.1006/jabr.1998.7511}{Georgia Benkart, Tom Roby,
\textit{Down-Up Algebras},
Journal of Algebra \textbf{209}, Issue 1, 1 November 1998, pp. 305--344.}
Addendum in: \href{https://doi.org/10.1006/jabr.1998.7854}{Journal of Algebra \textbf{213}, Issue 1, 1 March 1999, p. 378}.

\bibitem[BHDPS08]{GoF09}
\href{https://doi.org/10.1088/1751-8113/41/41/415204}{
P. Blasiak, A. Horzela, G. H. E. Duchamp, K. A. Penson, A. I. Solomon, 
\textit{Heisenberg-Weyl algebra revisited: Combinatorics of words and paths}, 
J. Phys. A: Math. Theor. \textbf{41} (2008), 415204.}

\bibitem[Blease77]{Ble77}
\href{https://doi.org/10.1088/0022-3719/10/18/012}
{J. Blease, \textit{Pair-connectedness for directed bond percolation on some two-dimensional lattices by series methods}, J. Phys. C: Solid State Phys. 10, 3461, 1977.}

\bibitem[BroHam57]{BroHam57}
\href{https://doi.org/10.1017/S0305004100032680}
{S. R. Broadbent, J. M. Hammersley, \textit{Percolation processes: I. Crystals and mazes}, Mathematical Proceedings of the Cambridge Philosophical Society. 53(3), 629--641, 1957.}

\bibitem[Conrad24]{KConrad-tp2}
Keith Conrad,
\textit{Tensor products II},
25 April 2024.
\url{https://kconrad.math.uconn.edu/blurbs/linmultialg/tensorprod2.pdf} .

\bibitem[CoGrRa23]{CoGrRa23}
\href{https://arxiv.org/abs/2312.02508v1}{Giuseppe Cotardo, Anina Gruica, Alberto Ravagnani, \textit{The Diagonals of a Ferrers Diagram}, arXiv:2312.02508v1.}

\bibitem[Dixmie68]{Dixmie68}
\href{https://doi.org/10.24033/bsmf.1667}{Jacques Dixmier, \textit{Sur les alg\`ebres de Weyl}, Bulletin de la Soci\'et\'e Math\'ematique de France \textbf{96} (1968), pp. 209--242}.

\bibitem[EGJT96]{EGJT96}
\href{https://iopscience.iop.org/article/10.1088/0305-4470/29/8/010/pdf}
{J. W. Essam, A. J. Guttmann, I. Jensen, D. TanlaKishani, \textit{Directed percolation near a wall}, Journal of Physics A: Mathematical and General \textbf{29} (1996), pp. 1619--1628}.

\bibitem[EGHetc11]{Etingof}
\href{http://www-math.mit.edu/~etingof/reprbook.pdf}{Pavel Etingof, Oleg
Golberg, Sebastian Hensel, Tiankai Liu, Alex Schwendner, Dmitry Vaintrob,
Elena Yudovina, \textit{Introduction to Representation Theory}, with
historical interludes by Slava Gerovitch, Student Mathematical Library
\textbf{59}, AMS 2011, updated version 2018}.

\bibitem[Flajol80]{Flajol80}
\href{https://doi.org/10.1016/0012-365X(80)90050-3}{P. Flajolet,
\textit{Combinatorial aspects of continued fractions}, Discrete Mathematics \textbf{32} (1980), Issue 2, pp. 125--161.}

\bibitem[FoaSch70]{FoaSch70}
\href{https://irma.math.unistra.fr/~foata/paper/pub13.pdf}
{D. Foata, M. P.
Sch\"{u}tzenberger, \textit{On the rook polynomials of Ferrers relations},
Combinatorial theory and its applications, II (Proc. Colloq.,
Balatonf\"{u}red, 1969), pp. 413--436, North-Holland, Amsterdam, 1970.}

\bibitem[Gaddis23]{Gaddis23}
\href{https://arxiv.org/abs/2305.01609}{Jason Gaddis,
\textit{The Weyl algebra and its friends: a survey},
arXiv:2305.01609.}

\bibitem[GarRem86]{GarRem86}
\href{https://doi.org/10.1016/0097-3165(86)90083-X}{A.M. Garsia, J.B. Remmel, \textit{Q-counting rook configurations and a formula of Frobenius}, Journal of Combinatorial Theory, Series A \textbf{41} (1986), Issue 2, pp. 246--275.}

\bibitem[GJW75]{GJW75}
\href{https://www.ams.org/journals/proc/1975-052-01/S0002-9939-1975-0429578-4/S0002-9939-1975-0429578-4.pdf}{Jay R. Goldman, J. T. Joichi and Dennis E. White,
\textit{Rook Theory. I. Rook Equivalence of Ferrers Boards},
Proceedings of the American Mathematical Society \textbf{52} (1975), pp. 485--492.} 

\bibitem[Grimme99]{Grimme99}
\href{https://link.springer.com/book/10.1007/978-3-662-03981-6}
{Geoffrey Grimmett, \textit{Percolation}, Second edition. Grundlehren der mathematischen Wissenschaften [Fundamental Principles of Mathematical Sciences] \textbf{321}, Springer-Verlag 1999.}

\bibitem[Haglun98]{Haglun98}
\href{https://doi.org/10.1006/aama.1998.0582}{James Haglund, \textit{q-Rook Polynomials and Matrices over Finite Fields}, Advances in Applied Mathematics \textbf{20}, Issue 4, May 1998, pp. 450--487.}

\bibitem[Lorenz18]{Lorenz18}\href{https://www.ams.org/books/gsm/193/}
{Martin Lorenz,
\textit{A Tour of Representation Theory},
Graduate Studies in Mathematics \textbf{193},
AMS 2018.}

\bibitem[ManSch16]{ManSch16}
\href{https://doi.org/10.1201/b18869}{Toufik Mansour, Matthias Schork, \textit{Commutation Relations, Normal Ordering, and
Stirling Numbers}, CRC Press 2016.}

\bibitem[Milici17]{Milici17}
\href{https://www.math.utah.edu/~milicic/Eprints/dmodules.pdf}{Dragan Milicic, \textit{Lectures on Algebraic Theory of D-Modules}, 2017-03-29,}
\url{https://www.math.utah.edu/~milicic/Eprints/dmodules.pdf}

\bibitem[Navon73]{Navon73}
\href{https://doi.org/10.1007/BF02828687}{A. M. Navon,
\textit{Combinatorics and Fermion Algebra},
Il Nuovo Cimento \textbf{16B} (1973), no. 2, pp. 324--330.}

\bibitem[OEIS]{OEIS}
\href{https://oeis.org}{OEIS Foundation Inc. (2024), The On-Line Encyclopedia of Integer Sequences, Published electronically at \url{https://oeis.org}.} 

\bibitem[Stanle88]{Stan88DiffPos}
\href{https://math.mit.edu/~rstan/pubs/pubfiles/77.pdf}{Richard P. Stanley, \emph{Differential Posets}, J. AMS \textbf{1}(4) (1988), 919--961.}

\bibitem[Stanle06]{Stanley-IHA}
\href{https://math.mit.edu/~rstan/arrangements/arr.html}{Richard P. Stanley, \textit{An Introduction to Hyperplane
Arrangements}, IAS lecture notes, 2006.}

\bibitem[StaAha94]{StaAha94}
\href{https://doi.org/10.1201/9781315274386}
{D. Stauffer, A. Aharony, \textit{Introduction to Percolation Theory: Revised Second Edition}, CRC Press 1994.}

\bibitem[vanOys13]{vanOys13}
\href{https://www.torrossa.com/digital/toc/2013/3939021_TOC.pdf}{Freddy Van Oystaeyen,
\textit{Algebraic structures as seen on the Weyl Algebra},
Editorial Universidad de Almer\'ia, 2013.}
\url{https://www.torrossa.com/digital/toc/2013/3939021_TOC.pdf}

\bibitem[Varvak04]{Varvak}
\href{https://doi.org/10.1016/j.jcta.2005.07.012}{Anna Varvak, \textit{Rook numbers and the normal ordering problem}, Journal of Combinatorial Theory, Series A \textbf{112} (2005), Issue 2, pp. 292--307.}

\bibitem[Zhao99]{Zhao99}
\href{https://doi.org/10.1006/jabr.1998.7695}{Kaiming Zhao, \textit{Centers of Down–Up Algebras},
Journal of Algebra \textbf{214} (1999), Issue 1, pp. 103--121.}
\end{thebibliography}
\end{document}